\documentclass[11pt]{amsart}
\usepackage{amssymb,amsgen,amsbsy,amsopn,amsfonts,graphicx}
\usepackage{amsmath}
\usepackage[dvips]{color}
\usepackage{multicol}
\usepackage{mathrsfs}
\usepackage{nccmath}
\usepackage{enumerate}
\usepackage{url}
\usepackage{amsthm}
\newtheorem{lemma}{Lemma}[section]
\newtheorem{theorem}[lemma]{Theorem}
\newtheorem{prop}[lemma]{Proposition}

\newtheorem{rem}[lemma]{Remark}

\theoremstyle{definition}

\numberwithin{equation}{section}
\numberwithin{lemma}{section}
\def\R{{\mathbb R}}
\newcommand{\ep}{\epsilon}
\newcommand{\wt}{\widetilde}
\newcommand{\eps}{\epsilon_{\star}}
\newcommand{\af}{A_f}

\newcommand{\ovr}{\overline{r}}
\newcommand{\wtt}{\mathscr{\widetilde{T}}_{cu}}

\newcommand{\p}{\partial}
\newcommand{\ST}{\mathscr{T}}
\newcommand{\CM}{\mathcal{M}}
\newcommand{\CA}{\mathcal{A}}
\newcommand{\wte}{\widetilde{\mathscr{T}}_*}

\begin{document}

\title[Normally Elliptic Singular Perturbations]{Normally Elliptic Singular Perturbations and persistence of homoclinic orbits}


%

\author[Lu and Zeng]{Nan Lu and Chongchun Zeng$^1$}
\address{School of Mathematics \\
Georgia Institute of Technology\\ Atlanta, GA 30332}
\email{nlu@math.gatech.edu \\zengch@math.gatech.edu}
\thanks{$^1$ The second author is funded in part by NSF-DMS 0801319.}
\maketitle

\begin{abstract} We consider a dynamical system, possibly infinite dimensional or non-autonomous, with fast and slow time scales which is oscillatory with high frequencies in the fast directions. We first derive and justify the limit system of the slow variables. Assuming a steady state persists, we construct the stable, unstable, center-stable, center-unstable, and center manifolds of the steady state of a size of order $O(1)$ and give their leading order approximations. Finally, using these tools, we study the persistence of homoclinic solutions in this type of normally elliptic singular perturbation problems.
\end{abstract}

\section{Introduction} \label{S:intro}

A singular perturbation system usually involves different temporal or spatial scales. Here we focus on multiple time scales in which case the system takes the abstract form of
\begin{equation} \label{E:SP0}
\dot x = F(x, y, \ep)\qquad \ep \dot y = G(x, y, \ep).
\end{equation}
The fast motions in the $y$ direction are often some noise or transient behaviors and the slow motions in the $x$ direction are more of the focus of the problem. In the singular limit as $\ep \to 0$, we obtain $g(x, y, 0)=0$. Suppose $y = \phi(x)$ (without the loss of generality, assuming $\phi\equiv 0$) solves this equation, the limit motion of $x$ is given by
\begin{equation} \label{E:SPL0}
\dot x = F(x, 0, , 0).
\end{equation}
Let $\tilde y = \frac y\ep$, the $y$ equation in \eqref{E:SP0} takes the form
\begin{equation} \label{E:NE0}
\dot {\tilde y} = \frac {G_y(x, 0, 0)}\ep \tilde y + g(x, \tilde y, \ep).
\end{equation}

The singular perturbation system \eqref{E:SP0} is called normally hyperbolic if,
for each $x$, the linear flow $e^{tG_y(x, 0, 0)}$ on the $y$ space is hyperbolic,
i.e. it is exponentially contracting on one closed subspace and exponentially expanding
in an complementary subspace. In this case, the standard normally hyperbolic invariant
manifold theory \cite{Fe71, HPS77, He81, BLZ98, BLZ99} applies to yield a persistent normally hyperbolic invariant slow manifold $M_\ep$ given by a graph $y = \ep \phi(x, \ep)$. In the fast (and hyperbolic in natural) motions outside $M_\ep$, solutions usually approach a neighborhood of $M_\ep$ exponentially along its stable direction. After some time moving along the slow manifold, the solutions leave the neighborhood exponentially along the unstable directions. These motions of multiple scales can be connected by tools such as invariant foliations \cite{Fe74, Fe77, HPS77, CLL91, BLZ00} and this geometric approach has led to a huge success in the study of the dynamics of singular perturbation system \eqref{E:SP0}. See for example \cite{Fe79, Jo95, JK95, BLZ08}.

In the normally elliptic case, i.e. $e^{tG_y(x, 0, 0)}$ is oscillatory instead of hyperbolic, on the one hand, the persistence of the slow manifold is not always guaranteed \cite{GL02, GL03}. On the other hand, solutions starting near $\{y=0\}$ should stay there at least for some $O(1)$ time period due to the lack of strong exponential instability in the $y$ direction. One typical situation of this type is when $G_y(x, 0, 0)$ is anti-self-adjoint.

In this paper, with applications to both ODEs and PDEs in mind, we study these normally elliptic singular perturbation problems in an infinite dimensional dynamical system and possibly non-autonomous framework. Assuming $G_y(x, 0, 0) = J$, a constant anti-self-adjoint operator, we first justify the limit equation \eqref{E:SPL0} of the slow variable $x$ through a careful averaging.

A more important question is how much of the dynamical structure of the limit slow system \eqref{E:SPL0} remains in the singularly perturbed system \eqref{E:SP0}. Elliptic type motions in the slow directions such as periodic or quasi-periodic solutions may be resonant with the oscillatory fast motions in the $y$ direction. Some results have been obtained on the persistence of periodic orbits for nonresonant $\ep<<1$ \cite{GL02, GL03, L00, Ma01, SZ02}. Here instead we focus on the basic hyperbolic structure -- the local invariant manifolds near a steady state. Suppose $(0, 0)$ persists as a steady state of \eqref{E:SP0} for $\ep <<1$. Assume the linearization of the limit slow system \eqref{E:SPL0} has invariant stable, unstable, and center subspaces $X^{u,s,c}$. For the expanded system \eqref{E:SP0}, the normal directions -- the $Y$ space -- with the oscillatory linearized flow $e^{tJ}$ should obviously be considered as additional center directions. The first observation is, even though system
\eqref{E:SP0} is singular, the existence of local invariant manifolds of $(0,0)$ is guaranteed by the standard theory (see, for example, \cite{BJ89, Ha61, CL88}) after a rescaling of the time by a factor of $\ep$. However, since the exponential growth/decay rates in the unstable/stable direction are $O(\ep)$ after the rescaling, this approach would only yield local invariant manifolds of the size of $O(\ep)$, which is far from being useful in most applications, such as studying the persistence of homoclinic orbits. 

Our main result in the manuscript is the existence and smoothness and the leading order approximation of invariant manifold of the steady state of the size of $O(1)$ based on a combination of the averaging and Lyapunov-Perron integral equation methods.

As an application which is also an fundamental problem itself, suppose there exists a homoclinic orbits in the limit slow system \eqref{E:SPL0} and we study its persistence in the singular perturbation system \eqref{E:SP0} which can be either weakly dissipative or conservative. In the former, we derive the Melnikov function, which include an additional term coming from the fast directions, whose simple zero indicates a persistence homoclinic orbit to $(0, 0)$. In the
latter, when the system is analytic in reasonably low dimensions, along with some other structures such as the Hamiltonian setting or the reversibility, it has been shown that the stable and unstable manifold miss each other by an error like $O(e^{-\frac C\ep})$ \cite{Sun98, G00, L00, To00}. Without these assumptions, we prove that there always exist orbits homoclinic to the center-manifold, forming a tube homoclinic to the center manifold. While we follow the well-developed geometric ideas in the finite or infinite dimensional regular perturbation problems \cite{GH83, HM81, LMSW96, SZ03}, the proof heavily depends on the invariant manifolds we studied.

Before finishing the introduction, we would like to give two simple examples which partially motivated us to study this subject, while it is also easy to come up with examples in infinite dimensions. One is an elastic pendulum with fast and slow frequencies itself and the other one is a bifurcation problem which does not have any singular parameter in the appearance.

A pendulum of the unit length with a fixed end is described by the Duffing equation. In a more careful model, the pendulum usually considered as rigid may have some small elasticity -- meaning large elastic constant $\frac 1{\ep^2}$ -- allowing the pendulum to be stretched or contracted slightly in the radial direction. Let $x$ be the angular and $1+y$ be the radial coordinates, respectively, and the system takes of the form of a normally elliptic singular perturbation problem
\begin{equation}\label{eq1.1}
\left\{\begin{aligned}(1+y)\ddot{x}+2\dot{x}\dot{y}& +g\sin{x}   +2\ep\gamma
(1+y)\dot{x}- \frac \ep{1+y} F_1(x,y,\ep,t)=0\\
\ddot{y}-(1+y)\dot{x}^2&+\frac{1}{\ep^2}y -g\cos{x}+2\ep
\gamma \dot{y}-\ep F_2(x,y,\ep,t)=0
\end{aligned} \right.
\end{equation}
Where we also included the small damping and forcing. Formally, as $\ep\to 0$, i. e. the pendulum converges to be rigid, the corresponding singular limit \eqref{E:SPL0} for the above system becomes
\begin{equation}\label{eq1.2}
y\equiv 0, \qquad  \ddot{x}+g\sin{x}=0.
\end{equation}
When there is no damping and the force is conservative, the problem is in the Lagrangian setting and the limit equation is justified in \cite{RU57, Ar78, Ta80}. In the dynamics, the state $(\pi, 0)$ is a hyperbolic steady state of \eqref{eq1.2} with a homoclinic orbit which often leads to chaos even under small regular perturbation \cite{GH83}. One may easily change the variables in the singular equation of $y$ and make it anti-self-adjoint. Our general results apply to \eqref{eq1.1} and give the criterion when the homolcinics persist under either dissipative or conservative perturbation. This example will be revisited in Section \ref{S:homo}.

The singular perturbations theory also applies to problems which may not be explicitly in the form of \eqref{E:SP0}. Consider an autonomous 4-dim ODE system with a parameter $\ep$ which has the origin $O$ as a fixed point for all $\ep<<1$. Assume, at $\ep=0$, the linearized systems has simple eigenvalues $\pm i$ and a double eigenvalue $0$. While the unfolding of the focal point has been studied thoroughly (see, for example \cite{CLW94}), we note that the oscillatory motions are essentially at a much faster scale in the directions of the pair of elliptic eigenvalues. Under these assumptions, some simple normal forms transformations and near identity time rescaling, the generic form of the system looks
\[
\dot x = \begin{pmatrix} a_{11}(\ep) & 1+ a_{12}(\ep) \\ a_{21}(\ep) & a_{22} (\ep) \end{pmatrix} x + O(|x|^2+ |y|^2) \qquad \dot y= \begin{pmatrix} b (\ep) & 1 \\ -1 & b (\ep) \end{pmatrix} y + O(|x|^2+ |y|^2)
\]
where $x, \, y \in \R^2$ and $a_{lm}(0)=b(0)=0$. Rescale the system again by
\[
x_1 = \ep \tilde x_1, \; x_2 = \ep^{\frac 32} \tilde x_2, \; y=\ep \tilde y, \; t=\ep^{-\frac 12} \tau
\]
we obtain a singular perturbed systems in the form of \eqref{E:SP0} of the normally elliptic type with the singular parameter $\mu = \ep^{\frac 12}$. If $\frac {d a_{21}}{d \ep} (0)> 0$, the origin becomes hyperbolic in the $x$ directions and we obtain the local center manifolds of order $O(1)$ size in the rescaled variables. If $\frac { db}{d \ep}(0) \ne0$ in addition, we are in the right position to study the Hopf bifurcation from the eigenvalues $\pm i$ in this rather degenerate case. (See \cite{Fe83} for an approach essentially different from the Hopf bifurcation.) A more detailed study of this type of bifurcation problems will be given in a forthcoming paper.

The rest of the manuscript is organized as the following. In Section~\ref{S:intro} we present the general framework and outline the main results on invariant manifolds and foliations. The justification of the limit slow equations and its linearization are obtained in Section~\ref{S:FT}. In Section~\ref{S:InMa} and~\ref{S:InFo} we study invariant manifolds and foliations and focus on their leading order approximations. Finally the homoclinic orbits are considered in Section~\ref{S:homo}. In the Appendix, we outline a process to block-diagonalize the linearized system of \eqref{E:SP0} at a steady state.

\section{Framework and main results on invariant manifolds and foliations}

We formulate the problem as a non-autonomous infinite dimensional dynamical systems with a singular parameter. This framework allows one to apply the general results to ODEs as well as PDEs or functional differential equations.

In a Banach space $Z$ and $z \in Z$, we usually denote a ball by $B_r(z, Z)$. Let $Z_1,Z_2$ be Banach spaces and $k\ge 1$ be an integer. We adopt the notations
\begin{align*}
&L_k (Z_1, Z_2) \triangleq L(\otimes^k Z_1, Z_2) =  \big\{\text{bounded } k-\text{linear operators } Z_1 \to Z_2\big\}\\
&|\phi| \triangleq \sup_{|z_1| \le 1, \ldots |z_k|\le 1} |\phi(z_1, \ldots, z_k)|, \quad \text{for } \phi \in L_k(Z_1, Z_2)\\
&C^k(Z_1,Z_2)=\big\{h\big|h:Z_1\rightarrow Z_2, \hspace{1.5mm}k\mbox{-times countinuously differentiable with} \\
& \qquad \qquad \qquad \qquad \text{finite } C^k \text{ norm}\big\}.
\end{align*}
Note $L_1(Z_1,Z_2)=L(Z_1,Z_2)$ is simply the space of bounded linear operators.

Throughout the manuscript, we use $D$ or $D^k$ to denote differentiations with respect to
variables in the phase space and we will use $\partial$ for derivatives with respect to time $t$ or other parameters.



Let $X$ be a Banach space and $Y$ a Hilbert space and we consider the system
\begin{equation}\label{eq3.1}
\left\{\begin{aligned}
&\dot{x}=Ax+f(x,y,t,\ep)\\
&\dot{y}=\frac{J}{\ep}y+g(x,y,t,\ep)
\end{aligned} \right.
\end{equation}

The following assumptions may look complicated which is only due to our intention to make the result applicable to PDEs where unbounded operators and different function spaces are involved. For ODE systems, these assumption would simply be
\begin{itemize}
\item $J$ is an anti-symmetric matrix and $(f, g)$ are smooth functions.
\end{itemize}
In general, we assume for some constants $C_0$,
\begin{enumerate}
\item[(A1)] $A: X_1 \triangleq D(A) \rightarrow X$, where $X_1 \subset X$ is endowed with the graph norm $|\cdot|_{X_1}$, generates a $C_0$-semigroup $e^{tA}$ on $X$ such that $|e^{tA}|\leq Me^{\omega t}$, $t\ge 0$, for some $M>0$ and $\omega\in \R$.
\item[(A2)] $J$ is an anti-self-adjoint operator on $Y$ with domain
$D(J)=Y_1$, endowed with the graph norm $|\cdot|_{Y_1}$, which generates a unitary group $e^{tJ}$. We further assume $|J^{-1}|_{L(Y,Y_1)}\leq C_0$.
\item[(A3)] For $k\geq1$,
\begin{eqnarray*}
(D^if,D^ig)&\in& C^0(X_1\times Y_1\times\mathbb{R}^2,L_i(X_1\times
Y_1,X_1\times Y_1)),\ 0\leq i\leq k, \\
(D^if,D^ig)&\in& C^0(X_1\times Y_1\times\mathbb{R}^2,\\
&&\hspace{1.5cm}L((X\times Y)\otimes^{i-1}(X_1\times Y_1),X\times
Y)), \ 1\leq i\leq k,
\end{eqnarray*}
whose norms are all bounded by $C_0$.
\item[(A4)] $|\cdot|_X\in C^k(X\backslash \{0\},\mathbb{R}^{+})$, where
$\mathbb{R}^{+}$ denotes the set of positive numbers.
\item[(A5)] $\partial_{\ep}f\in C^0(X_1\times
Y_1\times\mathbb{R}^2,X_1)$, $D_x\partial_{\ep}f\in C^0(X_1\times
Y_1\times\mathbb{R}^2,L(X_1,X_1))$, $D_x\partial_tf\in C^0(X_1\times
Y_1\times\mathbb{R}^2,L(Y_1,X))$, $\partial_tg\in C^0(X_1\times
Y_1\times\mathbb{R}^2,Y)$, $D\partial_tg\in C^0(X_1\times
Y_1\times\mathbb{R}^2,L(X_1\times Y_1,X\times Y))$ which are all bounded by $C_0$.
\end{enumerate}
Here the global boundedness are not important as we can always multiply them by a cut-off function. When unbounded, the linear operators often appear in the form of $\Delta$, $i\Delta$, or $\begin{pmatrix} 0 &1 \\ \Delta &0\end{pmatrix}$, etc. as coming from the linearization of PDEs \cite{Pa83}. The nonlinearities usually satisfy the above assumption on function spaces which are algebras or slightly better.
\begin{rem}\label{rem3.1}
In fact we can replace $J$ by $J(\ep)$ for each small $\ep$, then
all results in this manuscript still hold except Proposition \ref{P:appro} and \ref{thm4.20}.
\end{rem}

Throughout the manuscript, $C$ denotes a generic constant, possibly with subscripts, which may
have different values as in different lines, and it only
depends on the quantities involved in (A1)-(A5). Let $C'$ be another
generic constant, possibly with subscripts, and the dependence will
be specified in the context.

When $|y|<<1$, formally from \eqref{E:SPL0}, we expect the $x$ equation can be approximated by the singular limit
\begin{equation}\label{eq3.2}
\dot{x}_0=Ax_0+f(x_0,0,t,0).
\end{equation}
However, in the normally elliptic singular perturbation problems, there is usually not a persistent slow manifold and we will first prove the convergence to \eqref{eq3.2}.\\

\noindent {\bf Almost invariant slow manifolds.} In order to justify this limit, one need to estimate the $y$ equation in \eqref{eq3.1}. One of the key issues is to handle the $O(1)$ driving force $g(x, 0, t, \ep)$ which occurs even at $y=0$. It is very natural to first carry out a transformation
\begin{equation} \label{E:y_1}
y_1 = y + \ep J^{-1} g(x, 0, t, \ep)
\end{equation}
which yields
\begin{equation} \label{E:doty_1}
\dot y_1 = \frac  J\ep y_1 + g_1(x, y, t, \ep) \qquad g_1 (x, y, t,\ep) =O(\ep) + h(x, y_1, t, \ep)y_1.
\end{equation}
At $y=0$, the driving force in this equation $g_1(x, 0, t, \ep)=O(\ep)$. One may repeat this procedure and obtain
\[
\dot y_k = \frac  J\ep y_k + O(\ep^k) + O(|y_k|).
\]
Therefore this sequence of transformations yields an almost invariant slow manifold, close to $\{ y_k=0\}$, with an error (to the invariance in the equations) of $O(\ep^k)$. While this process increases the accuracy at the cost of the smoothness of the equation, it would not give an invariant slow manifold and also, in our general setting, the unbounded operators $A$ and $J$ without other assumptions could bring other complications. We will work directly
with \eqref{eq3.1} in most part of the manuscript. However in this manuscript, with these transformations in mind,
we actually often prove estimates with upper bounds in terms of $g(x, 0, t, \ep)$, so that they would yield much
finer estimates when combined with a sequence of transformations as in the above. For example, see Lemma \ref{L:FT},
Lemma \ref{L:FTL}, Proposition \ref{P:appro}, Remark \ref{R:Melnikov} to see statements of this type.

\begin{theorem}\label{thm3.1}
Assume (A1) -- (A3) and (A5). For any $t_0\in\mathbb{R}$, let $T>0$ and $(x(t),y(t))$, $x_0(t)$ be
solutions of \eqref{eq3.1} and \eqref{eq3.2} on $[t_0,t_0+T]$. Suppose
\[
|x(t_0)-x_0(t_0)|_{X_1}+|y(t_0)|_{Y_1}\leq C_1\ep,
\]
there exists a constant $C'$ which depends on $M,\omega,T,C_0,C_1$, and $|x_0(t_0)|_{X_1}$, such that
then for any $t\in[t_0,t_0+T]$,
\begin{equation}
|x(t)-x_0(t)|_{X_1}+|y(t)|_{Y_1}\leq C'\ep.
\end{equation}
\end{theorem}

A more careful estimate of approximations can be found in Lemma \ref{L:FT}. In addition to the convergence
of solutions on finite time interval,
we also need the convergence of solutions of the linearized
equations. Linearize \eqref{eq3.1} and we obtain
\begin{equation}\label{eq3.3}
\left\{\begin{aligned} &\dot{\delta x}=A\delta
x+D_xf(x,y,t,\ep)\delta
x+D_yf(x,y,t,\ep)\delta y\\
&\dot{\delta y}=\frac{J}{\ep}\delta y+D_xg(x,y,t,\ep)\delta
x+D_yg(x,y,t,\ep)\delta y, \end{aligned} \right.
\end{equation}
Let $\Phi(t,t_0,x,y,\ep)$ be the solution map of \eqref{eq3.1}. From the above equations assumptions (A1) --
(A3), and the Gronwall inequality, it is clear to see that $D\Phi$ is bounded uniformly in $\ep$. Higher order
derivatives of $\Phi$ in $x,y$ can be estimated in a similar way. In the leading order approximation of
\eqref{eq3.3}, we combine the linearized \eqref{eq3.2} and a linearized $y$ equation
\begin{equation}\label{eq3.4}
\left\{\begin{aligned} &\dot{\delta x}_0=A\delta
x_0+D_xf(x_0,0,t,0)\delta x_0 \\
&\dot{\delta y}_0=\frac{J}{\ep}\delta y_0+D_yg(x_0,0,t,0)\delta y_0.
\end{aligned} \right.
\end{equation}

\begin{theorem}\label{thm3.2}
Assume (A1) -- (A3) and (A5) for $k\ge2$. Let $(\delta x(t),\delta y(t))$ and $(\delta
x_0(t),\delta y_0(t))$ be solutions of \eqref{eq3.3} and
\eqref{eq3.4}, respectively. Suppose
\begin{eqnarray}\label{eq3.7}
&&|x(t_0)-x_0(t_0)|_{X_1}+|y(t_0)|_{Y_1}\leq C_1\ep,\\\label{eq3.8}
&& \ep(|\delta x_0(t_0)|_{X_1}+|\delta y_0(t_0)|_{Y_1}) + |\delta x(t_0)-\delta x_0(t_0)|_{X}+|\delta y(t_0)-\delta y_0(t_0)|_{Y_1}\leq C_1\ep.
\end{eqnarray}
Then there exists a
constant $C'$ which depends on $M,\omega,T,C_0,C_1,|x(t_0)|_{X_1}$,
such that
$$|\delta x(t)-\delta x_0(t)|_{X}+|\delta y(t)-\delta y_0(t)|_{Y_1}\leq
C'\ep$$ for all $t\in[t_0,t_0+T]$.
\end{theorem}

We can not obtain the estimates on $|\delta x(t)-\delta x_0(t)|_{X_1}$ even if we assume $|\delta x(t_0)-\delta x_0(t_0)|_{X_1}\leq C_1\ep$ unless both $|\delta y_0(t_0)|_{Y_1} \le C_1 \ep$. See Lemma \ref{L:FTL}, Remark
\ref{rem3.4}, and Remark \ref{rem3.6}. These theorems will be proved in Section \ref{S:FT}.

To study the local invariant manifolds, suppose $(0, 0)$ is always a steady state and the limit systems is autonomous, i.e.
\[
\partial_t f(x,y,t,0)=\partial_t g(x,y,t,0)=0 \qquad f(0,0,t,\ep)=g(0,0,t,\ep)=0.
\]
We assume linearized \eqref{eq3.2} at $0$ has the exponential trichotomy, i.e. there exist closed subspaces $X^{u,s,c}$ such that there exist constants $a_1 < \min\{a_2, 0\}$ and $a_2' > \max\{0, a_1'\}$ and for $t\ge 0$,
\[\begin{split}
&|e^{t(A + f_x(0))}|_{X^s} \le C_1 e^{a_1t} \quad |e^{-t(A + f_x(0))}|_{X^u} \le C_1 e^{-a_2't} \\
&|e^{t(A + f_x(0))}|_{X^c} \le C_1 e^{a_1't} \quad |e^{-t(A + f_x(0))}|_{X^c} \le C_1 e^{-a_2t}.
\end{split}\]
Moreover, we assume the linearized flow $e^{t(\frac J\ep + g_y(0))}$ satisfies the same assumption as $e^{-t(A + f_x(0))}|_{X^c}$ and thus the expanded center space of \eqref{eq3.1} should be $X^c \oplus Y$. Along with a few other technical assumptions, rough our main results on invariant manifolds and foliations in the phase space $X_1 \times Y_1$ is

\begin{theorem} \label{T:main}
For $\ep<<1$, in the space $X_1 \times Y_1$,
\begin{enumerate}
\item There exists smooth invariant stable, unstable, center-stable, center-unstable, and center integral
    manifolds of $(0, 0)$ which can be written as graphs of smooth mappings from a $\delta$-neighborhood of the
    corresponding subspaces to the complements whose norms and $\delta$ are independent of $\ep$. Moveover their
    derivatives in $t_0$, the time parameter of integral manifolds, is $O(\ep)$ when evaluated in the norm $|\cdot|_X + |\cdot|_Y$.
\item The center-stable and center-unstable manifolds are foliated into the disjoint union of smooth families of
    smooth stable and unstable fibers which also written as graphs of mappings whose norms are bounded independent of $\ep$.
\item The stable and unstable manifolds are $O(\ep)$ close to those of \eqref{eq3.2}.
\item The center-stable, center-unstable, and the center manifolds at $\{y=0\}$ are $O(\ep)$ close to those of \eqref{eq3.2} and their tangent spaces there are $O(\ep)$ close to the direct sum of the unperturbed ones and
    $Y$, respectively.
\end{enumerate}
\end{theorem}

Here by the term an integral manifold, we mean a family of manifold $M(t)$ parameterized by $t$ so that
the solution map of \eqref{eq3.1} starting at initial time $t_0$ and ending at $t_1$ maps $M(t_0)$ into
$M(t_1)$. They are independent of $t$ if the system is autonomous. The precise statement of these results
of the invariant manifolds are given in Section \ref{S:InMa} and \ref{S:InFo}.

\section{The singular limit system on finite time intervals} \label{S:FT}

The basic idea to handle the singular terms in the proofs of Theorem \ref{thm3.1} and \ref{thm3.2} is to average in time which appears in the estimate as integration by parts. Instead of \eqref{eq3.2}, we consider the following
regular perturbation problem as an initial approximation
\begin{equation} \label{E:x_*0}
\dot x_* = Ax_* + f(x_*, 0, t, \ep).
\end{equation}
In the rest of this section, we will use the notation
\begin{equation} \label{E:g_*}
g_*^\ep (x, t)= g(x, 0, t, \ep).
\end{equation}

\begin{lemma} \label{L:FT}
Assume (A1) -- (A3) and (A5). For any $t_0\in\mathbb{R}$, let $T>0$ and $(x(t),y(t))$ and $x_*(t)$ be solutions of
\eqref{eq3.1} and \eqref{E:x_*0} on $[t_0,t_0+T]$ such that $x(t_0)=x_*(t_0)$ and $y(t_0)=0$. Then there exists a
constant $C'$ which depends on $M,\omega,T,C_0$, and $|x_*(t_0)|_{X_1}$, such that for any $t \in [t_0,t_0+T]$,
\[ \begin{split}
|x(t)-x_*&(t)|_{X_1}+|y(t)|_{Y_1} \leq  C'\ep \big(|g_*^\ep|_{C_x^0C_t^1 (X_1 \times \R, Y)} + |D_x g_*^\ep|_{C^0 (X_1 \times \R, L(X, Y))}\big)
\end{split}\]
where $C_x^0 C_t^1$ denotes the space of functions $C^1$ in $t$ and $C^0$ in $x$.
\end{lemma}

\begin{proof}
By \eqref{eq3.1}, \eqref{eq3.2} and variation of parameters formula
\begin{eqnarray}
&(x-x_*)(t)=\int_{t_0}^te^{(t-\tau)A}(f(x,y,\tau,\ep)-f(x_*,0,\tau,\ep))d\tau, \notag\\
&y(t)=\int_{t_0}^te^{(t-\tau)\frac{J}{\ep}} (g(x,y,\tau,\ep)-g(x_*,0,\tau,\ep) + g(x_*,0,\tau,\ep)d\tau.
\label{E:y_*}
\end{eqnarray}
Due to the oscillatory nature of $e^{t\frac{J}{\ep}}$, we integrate the last terms by parts
\begin{align*}
&\int_{t_0}^te^{(t-\tau)\frac{J}{\ep}}g(x_*,0,\tau,\ep)d\tau = 
-e^{(t-\tau)\frac{J}{\ep}}\ep J^{-1}g(x_*,0,\tau,\ep)|_{t_0}^t\\
&\qquad +\int_{t_0}^te^{(t-\tau)\frac{J}{\ep}}\ep J^{-1} \big( \partial_tg(x_*,0,\tau,\ep)d\tau
+ D_xg(x_\ep,0,\tau,\ep)\big(Ax_*+f(x_*,0,\tau,\ep)\big) \big)d\tau,
\end{align*}
where we also use assumption (A3) to ensure the last term on the right hand side is well defined. Therefore,
\[\begin{split}
\Big|\int_{t_0}^te^{(t-\tau)\frac{J}{\ep}} &g(x_*,0,\tau,\ep)d\tau\Big|_{Y_1}\leq C'\ep
\big(|g_*^\ep|_{C_x^0 C_t^1 (X_1 \times \R, Y)} + |D_x g_*^\ep|_{C^0 (X_1 \times \R, L(X, Y))}\big),
\end{split}\]
where $C'$ depends on $C_0$ and $|x_*|_{X_1}$. Consequently,
\begin{eqnarray*}
(|x-x_*|_{X_1}&+|y|_{Y_1})(t)\leq C \int_{t_0}^t (|x-x_*|_{X_1}+|y|_{Y_1}) (\tau)d\tau
+ C'\ep \big(|g_*^\ep|_{C_x^0C_t^1 (X_1 \times \R, Y)}\\
&+|D_x g_*^\ep|_{C^0 (X_1 \times \R, L(X, Y))}\big).
\end{eqnarray*}
Then the desired estimates follows from the Gronwall's inequality.
\end{proof}

\begin{proof} {\it of Theorem \ref{thm3.1}}
Let $(x_1(t), y_1(t))$ be the solution of \eqref{eq3.1} with the
initial values $x_1(t_0) = x(t_0)$ and $y_1(t_0) =0$ and $x_*(t)$ be
the solution of \eqref{E:x_*0} such that $x_*(t_0) = x(t_0)$. On the
one hand, from Lemma \ref{L:FT}, for any $t \in [t_0, t_0+T]$,
\[
|x_1 - x_*|_{X_1} + |y_1|_{Y_1} \le C' \ep.
\]
On the other hand, by using the variation of parameter formula and the Gronwall's inequality, it is straight
forward to show, for any $t \in [t_0, t_0+T]$,
\[
|x- x_1|_{X_1} + |y-y_1|_{Y_1} + |x_* - x_0|_{X_1} \le C'\ep
\]
and thus the theorem follows.
\end{proof}

\begin{rem} \label{R:FT}
Combining Lemma \ref{L:FT} with the iteration of the type of the transformations \eqref{E:y_1}, we may
obtain asymptotic expansions of solutions of \eqref{eq3.1} with the leading order term given by solutions of
\eqref{E:x_*0} and the error of $O(\ep^k)$.
\end{rem}

For the linearization, we consider the following as the principle approximation
\begin{equation} \label{E:x_*0L} \begin{cases}
\dot{\delta x}_*=A\delta x_*+D_xf(x_*,0,t,\ep)\delta x_* \\
\dot{\delta y}_*=\frac{J}{\ep}\delta y_*+D_yg(x_*,0,t,\ep)\delta y_*.
\end{cases} \end{equation}

\begin{lemma} \label{L:FTL}
Assume (A1) -- (A3) and (A5) for $k=2$ and use the same notations as in Lemma \ref{L:FT}. Let $(\delta x(t),
\delta y(t))$ and $(\delta x_*(t), \delta y_*(t))$ be the solutions of \eqref{eq3.3} and \eqref{E:x_*0L} respectively
such that
\[
(\delta x(t_0), \delta y(t_0)) = (\delta x_*(t_0), \delta y_*(t_0)) \text{ and } |\delta x(t_0)|_{X_1}
+ |\delta y(t_0)|_{Y_1} \le 1.
\]
Then there exists a constant $C'$ depending on $M,\omega,T,$ and $C_0$, such that for any $t \in [t_0,t_0+T]$,
\[ \begin{split}
|\delta y(t) - \delta y_*(t)|_{Y_1} \leq C'\ep \Big(&|g_*^\ep|_{C_x^0 C_t^1 (X_1 \times \R, Y)}
+|D_x g_*^\ep|_{C_x^1 C_t^0 (X_1 \times \R, L(X, Y))} \\
&+ |D_x \p_t g_*^\ep|_{C^0 (X_1 \times \R, L(X_1,Y))} \Big).
\end{split}\]
Plus, $|\delta x(t) -\delta x_*(t)|_{X_1} \leq C' \ep$ if $|\delta y_*(t_0)|_{Y_1} \le C_1 \ep$. Otherwise
$|\delta x(t)-\delta x_*(t)|_{X} \le C'\ep$.
\end{lemma}

\begin{proof}
By standard semigroup theory in Banach space and (A2), we have
\begin{equation}\label{eq3.9}
\begin{aligned}
\Big(|\big(\delta x, \delta y\big)| + |\big(\delta x_*, \delta y_* \big)|\Big)_{C^0([t_0,t_0+T], X_1 \times Y_1)} \leq C'.
\end{aligned}
\end{equation}
First we use \eqref{eq3.3} and \eqref{E:x_*0L} to obtain
\[
\begin{aligned}
&\dot{\delta x}-\dot{\delta x}_*=A(\delta x-\delta x_*)+\big(D_xf(x,y,t,\ep)
-D_xf(x_*,0,t,\ep)\big)\delta x_*\\
& \qquad +D_xf(x,y,t,\ep)(\delta x-\delta x_*)+D_yf(x,y,t,\ep)(\delta y-\delta y_*)
+D_yf(x,y,t,\ep)\delta y_*\\
&\dot{\delta y}-\dot{\delta y}_*=(\frac{J}{\ep}+D_yg(x_*(t),0,t,\ep))(\delta y-\delta
y_*)+D_xg(x_*,0,t,\ep)\delta x\\
&\qquad  +\big(D_xg(x,y,t,\ep)-D_xg(x_*,0,t,\ep)\big)\delta
x+(D_yg(x,y,t,\ep)-D_yg(x_*,0,t,\ep))\delta y.
\end{aligned}
\]
By using assumption (A3) and \eqref{eq3.3}, we can write
\begin{eqnarray*}
\dot{\delta x}-\dot{\delta x_*}&=&A(\delta x-\delta x_*)+h_1(t,\ep)+D_yf(x,y,t,\ep)\delta y_*,\\
\dot{\delta y}-\dot{\delta y}_*&=&(\frac{J}{\ep}+D_yg(x_*(t),0,t,\ep))(\delta y-\delta y_*)+
h_2(t,\ep) +D_xg(x_*,0,t,\ep)\delta x,
\end{eqnarray*}
where by \eqref{eq3.9},
\[\begin{split}
&|h_1|_{X_1} \leq C_0(|\delta x-\delta x_*|_{X_1} +|\delta y-\delta y_*|_{Y_1}) + C' (|x-x_*|_{X_1} + |y|_{Y_1}) \\
& |h_1|_X \leq C_0(|\delta x-\delta x_*|_X +|\delta y-\delta y_*|_Y) + C' (|x-x_*|_{X_1} + |y|_{Y_1}) \\
&|h_2|_{Y_1}\leq C' (|x-x_*|_{X_1} + |y|_{Y_1}).
\end{split}\]
In the rest of the proof, we will simply write
\[
D_y g_*^\ep (t) \triangleq D_yg_*^\ep(x_*(t),t) = D_y g(x_*(t), 0, t, \ep)
\]
and similarly for $f$, $g$, and $g_x$ etc. By assumptions (A2) and (A3), $\frac{J}{\ep}+g_y$ generates an evolution
operator $E_*(t,s)$ or $E_*(t,s;x_*(s),\ep)$ which satisfies for $t\geq s$,
\[\begin{split}
&\p_s E(t,s) = - E(t,s) (\frac{J}{\ep}+D_y g_*^\ep(s)) \qquad
|E|_{L(Y,Y)} + |E|_{L(Y_1,Y_1)}\leq C_0e^{C_0(t-s)}.
\end{split}\]
The most troublesome term in the second equation is $g_x \delta x$. We use \eqref{eq3.9} and integrate by
parts to obtain
\begin{eqnarray*}
&&\Big|\int_{t_0}^tE(t,\tau) D_x g_*^\ep (\tau) \delta x (\tau)\ d\tau\Big|_{Y_1} = \ep \Big| E(t,t_0) (J +\ep
D_y g_*^\ep (t_0))^{-1} D_x g_*^\ep(t_0)\delta x(t_0) \\
&& \qquad   -(J +\ep D_y g_*^\ep (t))^{-1}D_x g_*^\ep (t) \delta x(t) +\int_{t_0}^t E(t,\tau) \p_\tau \Big(
(J + \ep D_y g_*^\ep)^{-1} D_x g_*^\ep \delta x \Big) d\tau)\Big|_{Y_1}
\end{eqnarray*}
From assumptions (A3) and (A5) and Theorem \ref{thm3.1} and equations \eqref{eq3.3} and \eqref{eq3.9}, we obtain
\[\begin{split}
\Big|\int_{t_0}^tE(t,\tau) D_x g_*^\ep (\tau) &\delta x (\tau) d\tau\Big|_{Y_1} \le C' \ep \Big(
|D_x g_*^\ep|_{C^0(X_1 \times \R, L(X, Y))} \\
&+ |D_x^2 g_*^\ep|_{C^0(X_1 \times \R, L(X \otimes X_1, Y))} + |\p_t D_x g_*^\ep|_{C^0(X_1 \times \R,
L(X_1, Y))} \Big),
\end{split}\]
which, along with the variation of parameter formula, the estimate on $h_2$, and the Gronwall inequality,
implies the desired estimate on $\delta y - \delta y_*$.

When $|\delta_* (t_0)|_{Y_1} \le C_1\ep$, it is clear $|\delta y_*(t)|_{Y_1} \leq C'\ep$ and thus the estimates
on $\delta x - \delta x_*$ also easily follows from the variation of parameter formula, the estimate on $h_1$,
and the Gronwall inequality. Otherwise, to deal with the most trouble term $D_yf(x,y,t,\ep)\delta y_*$ in the
$x$ equation, we use \eqref{E:x_*0L} to write
\[
\delta y_*(t)=\ep J^{-1}\dot{\delta y}_*(t)-\ep J^{-1} g_y (t)\delta y_*(t).
\]
Integrating by parts, we obtain from (A5) and \eqref{eq3.9}
\[\begin{aligned}
&\big|\int_{t_0}^te^{(t-\tau)A}D_yf(x,y,\tau,\ep)\delta y_*d\tau\big|_X \\
\leq &C'\ep \Big(1 + \big|\int_{t_0}^t
\p_\tau (e^{(t-\tau)A} D_y f(x,y,\tau,\ep))J^{-1}\delta y_*d\tau\big|_{X} \Big)\leq C'\ep
\end{aligned}\]
and thus the estimate on $\delta x - \delta x_*$ follows from the variation of parameter formula, the estimate
on $h_1$, and the Gronwall inequality.
\end{proof}

\begin{rem}\label{rem3.4}
If $\delta y_*(t_0)=0$, it is clear that $|\delta x - \delta x_*|_{X_1}$ satisfies the estimate as $\delta y
-\delta y_*$. Without the assumption $|\delta y_*(t_0)|_{Y_1} \le C_1 \ep$, we can not obtain the estimate on
$|\delta x-\delta x_0|_{X_1}$ since in the last step of integration by parts, there is a term
\[
\int_{t_0}^tAe^{(t-\tau)A}D_yf(x,y,\tau,\ep)J^{-1}\delta
y_0d\tau,
\]
which is only in $X$ under current assumptions.
\end{rem}

\begin{proof} {\it of Theorem \ref{thm3.2}}
Let $(\delta x_*(t), \delta y_*(t))$ be the solution of \eqref{E:x_*0L} with initial value $(\delta x (t_0),
\delta y(t_0))$. The estimates on $\delta x_* - \delta x_0$ and $\delta y_* - \delta y_0$ follow from the standard
Gronwall inequality, which along with Lemma \ref{L:FTL} implies Theorem \ref{thm3.2}.
\end{proof}

\begin{rem}\label{rem3.6}
Following from the same proof, if we assume instead
\eqref{eq3.8} by
\[\begin{aligned}
&|\delta x_0(t_0)|_{X_1}+|\delta y_0(t_0)|_{Y_1}\leq 1 \qquad |\delta x(t_0)-\delta
x_0(t_0)|_{X}+|\delta y(t_0)-\delta y_0(t_0)|_{Y}\leq C_1\ep.
\end{aligned}\]
Then there exists $C'$ such that for all $t\in[t_0,t_0+T]$,
\[
|\delta x(t)-\delta x_0(t)|_{X}+|\delta y(t)-\delta y_0(t)|_{Y}\leq C'\ep.
\]
\end{rem}

\section{Invariant Manifold} \label{S:InMa}

In this section, we study the local integral manifold (as the system may be non-autonomous)
of a stationary solution of \eqref{eq3.1}, namely, the center-unstable (stable)
manifold, unstable (stable) manifold and etc in the framework of the Lyapunov-Perron integral equation.
The main point is to obtain these manifolds of size $O(1)$ and their leading order approximations.
Hypotheses (A1)-(A4) will be assumed and (A5) will be
needed in some theorems with from Theorem \ref{thm4.14} as specified.

\subsection{Preliminary}

In addition to (A1)-(A4) we assume
\begin{enumerate}
\item[(B1)] $\partial_tf(x,y,t,0)=\partial_tg(x,y,t,0)=0$,
\item[(B2)] $f(0,0,t,\ep)=g(0,0,t,\ep)=0$,
\item[(B3)] When $k=1$, assume $(Df,Dg)$ are equicontinuous functions in $x,y$ and
$\ep$ with respect to $t$ at $x=0,y=0,\ep=0$, i.e. for any $s>0$,
there exists $\delta>0$ such that if
$|x|_{X_1}<\delta,|y|_{Y_1}<\delta,|\ep|<\delta$, for any
$t\in\mathbb{R}$,
\[\begin{aligned}
&\big|Df(x,y,t,\ep)-Df(0,0,t,0)\big|_{L(X_1\times Y_1,X_1),L(X\times
Y,X)}<s,\\
&\big|Dg(x,y,t,\ep)-Dg(0,0,t,0)\big|_{L(X_1\times Y_1,Y_1),L(X\times
Y,Y)}<s.
\end{aligned}\]
\end{enumerate}
We will write
\[
f_{x, y} \triangleq D_{x,y}f(0,0,t,0), \qquad g_{x, y} \triangleq D_{x,y} g(0,0,t,0)
\]
which are independent of $t$.

For $(x,y)\in X_1\times Y_1$, let
\begin{equation}\label{eq4.1}
\begin{aligned}
F_1(x,y,t,\ep)=f-f_xx-f_yy\qquad G_1(x,y,t,\ep)=g-g_xx-g_yy
\end{aligned}
\end{equation}
and $\lambda(r)\in C_c^{\infty}(\mathbb{R})$ such that
\[
\lambda(r)=\left\{\begin{aligned}   &1,\hspace{3mm}r<\frac{1}{3}\\
&0, \hspace{3mm}r>1 \end{aligned} \right.\hspace{2mm}, \hspace{11mm}
|\lambda^\prime(r)|\leq 3.
\]
To take the advantage of the linear approximation, we cut off the nonlinearity
\begin{equation}\label{eq4.2}
\begin{aligned}
&F(x,y,t,\ep)=\lambda(\frac{|x|_{X_1}+|y|_{Y_1}}{r})F_1 \qquad
G(x,y,t,\ep)=\lambda(\frac{|x|_{X_1}+|y|_{Y_1}}{r})G_1,
\end{aligned}
\end{equation}
then by assumption (A3), (A5) and (B2), $F$ and $G$ satisfy:
\begin{equation}\label{eq4.3}
\begin{aligned}
&F(0,0,t,\ep)=G(0,0,t,\ep)=0 \qquad |F|_{X_1} + |G(x,y,t,\ep)|_{Y_1} \leq \ovr r,
\\
&|D (F, G)|_{L(X_1 \times Y_1)} + |D (F, G)|_{L(X \times Y)} \leq \ovr
\end{aligned}
\end{equation}
where $\ovr=\ovr(r,\ep) = C \widetilde{r}$ with $C$ depending only on $C_0$ and
\begin{equation}\label{eq4.4}
\begin{aligned}
&\widetilde{r}=\widetilde{r}(r,\ep_0) =\sup_{\scriptstyle|x|_{X_1}+|y|_{Y_1}\leq
r,\atop\scriptstyle\ep\in[0,\ep_0), \; t\in \R}
|D(F_1, G_1)|_{L(X_1 \times Y_1)} + |D(F_1, G_1)|_{L(X \times Y)}.
\end{aligned}
\end{equation}
Clearly, (A3) implies $\lim\limits_{r,\ep_0\rightarrow0}\ovr=0$. Let
\[
\af=A+f_x.
\]
Since \eqref{eq3.1} is non-autonomous, in order to construct the local integral manifolds at time $t_0$,
we translate the equation by $t_0$ and modify it by cutting off the nonlinearity
\begin{equation}\label{eq4.5}
\left\{\begin{aligned}
&\dot{x}(t)=\af x+f_y y + F(x,y,t+t_0,\ep)\\
&\dot{y}(t)=(\frac{J}{\ep}+g_y) y+ g_x x + G(x,y,t+t_0,\ep).
\end{aligned}\right.
\end{equation}
Clearly the system is unchanged in the $\frac r3$-neighborhood of $(0, 0)$ and we will
construct its global integral manifold which can be characterized by the exponential
decay as $t \to \pm\infty$ at a rate close to that of $e^{t\af}|_{X^{s, cu}}$. Naturally
we need the following weighted continuous function spaces. In general, given
$\eta\in\mathbb{R}$ and $Z$ a Banach space, let
\[
C_{\eta}^\pm(Z)=\Big\{z \in C^0(\R^\pm, Z) \big|\sup_{\pm t\geq 0}e^{-\eta
t}|z(t)|_Z< +\infty\Big\}
\]
with norm $|\cdot|_{\eta,\eps,Z}^\pm$, where $\eps$ is a parameter
\[
|z(\cdot)|_{\eta,\eps,Z}^\pm=\sup_{\pm t\geq 0}e^{-\eta
t}\frac{|z(t)|_Z}{\eps}.
\]
In order to hand the linear terms $f_y$ and $g_x$ in system \eqref{eq4.5} which are not
there in the usual study of local invariant manifold, we will introduce the following
spaces which allow us to average. Let
\begin{equation}\label{eq4.6}\begin{split}
B_{\eta}^\pm(\rho)=\Big\{\big(x,y\big)\in
&C_{\eta}^\pm (X_1)\times C_{\eta}^\pm(Y_1)\Big|\\
&\big|(x,y)\big|_{\eta,\eps}^\pm=
|x|_{\eta,1,X_1}^\pm+|y|_{\eta,\eps,Y_1}^\pm+|\dot{x}|_{\eta,1,X}^\pm<\rho\Big\},\end{split}
\end{equation}
where we also use $B_{\eta}^\pm(\infty)$ to denote the corresponding Banach space.

\subsection{Existence and smoothness of integral manifolds}

We start with the center-unstable and stable manifold. Assume the exponential
dichotomy of the linearized system at $(0,0)$
\begin{enumerate}
\item[(B4)] There exists a pair of continuous projections $(P_s,P_{cu})$ on $X$, such
that $P_s+P_{cu}=I_X$ \footnote[1]{Throughout this manuscript, we
will use notations $I_X$, $I_Y$ for identity maps on $X$, $Y$,
respectively. With slight abuse of notation, we also use the
projections $P_s$ and $P_{cu}$ to denote their composition with the
projection from $X\times Y$ to $X$.}, clearly $X=P_sX\oplus P_{cu}X$
and
\begin{eqnarray*}
X^{s,cu} \triangleq P_{s,cu}X\hspace{1.5mm}&&\mbox{are positively invariant under}\quad
e^{t\af},\\
e^{t\af}\hspace{1.5mm}&&\mbox{can be extended to a group on}\quad
X^{cu}.
\end{eqnarray*}
\item[(B5)] There exist constants $a_1<0$ and $a_1<a_2$, such that
\begin{equation*} \begin{aligned}
&& &|e^{t\af}P_{s}x|_{X}\leq Ke^{a_1t}|x|_{X}\quad&&\mbox{for}&&\quad
t\geq0,\quad x\in X,\\
&& &|e^{t\af}P_{cu}x|_{X} + |e^{t(\frac{J}{\ep}+g_y)}y|_{Y} \leq
Ke^{a_2t}(|x|_{X} + |y|_{Y}) \quad&&\mbox{for}&&\quad t\leq0,\quad x\in X,\; y \in Y
\end{aligned}\end{equation*}
\end{enumerate}

\begin{rem}\label{rem4.1}
Let $P_sX_1=X_1^s$ and $P_{cu}X_1=X_1^{cu}$, (B4) and (B5) imply
$e^{t\af}$ and $e^{t(\frac{J}{\ep}+g_y)}$ satisfy the same estimates
with all norms replaced by $|\cdot|_{X_1}$ and $|\cdot|_{Y_1}$.
\end{rem}

\begin{theorem}\label{thm4.3}
Assume hypotheses (A1)--(A4), (B1)--(B5), and there exists $\eta$
such that $\eta, k \eta \in (a_1, a_2)$ then
\begin{enumerate}
\item There exist $r>0$ and $\ep_0>0$ and a mapping $h_s: B_r (0, X_1^{cu}
\times Y_1) \times \R \times (0, \ep_0) \to B_r (0, X_1^s)$ such that the family of
graphs $\CM_\ep^{cu} (t_0)$ form a locally invariant center-unstable integral
manifold of $(0, 0)$ of \eqref{eq3.1}.
\item A backward flow is well-defined on the center-unstable manifold.
\item $h_s(\xi_{cu}, \xi_y, t_0, \ep)$ is $C^k$ in $\xi_{cu}$ and $\xi_y$
and continuous in $t_0$ with the norms independent of $\ep$.
\item If we assume hypothesis (A5') to be introduced before Proposition
\ref{thm4.14}, then
\[
\partial_{t_0}h_s(\cdot,\cdot,\ep) \in C^0(B_r (0, X_1^{cu} \times Y_1)
\times\mathbb{R},\, X^s)
\]
with the norm independent of $\ep$.
\end{enumerate}
\end{theorem}

Similar statements are given in Theorem \ref{thm4.5}, \ref{thm4.6},
and \ref{thm4.7} for stable, unstable, center-unstable, and center
manifolds. These invariant manifolds are constructed in small, of
$O(1)$ though, neighborhoods of $0$. Sometimes, we do want to track
the invariant manifolds in larger ranges. This can be achieved by
combining the local invariant theorems and Theorem \ref{thm3.1} and
\ref{thm3.2}. See Proposition \ref{thm4.20}.

As in the standard approach, we will work on the global center-unstable
manifold of \eqref{eq4.5} which yields the local center-unstable manifold
of \eqref{eq3.1}.

For any $\eta\in(a_1,a_2)$, there
exists $\underline{\ep}>0$ such that for any $\eps\in[0,\underline{\ep})$, there exist
$r_0,\ep_0>0$ satisfying that for any $\ep\in[0,\ep_0)$ and $r\in(0,r_0)$, it holds
\begin{equation}\label{eq4.7}
\sigma(\eta)=\min\{\sigma_1,\sigma_2(\eta),\sigma_3(\eta)\}>0,
\end{equation}
where
\begin{eqnarray*}
&\sigma_1=\frac{1}{2}-C_0^2\ep, \qquad \sigma_2(\eta)=1-3(\frac{K}{a_2-\eta}+\frac{K}{\eta-a_1}+1)\big(\ovr+C_0\eps
\big),\\
&\sigma_3(\eta)=1-\frac3{\eps} (\frac{K}{a_2-\eta}+K+1)\big(\ovr+2C_0^2\ep\big).
\end{eqnarray*}
In the rest of this whole section, we always assume $\eqref{eq4.7}$.

Given $t_0 \in \R$, to simplify notation, we write
\[
F(x,y,\tau+t_0,\ep) \triangleq F(x(\tau),y(\tau),\tau+t_0,\ep)=F(x,y,\tau+t_0,\ep) \qquad f_y y \triangleq f_yy(\tau)
\]
and such notation also applies to $G$ and $g$. Let
\[
U(t,\ep)=\begin{pmatrix}   e^{t\af} & 0\\
0 & e^{t(\frac{J}{\ep}+g_y)}
\end{pmatrix},
\]
and for any $(x,y)\in B_{\eta}^{-}(\rho)$ and $\xi=(\xi_{cu},\xi_y)\in X_1^{cu}\times Y_1$,
\begin{eqnarray*}\label{eq4.3}
\mathscr{T}_{cu}(x,y,\xi,t_0,\ep)(t)&=&U(t,\ep)\xi +\int_{0}^{t}U(t-\tau,\ep)\begin{pmatrix}
    P_{cu}\Big(F(x,y,\tau+t_0,\ep)+f_yy\Big)\\G(x,y,\tau+t_0,\ep)+g_xx
  \end{pmatrix}d\tau\\\nonumber
  &&+\int_{-\infty}^{t}U(t-\tau,\ep)\begin{pmatrix}
  P_{s}\Big(F(x,y,\tau+t_0,\ep)+f_yy\Big)\\0
  \end{pmatrix}d\tau.\nonumber
\end{eqnarray*}
It is standard to verify that $(x,y) \in B_{\eta}^{-}(\rho)$ is a fixed point of $\mathscr{T}$ if and
only if $(x(t), y(t)$ is a solution of \eqref{eq4.5} with $(I-P_s)(x, y)(0)=(\xi_{cu}, \xi_y)$. Thus
we focus on the fixed point equation
\begin{equation}\label{eq4.8}
(x, y)=\mathscr{T}_{cu}(x,y,\xi,t_0,\ep).
\end{equation}

\begin{lemma}\label{lemma4.2}
For any $\eta$ with $a_1<\eta<a_2$ and $\eps,r,\ep_0$ satisfy
\eqref{eq4.7}, there exists $\rho_0$ depending on
$|\xi_{cu}|_{X_1},|\xi_y|_{Y_1},K,\eps,\sigma$ ($\sigma$ defined in
\eqref{eq4.7}), such that for any
$\ep\in[0,\ep_0)$ and $\rho\in[\rho_0,\infty]$, $\mathscr{T}_{cu}$
defines a contraction mapping on $B_{\eta}^{-}(\rho)$ under the norm
$|\cdot|_{\eta,\eps}^{-}$.
\end{lemma}

\begin{proof}
To keep the exposition clean, we will skip in most places the parameters
$\xi_{cu}$, $\xi_y$, $t_0$, and $\ep$, which are fixed in this lemma. It
is easy to obtain from the definition,
\begin{eqnarray*}
&&\sup_{t\leq 0}e^{-\eta t}\Big|(P_s+P_{cu})\mathscr{T}_{cu}(x,y)(t)
\Big|_{X_1} \\
\leq && K|\xi_{cu}|_{X_1}+\big(\frac{K}{a_2-\eta}+\frac{K}{\eta-a_1}\big)\big(|DF|_{C^0}+\eps|f_y|\big)
\big|(x,y)\big|_{\eta,\eps}^{-}.
\end{eqnarray*}
Since $(x,y)\in C_{\eta}^{-}(X_1)\times C_{\eta}^{-}(Y_1)$ and $\af$
is a closed operator, one can verify $(P_s+P_{cu})\mathscr{T}_{cu}(x,y)
\in C^1\big((-\infty,0),X\big)$ and
\begin{equation}\label{eq4.9}
\begin{aligned}
&(P_s+P_{cu})\frac{d}{dt}\mathscr{T}_{cu}(x,y)(t)\\
=&\af (P_s+P_{cu}) \mathscr{T}_{cu}(x,y)(t) + F(x(t),y(t),t+t_0,\ep)+f_yy(t).
\end{aligned}
\end{equation}
Therefore,
\begin{eqnarray*}
&&\sup_{t\leq 0}e^{-\eta
t}\Big|(P_s+P_{cu})\frac{d}{dt}\mathscr{T}_{cu}(x,y)(t)
\Big|_{X}\\
&\leq&
K|\xi_{cu}|_{X_1}+\big(\frac{K}{a_2-\eta}+\frac{K}{\eta-a_1}+1\big)\big(|DF|_{C^0}+\eps|f_y|\big)
\big|(x,y)\big|_{\eta,\eps}^{-}.
\end{eqnarray*}
Again, by the definition of $\mathscr{T}_{cu}$ and integration by parts
\begin{eqnarray*}
&&(I-P_s-P_{cu})\mathscr{T}_{cu}(x,y)(t)
\\
&=&e^{t(\frac{J}{\ep}+g_y)}\xi_y+\int_{0}^{t}e^{(t-\tau)(\frac{J}{\ep}+g_y)}G(x,y,\tau+t_0,\ep)d\tau \\
&&-\ep e^{(t-\tau)(J+\ep g_y)}(\frac{J}{\ep}+g_y)^{-1}g_xx\Big|_{0}^{t}
+\ep \int_{0}^{t}e^{(t-\tau)(J+ \ep g_y)}(\frac{J}{\ep}+g_y)^{-1}g_x\dot{x}d\tau.
\end{eqnarray*}
Consequently, we obtain
\begin{eqnarray*}
&&\sup_{t\leq 0}\frac{1}{\eps}e^{-\eta
t}\Big|(I-P_s-P_{cu})\mathscr{T}_{cu}(x,y)(t)
\Big|_{Y_1}\\
&\leq&
\frac{K|\xi_y|_{Y_1}}{\eps}+(\frac{K}{a_2-\eta}+K+1)\frac{|DG|_{C^0}+2\ep|J^{-1}||g_x|}{\eps}
\big|(x,y)\big|_{\eta,\eps}^{-}.
\end{eqnarray*}
Using \eqref{eq4.7}, clearly, there exists $\rho_0>0$ determined by
$|\xi_y|_{Y_1},|\xi_{cu}|_{X_1},K,\eps$ and $\sigma$ such that for
any $\rho\in(\rho_0,+\infty]$, the above inequalities imply
that $\mathscr{T}_{cu}$ maps $B_{\eta}^{-}(\rho)$ to
$B_{\eta}^{-}(\rho)$.

To prove it is a contraction, we can estimate
in a similar fashion
\begin{equation}\label{eq4.10}
\begin{aligned}
&\sup_{t\leq 0}e^{-\eta
t}\Big|(P_s+P_{cu})\big(\mathscr{T}_{cu}(x_1,y_1)-
\mathscr{T}_{cu}(x_2,y_2)\big)(t)\Big|_{X_1}\\
\leq&(\frac{K}{a_2-\eta}+\frac{K}{\eta-a_1})(|DF|_{C^0}+\eps|f_y|)
(|x_1-x_2|_{\eta,1,X_1}^{-}+|y_1-y_2|_{\eta,\eps,Y_1}^{-}),\\
&\sup_{t\leq 0}e^{-\eta
t}\Big|(P_s+P_{cu})\frac{d}{dt}\big(\mathscr{T}_{cu}(x_1,y_1)-
\mathscr{T}_{cu}(x_2,y_2)\big)(t)\Big|_{X}\\
\leq&\big(\frac{K}{a_2-\eta}+\frac{K}{\eta-a_1}+1\big)\big(|DF|_{C^0}+\eps|f_y|\big)
\big(|x_1-x_2|_{\eta,1,X_1}^{-}+|y_1-y_2|_{\eta,\eps,Y_1}^{-}\big),\\
&\sup_{t\leq 0}e^{-\eta
t}\Big|(I-P_s-P_{cu})\big(\mathscr{T}_{cu}(x_1,y_1)-
\mathscr{T}_{cu}(x_2,y_2)\big)(t)\Big|_{Y_1}\\
\leq&\big(\frac{K}{a_2-\eta}+K+1\big)\frac{|DG|_{C^0}+2\ep|J^{-1}||g_x|}{\eps}
\big|(x_2 -x_1, y_2 -y_1)|_{\eta, \eps}^-
\end{aligned}
\end{equation}
Therefore, $\mathscr{T}_{cu}$ defines a contraction mapping on
$B_{\eta}^{-}(\rho)$ under the norm $|\cdot|_{\eta,\eps}^{-}$.
\end{proof}

Many other proofs in this manuscript will be very much in the fashion of that of the
above lemma in the sense that integration by parts often provides an effective way to
provide an extra $\ep$ in the estimate. We will skip some details.

For any $(\xi,t_0,\ep)\in X^{cu}_1\times Y_1\times\mathbb{R}^2$, let
$\big(x(t),y(t)\big)$ be the fixed point of $\mathscr{T}_{cu}$, and
\begin{equation}\label{eq4.11}
h_{s}(\xi,t_0,\ep)=P_s x(0)=\int_{-\infty}^{0}
e^{-\tau\af}P_s\Big(F\left(x,y,\tau+t_0,\ep\right)+f_yy\Big) d\tau,
\end{equation}
\[
\mathcal{M}_{\ep}^{cu} (t_0)=\big\{\xi+h_{s}(\xi,t_0,\ep)\big|
\xi\in X_1^{cu}\times Y_1\big\}.
\]
From the standard argument based on the uniqueness of the contraction, one can show that $h_s$ and
the center-unstable manifold $\mathcal{M}_{\ep}^{cu} (t_0)$, for any $t_0\in \R$ are independent of
$\eta\in(a_1,a_2)$. Moreover, the flow map of \eqref{eq4.5} starting on $\mathcal{M}_{\ep}^{cu} (t_0)$ is
well-defined both forwardly and backwardly in $t$ and, from $t_1$ to $t_2$, it maps $\mathcal{M}_\ep^{cu}
(t_0 + t_1)$ to $\mathcal{M}_\ep^{cu}(t_0 + t_2)$.

\begin{rem}\label{rem4.4}
Note since $X_1^{cu}\times Y_1\neq T_0\mathcal{M}_\ep^{cu}$, we
cannot prove $h_s$ is bounded. See the Appendix. Also, as usual, $\mathcal{M}_{\ep}^{cu} (t_0)$
depends on the cut-off and thus is not unique for \eqref{eq3.1}.
\end{rem}

\noindent {\bf Smoothness in $\xi_{cu}$ and $\xi_y$.} Let $z(\xi)(t)
= (x(t), y(t))$ denote the fixed point of $\ST (\xi)$. From \eqref{eq4.8},
formally,
\begin{equation} \label{E:Dz}
D_\xi z(\xi) =(\phi, \psi) = \big(I - D_z \ST_{cu} (z, t_0, \ep)\big)^{-1} U
\end{equation}
where
\begin{eqnarray*}
D_z \mathscr{T}_{cu} (z, t_0)(\phi,\psi) (t)
&=&\int_{0}^{t}U(t-\tau)\begin{pmatrix}
    P_{cu}\big(DF(z, \tau + t_0)(\phi,\psi)
    +f_y\psi\big)
    \\DG(z, \tau + t_0, \ep)(\phi,\psi)+g_x\phi
  \end{pmatrix}d\tau\\
  &&+\int_{-\infty}^{t}U(t-\tau)\begin{pmatrix}
  P_{s}\big(DF(z, \tau +t_0)(\phi,\psi)+f_y\psi\big)\\ 0
  \end{pmatrix}d\tau
\end{eqnarray*}
with the parameter $\ep$ skipped. It is easy to verify that $\big(I -
D_z \ST_{cu} (z, t_0, \ep)\big)^{-1}$ and thus the right side of
\eqref{E:Dz} are well-defined with respect to the same exponential rate
$\eta$. In order to estimate $z(\xi +\xi') - z(\xi)- D_\xi z (\xi) \xi'$,
consider
\begin{eqnarray*}
& &\big(I - D_z \ST_{cu} (z, t_0, \ep)\big) \big(z(\xi +\xi') - z(\xi)
- D_\xi z (\xi) \xi'\big)\\
& =& z(\xi +\xi') - z(\xi) - D_z \ST_{cu} (z, t_0, \ep) \big(z(\xi
+\xi') - z(\xi)\big) - U \xi'.
\end{eqnarray*}
In the above right side the linear terms $f_y$ and $g_x$, which do not vanish at
$(0, 0)$, actually disappear. Moreover, the evolution operator has a uniform bound
though it depends on $\ep$. Therefore the exactly standard argument \cite{CL88, CLL91}
(where no differentiation in $t$ is need which would product $\frac 1\ep$) implies
the $C^1$ smoothness of $z(\cdot, t_0, \ep) : X_1^{cu} \times Y_1 \to B_{\eta'}^-$
for any $\eta' < \eta$. This establish the $C^1$ smoothness of $\CM_\ep^{cu}(t_0)$ in
$\xi_{cu}$ and $\xi_y$. Similarly, the $C^k$ smoothness of $\CM_\ep^{cu}$ can be
obtained if we assume that there exists $\eta$ such that $a_1 < \eta, \, k\eta < a_2$.

\noindent {\bf Dependence on $t_0$.} In order to smoothness of
$\mathcal{M}_{\ep}^{cu}$ in $t_0$, in addition to hypotheses (A1)-(A4) for $k=1$
and (B1)-(B3), we assume and
\begin{enumerate}
\item[(A5$'$)] $(\partial_tf,\partial_tg)\in C^0(X_1\times
Y_1\times\mathbb{R}^2,X\times Y)$, $(D\partial_tf,D\partial_tg)\in
C^0(X_1\times Y_1\times\mathbb{R}^2,L(X_1\times Y_1,X\times Y))$,
$(D\partial_{\ep}\partial_tf,D\partial_{\ep}\partial_tg)\in
C^0(X_1\times Y_1\times\mathbb{R}^2,L(X_1\times Y_1,X\times Y))$
Moreover, their norms are bounded by $C_0$.
\end{enumerate}

\begin{rem}\label{rem4.13}
In fact, the assumptions in (A5$'$) on $D\partial_t (f, g)$ and $D\p_\ep
\partial_t(f, g)$ are only needed when one has to work with $\eta>0$. In fact, if
$\eta<0$, our next theorem still holds if we only assume $\p_t(f, g)
\in X \times Y$ $\p_t \p_\ep (f, g) \in X \times Y$.
\end{rem}

To avoid dealing with $\dot x(t)$ as in the above contraction argument, which may introduce a factor of
$\frac 1\ep$, we introduce a slight variation of $\ST_{cu}$ which will be somewhat more easily used in the
proofs of the following proposition as no time derivative is directly involved. Given parameters $\xi \in
X_1^{cu} \times Y_1$, $t_0$, $\ep$, for $z=(x,y) \in C_{\eta}^{-}(X_1) \times C_{\eta}^{-}(Y_1)$, let
\begin{equation}\label{eq4.41}
\wtt(z,t_0)(t) = U(t,\ep)\xi + (\phi(t), \psi(t))
\end{equation}
where
\begin{align*}
\phi(t) = \int_{0}^{t} e^{(t-\tau)\af} P_{cu}\big(&F(z,\tau+t_0,\ep)+f_yy\big) d\tau \\
&+ \int_{-\infty}^{t} e^{(t-\tau)\af} P_s \big(F(z,\tau+t_0,\ep)+f_yy\big) d\tau
\end{align*}
and
\begin{align*}
\psi(t) = & \int_{0}^{t} e^{(t-\tau) (\frac J\ep +g_y)} G(z,\tau+t_0,\ep) d\tau
-\ep (J+\ep g_y)^{-1} \big (g_x x(t) + e^{t(\frac{J}{\ep}+g_y)}g_xx(0) \big)\\
& + \ep \int_0^t(J +\ep g_y)^{-1}e^{(t-\tau)(\frac{J}{\ep}+g_y)}g_x(\af
x+F(z,\tau+t_0,\ep)+f_yy)d\tau.
\end{align*}
We equip $C_{\eta}^{-}(X_1)\times C_{\eta}^{-}(Y_1)$ with the norm
\begin{equation}\label{eq4.42}
\|z\|_{\eta}^1=\sup_{t\leq0}e^{-\eta
t}\big(|x(t)|_{X_1}+\frac{|y(t)|_{Y_1}}{\eps}\big),
\end{equation}
and we will also use $\|\cdot\|_{\eta}$ to denote
\[
\|z\|_{\eta}=\sup_{t\leq0}e^{-\eta
t}\big(|x(t)|_{X}+\frac{|y(t)|_{Y}}{\eps}\big).
\]
We denote the balls in these norms by
\begin{equation} \label{E:balls}
\widetilde{B}_{\eta'}^{-}(\infty)=\Big\{z\big|\|z\|_{\eta'}^1<\infty\Big\} \qquad
\overline{B}_{\eta}^{-}(\infty)=\Big\{z\big|\|z\|_{\eta}<\infty\Big\}
\end{equation}
Obviously, $\wtt$ is come up with from $\ST$ after integrating by parts in the $y$
component. Using (A3), it is straight forward to prove that $\wtt$ is still a
contraction on $C_{\eta}^{-}(X_1)\times C_{\eta}^{-}(Y_1)$ under the norm in \eqref{eq4.42}. Moreover, its
linearization is also a contraction under both of the norms $\|\cdot\|_\eta^1$ and $\|\cdot\|_\eta$. Namely, for
some $0<\sigma'<1$ with a similar form as $\sigma$ defined in \eqref{eq4.7}, we have
\begin{equation}\label{eq4.43}
\left\|\wtt(z,t_0)-\wtt(z',t_0)\right\|_{\eta}\leq(1-\sigma')\|z-z'\|_{\eta} \qquad \|D \wtt z\|_\eta \leq
(1-\sigma')\|z\|_{\eta}.
\end{equation}
By the uniqueness, $\wtt$ and $\mathscr{T}_{cu}$ have the same fixed point.

\begin{rem}\label{rem4.15}
In the following, we will repeatedly use the fact that $z(\xi)$, the fixed
point of $\wtt$ and $\ST_{cu}$, belongs to $\wt{B}_{\eta}^{-}(\infty)$ for
any $\eta\in(a_1,a_2)$ as long as \eqref{eq4.7} is satisfied.
\end{rem}

Since the time derivatives involves unbounded operators $A$ and $J$, we do not
expect $\CM_{cu}^\ep(t_0)$ to be smooth in $t_0$ in $X_1 \times Y_1$.

\begin{prop}\label{thm4.14}
Assume the conditions in Theorem \ref{thm4.3} for $k=1$, then
\[
h_s(\cdot,\cdot,\ep)\in C^0(X_1^{cu}\times Y_1\times\mathbb{R},X_1^s).
\]
If we further assume (A5$'$),
\[
\partial_{t_0}h_s(\cdot,\cdot,\ep)\in C^0(X_1^{cu}\times Y_1\times
\mathbb{R},X^s).
\]
\end{prop}

\begin{rem} \label{R:p_th_s1}
If in (A5) and (A5') we assume the smoothness of $(f, g) : X_1 \times Y_1 \times \R^2 \to X_1 \times Y_1$, the same
proof implies that $\p_{t_0} h_s \in X_1^s$ is continuous. See also Remark \ref{R:p_th_s2}.
\end{rem}

\begin{proof}
To simplify notations, we will also ignore the $\ep$ variable in $F,G$.
We first claim
\begin{equation}\label{eq4.44}
\lim_{t_1\rightarrow
t_0} \left\|\wt{\mathscr{T}}_{cu}(z,t_1)-\wt{\mathscr{T}}_{cu}(z,t_0)\right\|_{\eta}^{1}=0,
\end{equation}
for any $a_1<\eta<\eta'<a_2$ and $z\in \wt{B}_{\eta'}^{-}(\infty)$,
where
$\widetilde{B}_{\eta'}^{-}(\infty)=\Big\{z\big|\|z\|_{\eta'}^1<\infty\Big\}$.
In fact, for any $s>0$, we will show the above quantity is bounded by some $Cs$ when
$|t_1 - t_0|<<1$. Let $T_2 = \displaystyle\frac{\log{s}}{\eta'-\eta}$. By (A3),
$DF(pz(t),t'),DG(pz(t),t')$ are uniformly continuous on
$(p,t,t')\in[0,1]\times[T_2,0]\times[T_2+t_0-1,t_0+1]$. Therefore,
there exists $s'>0$ such that if $|t_1-t_0|<s'$,
\[\begin{aligned}
&\big|D (F, G) (pz(t),t+t_1)-D(F, G)(pz(t),t+t_0)\big|_{L(X_1\times
Y_1,X_1 \times Y_1)}<s
\end{aligned}\]
for $(p,t)\in[0,1]\times[T_2,0]$. Rewrite $F(z(t),t+t_1)
-F(z(t),t+t_0)$ as
\begin{equation}\label{eq4.45}
\begin{aligned}
&F(z(t),t+t_1)-F(z(t),t+t_0)\\
=&F(z(t),t+t_1)-F(0,t+t_1)-F(z(t),t+t_0)+F(0,t+t_0)\\
=&\left(\int_0^1DF(pz(t),t+t_1)-DF(pz(t),t+t_0)dp\right)z(t).
\end{aligned}
\end{equation}
It follows that for
$|t_1-t_0|<s'$,
\begin{equation}\label{eq4.46}
\begin{aligned}
&\big|F(z(t),t+t_1)-F(z(t),t+t_0)\big|_{X_1}\leq s|z(t)|_{X_1\times Y_1},
\end{aligned}
\end{equation}
which is also true for $G$. To obtain \eqref{eq4.44}, we split the
integration intervals in the definition of $\wtt$ into $t<T_2$
and $t>T_2$. On the former, the estimate can be obtained by using
\eqref{eq4.46} and the exponential bound of $z$ and on the latter we
only need to notice
\[
|(F, G) (z(\tau))|_{X_1} \le C \bar r |z(\tau)|_{X_1 \times Y_1} \le C\bar r
e^{\eta \tau} < C\bar r s e^{\eta' \tau}, \qquad \text{ for } \tau >T_2.
\]
An immediate consequence of \eqref{eq4.44} is
\begin{equation} \label{eq4.47}
\wtt\in C^0(\widetilde{B}_{\eta'}^{-}(\infty)\times\mathbb{R},
\widetilde{B}_{\eta}^{-}(\infty)).
\end{equation}
In fact, by the same procedure we can also prove the following stronger statements,
\begin{eqnarray}\label{eq4.48}
D\wtt&\in&C^0(\widetilde{B}_{\eta'}^{-}(\infty)\times\mathbb{R},
L(\widetilde{B}_{\eta'}^{-}(\infty),\widetilde{B}_{\eta}^{-}(\infty)),
\end{eqnarray}
where $a_1<\eta<\eta'<a_2$ and $D$ is the differentiation with
respect to $z$.

Let $z_i$ be the fixed point of $\wt{\mathscr{T}}_{cu}(\cdot,t_i)$
for $i=0,1$. From \eqref{eq4.43} it is easy to see
\begin{eqnarray*}
\|z_1-z_0\|_{\eta}^{1}&\leq&(1-\sigma')\|z_1-z_0\|_{\eta}^{1}+\Big\|
\wt{\mathscr{T}}_{cu}(z_0,t_1)-\wt{\mathscr{T}}_{cu}(z_0,t_0)\Big\|_{\eta}^{1}.
\end{eqnarray*}
Together with \eqref{eq4.44} and Remark \ref{rem4.15}, it implies
\begin{equation}\label{eq4.49}
\lim_{t_1\rightarrow t_0}\|z(t_1)-z(t_0)\|_\eta^1 \leq\frac{1}{\sigma'}
\lim_{t_1\rightarrow t_0}\Big\|\mathscr{T}_{cu}(z_0,t_1)-\mathscr{T}_{cu}(z_0,t_0)\Big\|_\eta^1=0,
\end{equation}
From the definition of $h_s$, we obtain its continuity in $t_0$.

To prove the second part, by our assumptions in (A4) and (A5$'$)
involving $X$ and $Y$, one may also prove
\begin{eqnarray}\label{eq4.50}
D\wtt&\in&C^0(\widetilde{B}_{\eta'}^{-}(\infty)\times\mathbb{R},
L(\overline{B}_{\eta'}^{-}(\infty),\overline{B}_{\eta}^{-}(\infty)))\\
\label{eq4.51}
\partial_{t_0}\wtt&\in&C^0(\widetilde{B}_{\eta'}^{-}(\infty)\times\mathbb{R},
\overline{B}_{\eta}^{-}(\infty)),
\end{eqnarray}
where
$\overline{B}_{\eta}^{-}(\infty)=\Big\{z\big|\|z\|_{\eta}<\infty\Big\}$.

Since (B2) implies $F(0,t)=G(0,t)=0$, we write
\begin{equation}\label{eq4.52}
\begin{aligned}
&F(z_0,t+t_1)-F(z_0,t+t_0)-\partial_{t_0}F(z_0,t+t_0)(t_1-t_0)\\
=&(t_1-t_0)\Big(\int_0^1\int_0^1D\partial_{t_0}F(qz_0,t+pt_1+(1-p)t_0)\\
&-D\partial_{t}F(qz_0,t+t_0)dqdp\Big)\big(z_0(t)\big).
\end{aligned}
\end{equation}
Assumptions (A5$'$), \eqref{eq4.52}, and a similar estimate for $G$ yield
$\|\partial_{t_0}\wtt(z)\|_{\eta}\leq C'\|z\|_{\eta}^1$, which implies,
for any $z \in \widetilde{B}_{\eta'}^{-}(\infty)$,
\begin{equation}\label{eq4.53}
\left\|\wtt(z,t_1)-\wtt(z,t_0)\right\|_{\eta}\leq
C'|t_1-t_0|\|z\|_{\eta}^1,
\end{equation} where $C'$ depends on $C,a_1,a_2,\eta$. Since
\begin{equation}\label{eq4.54}
z_1-z_0=\wtt(z_1,t_0)-\wtt(z_0,t_0)+\wtt(z_1,t_1)-\wtt(z_1,t_0),
\end{equation}using \eqref{eq4.50} and \eqref{eq4.53}, we obtain
\begin{equation}\label{eq4.55}
\|z_1-z_0\|_{\eta}\leq\frac{C'}{\sigma'}\|z_1\|_{\eta}^1|t_1-t_0|.
\end{equation}
We continue to write $z_1-z_0$ as
\begin{equation}\label{eq4.56}
\begin{aligned}
z_1-z_0
=\partial_{t_0}\wtt(z_0,t_0)(t_1-t_0)+D\wtt(z_0,t_0)(z_1-z_0)+R_1+R_2,
\end{aligned}
\end{equation}
where
\begin{eqnarray*}
R_1&=&\wtt(z_1,t_1)-\wtt(z_1,t_0)-\partial_{t_0}\wtt(z_0,t_0)(t_1-t_0),\\
R_2&=&\wtt(z_1,t_0)-\wtt(z_0,t_0)-D\wtt(z_0,t_0)(z_1-z_0).
\end{eqnarray*}
By \eqref{eq4.49} and \eqref{eq4.51},
\[\begin{aligned}
\|R_1\|_{\eta}=&|t_1-t_0|\big\|\int_0^1\partial_{t_0}\wtt(z_1,pt_1+(1-p)t_0)
-\partial_{t_0}\wtt(z_0,t_0)dp\big\|_{\eta} = o(|t_1-t_0|).
\end{aligned}\]
Using \eqref{eq4.50}, we have $\|R_2\|_{\eta}=o(\|z_1-z_0\|_{\eta})=o(|t_1-t_0|)$.
Therefore, by \eqref{eq4.56},
\begin{equation}\label{eq4.57}
\partial_{t_0}z_0=\left(I-D\wt{\mathscr{T}}_{cu}
(z_0,t_0)\right)^{-1}\partial_{t_0}\wtt(z_0,t_0).\end{equation}
and we obtain $\partial_{t_0}h_s(\cdot,\cdot,\ep)\in
C^0(X_1^{cu}\times Y_1\times \mathbb{R},X^s)$.
\end{proof}

Since the cut-off function does not change the system in a neighborhood of radius
$\frac r3$ and $r$ is taken independent of $\ep$, we obtain Theorem \ref{thm4.3}.

From the exponential dichotomy, we can also construct the stable
integral manifold. For $\xi_s\in X_1^s$ and $(x,y)\in B_{\eta}^{+}(\infty)$, define
\begin{align}
\mathscr{T}_{s}(x,y,\xi_s,t_0,\ep)(t) =U(t,\ep)\xi_s&+\int_0^tU(t-\tau,\ep)
\begin{pmatrix}
P_{s}\big(F_1(x,y,\tau+t_0,\ep)+f_yy\big)\\0
\end{pmatrix}d\tau \notag\\
&+\int_{+\infty}^{t}U(t-\tau,\ep)
\begin{pmatrix}
P_{cu}\big(F_1(x,y,\tau+t_0,\ep)+f_yy\big)\\G_1(x,y,\tau+t_0,\ep)+g_xx
\end{pmatrix}d\tau, \label{eq4.13}
\end{align}
where $F_1,G_1$ are introduced in \eqref{eq4.1}. One may note that
$F_1$ and $G_1$ are not cut off in $\mathscr{T}_s$ as opposed in the construction
of the center-unstable integral manifold. This rather standard
practice is due to the fact $a_1<0$ and thus the local information near the
steady state along is sufficient to determine the unique stable integral
manifolds where the solutions decay exponentially. Following the same proof as in
Lemma \ref{lemma4.2}, one can prove that if \eqref{eq4.7} is satisfied,
$\mathscr{T}_s$ defines a contraction on $B_{\eta}^{+}(\rho)$, where $\rho$ is
sufficiently small but independent of $\ep<<1$, under the norm
$|\cdot|_{\eta,\eps}^{+}$ given in \eqref{eq4.6}. Therefore

\begin{theorem}\label{thm4.5}
Assume hypotheses (A1)--(A4), (B1)--(B5), and there exists $\eta$ such that $\eta,
k \eta \in (a_1, a_2)$ then
\begin{enumerate}
\item There exist $r>0$ and $\ep_0>0$ and a unique mapping $h_{cu}: B_r (0, X_1^s)
\times \R \times (0, \ep_0) \to B_r (0, X_1^{cu} \times Y_1)$ such that the family of
graphs $\CM_\ep^s (t_0)$ form the locally invariant stable integral manifold of $(0, 0)$
of \eqref{eq3.1}.
\item Solutions are on $\CM_\ep^s$ if and only if they decay with exponential rate
$\eta$ as $t \to +\infty$.
\item $h_{cu} (\xi_s, t_0, \ep)$ is $C^k$ in $\xi_s$ and continuous in $t_0$ with the
norms independent of $\ep$.
\item If we assume hypothesis (A5'), then
\[
\partial_{t_0}h_{cu}(\cdot,\cdot,\ep) \in C^0(B_r (0, X_1^s) \times\mathbb{R},\, X^{cu}
\times Y)
\]
with the norm independent of $\ep$.
\end{enumerate}
\end{theorem}

In the following, we will give the hypotheses on center-stable and unstable integral
manifold.
\begin{enumerate}
\item[(C1)] There exists a pair of continuous projections $(P_{cs},P_{u})$ on $X$, such
that $P_{cs}+P_u=I_X$ and $X^{cs, u} \triangleq P_{cs,u}X$ are positively invariant under
$e^{t\af}$.
\item[(C2)] There exist constants $a_2'>0$, and $a_1'<a_2'$,
\begin{equation*} \begin{aligned}
&& &|e^{t\af}P_{cs}x|_{X} + |e^{t(\frac{J}{\ep}+g_y)}y|_{Y} \leq Ke^{a_1't}(|x|_{X}
+ |y|_Y) \quad&&\mbox{for}&&\quad t\geq0,\quad x\in X, \; y \in Y \\\
&& &|e^{t\af}P_ux|_{X}\leq Ke^{a_2't}|x|_{X}\quad&&\mbox{for}&&\quad t\leq0,\quad x\in X.
\end{aligned}\end{equation*}
\end{enumerate}
As in Remark \ref{rem4.1}, $e^{t\af}$ and $e^{t(\frac{J}{\ep}+g_y)}$ satisfy the same
estimates with norms replaced by $|\cdot|_{X_1}$ and $|\cdot|_{Y_1}$. Let $P_{cs}X_1=X_1^{cs}$
and $P_uX_1=X_1^u$. By the same proof,

\begin{theorem}\label{thm4.6}
Assume (A1)--(A4), (B1)--(B3), (C1)--(C2), and there exists $\eta$ such that $\eta,
k \eta \in (a_1', a_2')$ then
\begin{enumerate}
\item There exist $r>0$ and $\ep_0>0$ and mappings
\[\begin{aligned}
&& h_{cs}: B_r (0, X_1^u) \times \R \times (0, \ep_0) \to B_r (0, X_1^{cs} \times Y_1) \\
&& h_u: B_r (0, X_1^{cs} \times Y_1) \times \R \times (0, \ep_0) \to B_r (0, X_1^u)
\end{aligned}\]
such that the two families of graphs $\CM_\ep^u (t_0)$ and $\CM_\ep^{cs} (t_0)$ form locally invariant
unstable and center-stable integral manifolds of $(0, 0)$ of \eqref{eq3.1}.
\item A backward flow is well-defined on $\CM_\ep^u$ and solutions are on $\CM_\ep^u$ if and only if they
decay with exponential rate $\eta$ as $t \to -\infty$.
\item $h_{cs} (\xi_u, t_0, \ep)$ is $C^k$ in $\xi_u$ and $h_u (\xi_s, \xi_y, t_0, \ep)$ is
$C^k$ in $\xi_s$ and $\xi_y$ and both are continuous in $t_0$ with the norms independent of $\ep$.
\item If we assume hypothesis (A5'), then
\[
\partial_{t_0}h_{cs}(\cdot,\ep) \in C^0(B_r (0, X_1^u) \times\mathbb{R},\, X^{cs} \times Y)\quad
\p_{t_0}h_u(\cdot,\ep) \in C^0(B_r (0, X_1^s \times Y_1) \times\mathbb{R},\, X^u)
\]
with the norms independent of $\ep$.
\end{enumerate}
\end{theorem}

Let $X_1^c=X_1^{cu}\cap X_1^{cs}=(I_X-P_s-P_u)X_1\triangleq P_cX_1$. By taking the intersection of $\mathcal{M}_{\ep}^{cu}$ and $\mathcal{M}_{\ep}^{cs}$, we can obtain a center manifold in the standard
way.

\begin{theorem}\label{thm4.7}
Assume (A1)--(A4), (B1)--(B5), (C1)--(C2), and there exists $\eta_\pm$ such that $\eta_+,
k \eta_+ \in (a_1', a_2')$ and $\eta_-, k\eta_- \in (a_1, a_2)$ then
\begin{enumerate}
\item There exist $r>0$ and $\ep_0>0$ and mappings
\[
\Psi = (\Psi_s, \Psi_u) : B_r (0, X_1^c \times Y_1) \times \R \times (0, \ep_0) \to
B_r (0, X_1^s \times X_1^u)
\]
such that the family of graphs $\CM_\ep^c (t_0)$ form a locally invariant center integral manifolds
of $(0, 0)$ of \eqref{eq3.1}.
\item A backward flow is well-defined on $\CM_\ep^c$.
\item $\Psi (\xi_c, \xi_y, t_0, \ep)$ is $C^k$ in $\xi_c$ and $\xi_y$ and is continuous in $t_0$
with the norms independent of $\ep$.
\item If we assume hypothesis (A5'), then
\[
\partial_{t_0} \Psi(\cdot,\ep) \in C^0(B_r (0, X_1^c \times Y_1) \times\mathbb{R},\, X^s
\times X_1^u)
\]
with the norms independent of $\ep$.
\end{enumerate}
\end{theorem}

\subsection{Asymptotic estimates of Invariant Manifolds} \label{SS:asym}

In Theorems \ref{thm3.1} and \ref{thm3.2}, we have demonstrated that
\eqref{eq3.2} can be viewed as the singular limit of \eqref{eq3.1}
as $\ep\rightarrow0$. Therefore, one may expect the perturbed
invariant manifolds should be close to the unperturbed ones. In this subsection, we will give the
leading order approximation of the invariant manifolds.

In this subsection, we will use the notation
\begin{equation} \label{E:G_*}
F_*^\ep (x, t) = F(x, 0, t, \ep) \qquad G_*^\ep (x, t) = G(x, 0, t, \ep).
\end{equation}
Instead of \eqref{eq3.1} directly, we first consider compare \eqref{eq4.5} with the
following regular perturbation problem
\begin{equation} \label{E:x_*}
\dot{x}_*=\af x_*+F_*^\ep(x_*, t +t_0)
\end{equation}
where we also included the dependence of $F$ on $\ep$. Under assumptions in
Theorem \ref{thm4.3}, $x_*\equiv 0$ is an steady solution and it has local integral
manifolds. In fact, for $\xi_{cu}\in X_1^{cu}$ and $x\in C_{\eta}^{-}(X_1)$, let
\begin{equation} \label{E:wte}\begin{aligned}
\wte(x)(t)=&e^{t\af}\xi_{cu} +\int_{0}^{t}e^{(t-\tau)\af}P_{cu} F_*^\ep(x(\tau),\tau+t_0) d\tau\\
&+\int_{-\infty}^{t}e^{(t-\tau)\af}P_sF_*^\ep(x(\tau),\tau+t_0)d\tau.
\end{aligned} \end{equation}
By the exponential dichotomy (B5), $\wte$ is a contraction on $C_{\eta}^{-}(X_1)$.
For $\xi_{cu}\in X_1^{cu}$, let $x_*(t)$ be the fixed point of $\wte$. Define
\begin{eqnarray*}
&&h_{s}^*(\xi_{cu}, t_0, \ep)=P_s x_*(0)=\int_{-\infty}^0 e^{-\tau \af}P_s F_*^\ep(x_* (\tau), \tau+t_0)
d\tau \\
&&\mathcal{M}_*^{cu}(t_0)=\Big\{\xi_{cu}+h_{s}^*(\xi_{cu}, t_0, \ep)\big|\xi_{cu}\in
X_1^{cu}\Big\}
\end{eqnarray*}
which is the center-unstable manifold of \eqref{E:x_*}.

\begin{prop} \label{P:appro}
Assume the hypotheses in Theorem \ref{thm4.3}, (A5) and (A5'), then
\[\begin{split}
\big|h_{s}(\xi_{cu},0,t_0,\ep)-h_{s}^*(\xi_{cu}, t_0, \ep)&\big|_{X_1}\leq
C'\ep \Big(|g_x|_{L(X, Y)} \\
&+ |D_x G_*^\ep|_{C^0(X_1 \times \R, L(X, Y))} + |\p_t G_*^\ep|_{C^0_t C_x^1 (X_1 \times \R, Y)}\Big)
\end{split}\]
where $h_s$ is defined in \eqref{eq4.11}. If $k\geq 2$, we also have
\[\begin{aligned}
& \big|D_{\xi_y}h_s(\xi_{cu},0,t_0,\ep)\big|_{L(Y_1,X^s)} \leq C' \ep \\
&\big|D_{\xi_{cu}}h_{s}(\xi_{cu},0,t_0,\ep)
-D_{\xi_{cu}}h_{s}^*(\xi_{cu}, t_0, \ep)\big|_{L(X_1^{cu},X_1^s)}\\
\leq & C'\ep \big(|g_x|_{L(X, Y)}
+ |D_x G_*^\ep|_{C_t^0 C_x^1 (X_1 \times \R, L(X, Y))} + |\p_t G_*^\ep|_{C_t^0 C_x^1(X_1 \times \R, Y)}\big)
\end{aligned}\]
where $C'$ depends on $C_0, K,a_1,a_2$, and $|\xi_{cu}|_{X_1}$ and the norm
$C_t^0 C_x^1$ means $C^0$ in $t$ and $C^1$ in $x \in X_1$.
\end{prop}

The reason the estimate on $D_{\xi_y}h_s$ is only in $Y$ is similar to that for Lemma \ref{L:FTL}.

\begin{rem} \label{R:other}
The exactly same estimates in this subsection also hold for the center, center-stable,
stable, and unstable manifolds except for the latter two, there is no $D_{\xi_y}$ involved.
\end{rem}

\begin{rem} \label{R:appro}
The reason we include some these terms in the above upper bound goes back to transformation
\eqref{E:y_1}. Conceptually, with certain (mild) assumptions on $A$ and $J$, one may carry out a sequence of
transformations in the form of \eqref{E:y_1} to make $g_x =0$ and $G(x, 0, t, 0) = O(\ep^{k-1})$ when
measured in appropriate norms. Therefore the above estimates immediately implies that the integral
manifolds of \eqref{E:x_*} are approximations of those of \eqref{eq4.5} at $\{y=0\}$ with an error of
$O(\ep^k)$. Since \eqref{E:x_*} is a regular perturbation problems of \eqref{eq3.2}, one may compute the
Taylor expansions of the integral manifolds of \eqref{E:x_*} up to the order $O(\ep^{k-1})$ and thus they
also serve as the leading order expansions of the integral manifolds of \eqref{eq4.5} at $\{y=0\}$.
\end{rem}

\begin{proof}
We will denote $P_X$ and $P_Y$ the projection to $X$ and $Y$. Let $(x,y)$ be the fixed point of
$\wtt$ with parameters $\xi_{cu}$ and $\xi_y=0$ and $x_*$ be the fixed point of $\wte$ with the
parameter $\xi_{cu}$. In the rest of the proof, we will use $\wtt(x,y,\ep)$ to denote $\wtt(z,t_0)$,
which is introduced in \eqref{eq4.41}. From \eqref{eq4.43},
\begin{eqnarray*}
\|(x-x_*,y)\|_{\eta}^1&\leq&\|\wtt(x,y)-\wtt(x_*,0)\|_{\eta}^1+\|\wtt(x_*,0)-\wte(x_*)\|_{\eta}^1\\
&\leq&(1-\sigma')\|(x-x_*,y)\|_{\eta}^1+\|\wtt(x_*,0)-\wte(x_0)\|_{\eta}^1,
\end{eqnarray*}
which implies
\[
\|(x-x_*,y)\|_{\eta}^1\leq\frac{1}{\sigma'}\|\wtt(x_*,0)-\wte(x_*)\|_{\eta}^1.
\]
From the definitions of $\wtt$ and $\wte$, one may compute by integrating by parts
\begin{eqnarray*}
&&\big(\wtt(x_*,0)-\wte(x_*) \big) (t) = P_Y \wtt(x_*,0) (t)\\
&=& \ep (J + \ep g_y)^{-1} \Big( \int_0^t e^{(t-\tau)(\frac{J}{\ep}+g_y)} \Big( \p_t G_*^\ep
+ \big(g_x + D_x G_*^\ep \big) \big(\af x_*+F_*^\ep \big)\Big) d\tau \\
&& + e^{t(\frac{J}{\ep}+g_y)} \big(g_x x_*(0) + G_*^\ep(x_*(0), t_0)\big) - \big(g_x
x_*(t) + G_*^\ep(x_*(t), t_0 +t) \big) \Big)
\end{eqnarray*}
where $\p_t G_*^\ep$, $D_x G_*^\ep$, and $F_*^\ep$ in the above integral are evaluated at $(x_*(\tau),
t_0 + \tau)$. Using assumption (B2), we immediately obtain
\begin{equation}\label{eq4.61} \begin{split}
\|(x-x_*,y)\|_{\eta}^1\leq& C'\ep (|g_x|_{L(X, Y)} + |D_x G_*^\ep|_{C^0(X_1 \times \R, L(X, Y))}
+ |\p_t G_*^\ep|_{C_t^0C_x^1(X_1 \times \R, Y)})
\end{split}\end{equation}
where we need $|\p_t G_*^\ep|_{C_t^0C_x^1}$ to bound $\p_t G_*^\ep$ by $|\p_t D G_*^\ep||x|$ and $|x|$ provides
the necessary decay in $t$. Consequently the estimate on $h_s - h_s^*$ follows.

To prove the second part, choose $\eta$ such that $a_1<\eta,2\eta<s_2$. Let
$(\phi^{\ep}(t),\psi^{\ep}(t))$ be the derivative of $(x(t),y(t))$ with respect to
$\xi_{cu}$ at $(\xi_{cu},0)$ and $\phi^*(t)$ be the derivative of $x_*(t)$ with
respect to $\xi_{cu}$, so we have
\begin{equation}\label{eq4.62}
(\phi^{\ep},\psi^{\ep})=D\wtt(x,y)(\phi^{\ep},\psi^{\ep})+e^{t\af}\quad,
\quad\phi^*=D\wte(x_*)(\phi^*)+e^{t\af}.
\end{equation}
As in Theorem \ref{thm4.3}, It is easy to show $\|\phi^*\|_{L(X_1^{cu},C_{i\eta}^{-}
(X_1))}$ is bounded uniformly in $\ep$ for $i=1,2$. By \eqref{eq4.62},
\begin{eqnarray}\label{eq4.63}
\begin{aligned}
\Big|(\phi^{\ep}-\phi^*,\psi^{\ep})\Big|_{L(X_1^{cu},\widetilde{B}_{2\eta}^{-}(\infty))}
\leq &\Big|D\wtt(x,y)(\phi^{\ep}-\phi^*,\psi^{\ep})\Big|_{L(X_1^{cu},\widetilde{B}_{2\eta}^{-}(\infty))}\\
&+\Big|(D\wtt(x,y)-D\wtt(x_*,0))(\phi^*,0)\Big|_{L(X_1^{cu},\widetilde{B}_{2\eta}^{-}(\infty))}\\
&+\Big|D\wtt(x_*,0)(\phi^*,0)-D\wte(x_*)\phi^*\Big|_{L(X_1^{cu},\widetilde{B}_{2\eta}^{-}(\infty))}.
\end{aligned}
\end{eqnarray}
In the first term on the right side, $D\wtt(x,y)$ is bounded by $1-\sigma'$ according to \eqref{eq4.43}.
Using \eqref{eq4.61} and the fact that $F$ and $G$ are $C^2$, we obtain through straight forward
computation
\begin{equation} \label{eq4.64} \begin{split}
&\Big|\big(D\wtt(x,y)-D\wtt(x_*,0)\big)(\phi^*,0)\Big|_{L(X_1^{cu},\widetilde{B}_{2\eta}^{-}(\infty))}\\
\leq& C'\ep \Big(|g_x|_{L(X, Y)} + |D_x G_*^\ep|_{C^0(X_1 \times \R, L(X, Y))} + |\p_t G_*^\ep|_{C_t^0 C_x^1
(X_1 \times \R, Y)}\Big).
\end{split}\end{equation}
For the last term in \eqref{eq4.63}, one may calculate
\[\begin{split}
& D\wtt(x_*,0)(\phi^*,0)-D\wte(x_*)\phi^* = P_Y D\wtt(x_*,0)(\phi^*,0)\\
= &\ep (J + \ep g_y)^{-1} \Big( \int_0^t e^{(t-\tau)(\frac{J}{\ep}+g_y)} \big( D_x \p_t G_*^\ep + D_x^2 G_*^\ep
(\af x_*+F_*^\ep) \\
& + (g_x + D_x G_*^\ep ) (\af +D_x F_*^\ep)\big) \phi_* d\tau \\
& + e^{t(\frac{J}{\ep}+g_y)} \big(g_x + D_x G_*^\ep(x_*(0), t_0)\big) \phi_* (0) - \big(g_x
+ D_x G_*^\ep(x_*(t), t_0 +t) \big) \phi_*(t) \Big)
\end{split}\]
where $F_*^\ep$, $D_x \p_t G_*^\ep$, $D_x^2 G_*^\ep$, and $D_x F_*^\ep$ in the above integral are evaluated
at $(x_*(\tau), t_0 + \tau)$. Along with \eqref{eq4.43}, \eqref{eq4.63}, and \eqref{eq4.64}, it implies
\[
\begin{aligned}
\Big|(\phi^{\ep}-\phi^*,\psi^{\ep})\Big|_{L(X_1^{cu},\widetilde{B}_{2\eta}^{-}(\infty))}
\leq & C'\ep \big(|g_x|_{L(X, Y)} \\
&+ |D_x G_*^\ep|_{C_t^0 C_x^1(X_1 \times \R, L(X, Y))} + |\p_t G_*^\ep|_{C_t^0 C_x^1 (X_1 \times \R, Y)}\big)
\end{aligned}\]
and thus the estimates on $D_{\xi_{cu}} h_s - D_{\xi_{cu}} h_s^*$.

Finally, with slight abuse of notation, we still use $(\phi^{\ep},\psi^{\ep})$ to denote
the derivative of $(x,y)$ with respect to $\xi_y$ at $\xi_y=0$. Using \eqref{eq4.43} and
\eqref{eq4.61}, it is straight forward to show that, at $\xi_y =0$,
\begin{equation} \label{E:y=0}
|y|_{C_{\eta'}^- (Y_1)} \le C'\ep \qquad  |(\phi^{\ep},\psi^{\ep})|_{L(Y_1,
\tilde B_{\eta'}^- (\infty))} \le C'
\end{equation}
where $\eta'$ can be taken in a compact subinterval of $(a_1, a_2)$. Clearly
\[
\phi^{\ep}=P_X \big(D_x \wtt (x, y) \phi^\ep + D_y \wtt(x, y)
\psi^\ep\big).
\]
Like $\wte$, one can show $P_X D_x \wtt (x, y)$ is a linear contraction on $C_\eta^- (X)$. Therefore
\begin{equation}\label{eq4.67}
|\phi^{\ep}|_{L(Y_1, C_{2\eta}^- (X))} \leq C' |P_X D_y \wte(x, y) \psi^\ep|_{L(Y_1, C_{2\eta}^-
(X))}.
\end{equation}
To estimate the right side, we notice that $\psi^\ep$ satisfies
\[
\psi_t^{\ep}(t)=\big(\frac{J}{\ep}+g_y+D_yG(x,y,t+t_0,\ep)\big)\psi^{\ep}(t)
+\big(D_xG(x,y,t+t_0,\ep)+g_x\big)\phi^{\ep}(t)
\]
which can be rewritten as
\[
\psi^{\ep}=\ep J^{-1}\psi_t^{\ep}-\ep J^{-1}\big((g_y+D_yG)\psi^{\ep}+(D_xG+g_x)\phi^{\ep}\big).
\]
Substitute this identity into \eqref{eq4.67} and use \eqref{E:y=0}, we obtain
\[\begin{split}
&|\phi^{\ep}|_{L(Y_1, C_{2\eta}^- (X))} \leq C' \ep\big(1 + |P_X D_y \wte(x, y) J^{-1}
\psi_t^{\ep}|_{L(Y_1, C_{2\eta}^- (X))}\big)\\
\leq & C' \ep \Big(1 + \sup_{t\leq0}e^{-2\eta t}\Big| \big(\int_{0}^{t} P_{cu} + \int_{-\infty}^t
P_s \Big) e^{(t-\tau)\af} \big(D_yF(x,y,\tau+t_0,\ep)+f_y\big) \\
&\qquad J^{-1} \psi_t^{\ep} d\tau\Big|_{L(Y_1,X^{cu})} \Big)
\end{split}\]
Integrating by parts and using \eqref{E:y=0} to control $\dot y$, we obtain the
desired estimates. The estimates can not be improved to the norm $X_1^s$ as $A$ is
produced in the integration by parts.
\end{proof}

To consider invariant manifolds in larger ranges, let $\Phi(T,t_0,z,\ep)$ and $\Phi^*(T,t_0,x, \ep)$ be
solutions of \eqref{eq4.5} and \eqref{E:x_*} from time $0$ to $T-t_0$ (so from time $t_0$ to $T$ for \eqref{eq3.1}
and \eqref{E:x_*0}) with $\Phi(t_0,t_0,z,\ep)=z=(x, y)$ and $\Phi^*(t_0,t_0,x, \ep)x$. We will skip writing $t_0$ and $\ep$ in $\Phi$ and $\Phi^*$ in the next proposition
as it would not be altered. Combining Lemma \ref{L:FT}, Lemma \ref{L:FTL}, Remark \ref{rem3.4}, \ref{rem3.6}, and
Proposition \ref{P:appro}, we obtain in a straight forward manner

\begin{prop}\label{thm4.20}
If the hypotheses in Theorem \ref{thm4.3}, (A5) and (A5') hold for $k=2$, then
there exists $C'$ which depends on
$C,K,\eta,a_1,a_2,\ovr,\eps,|T-t_0|,|\xi_{cu}|_{X_1}$ such that
\begin{eqnarray*}
&&\Big|\Phi(T, \xi_{cu}+h_{s}(\xi_{cu},0))-\Phi^*(T,\xi_{cu}+h_{s}^*(\xi_{cu}))
\Big|_{X_1\times Y_1}\\
\leq && C'\ep \Big(|g_x|_{L(X, Y)} + |D_x G_*^\ep|_{C^0(X_1 \times \R, L(X, Y))} + |\p_t G_*^\ep|_{C^0(X_1
\times \R, Y)} \Big);\\
&&\Big|D_{\xi_{cu}}\left(\Phi(T,\xi_{cu}+h_{s}(\xi_{cu},0))\right)
-D_{\xi_{cu}}\left(\Phi^*(T,\xi_{cu}+h_{s}^*(\xi_{cu}))\right)\Big|_{L(X_1^{cu},X_1\times
Y_1)} \\
&&+ \Big|P_Y\big(D_{\xi_{y}}\left(\Phi(T,\xi_{cu}+\xi_y+ h_{s}(\xi_{cu},\xi_y))\right)
-E(T,\xi_{cu}+h_s^*(\xi_{cu}))\big)\Big|_{\xi_y=0}\Big|_{L(Y_1,Y_1)}\\
\leq && C'\ep \big(|g_x|_{L(X, Y)} + |D_x G_*^\ep|_{C_t^0 C_x^1 (X_1 \times \R, L(X, Y))}
+ |\p_t G_*^\ep|_{C_t^0 C_x^1(X_1 \times \R, Y)}\big)\\
&&\Big|P_X\big(D_{\xi_{y}}\left(\Phi(T,\xi_{cu}+\xi_y+h_{s}(\xi_{cu},\xi_y))\right)\big)
\Big|_{\xi_y=0}\Big|_{L(Y_1,X)} \leq C'\ep
\end{eqnarray*}
where $E(T, t_0, \xi_{cu}+h_s^*(\xi_{cu}),\ep)$ is the evolution operator generated by $\frac{J}{\ep}+D_yg_*^\ep
(\Phi^*(t,\xi_{cu}+h_{s}^*(\xi_{cu})), t)$ with initial time $t_0$ and terminal time $T$ and $P_X,P_Y$ denote
the projection from $X\times Y$ to $X$ and $Y$, respectively.
\end{prop}

See Remark \ref{R:appro} for the explanation why the above upper bounds are taken in such a tedious form.
Also, the reason $P_X D_{\xi_{y}}\Phi$ is estimated only in $X$ is Lemma \ref{L:FTL}.

Since the system is autonomous when $\ep=0$, we also expect the
derivatives of the integral manifolds in $t_0$ is of order $O(\ep)$.
This will be used in studying homoclinic orbits.

\begin{prop}\label{thm4.26}
Assume the same condition as in Proposition \ref{thm4.14} for $k=2$ and
$|\xi_y|_{Y_1}\leq C_1\ep$, then
\[
|\partial_{t_0}h_s(\xi_{cu},\xi_y,\cdot,\ep)|_{C^0(\mathbb{R},X^s)}\leq C'\ep,
\]
where $C'$ depends on $C_1,|\xi_{cu}|_{X_1}$ and constants in
assumptions.
\end{prop}

\begin{rem} \label{R:p_th_s2}
If in (A5) and (A5') we assume the smoothness of $(f, g) : X_1 \times Y_1 \times \R^2 \to X_1 \times Y_1$, the same
proof implies $\p_{t_0} h_s \in X_1^s$ is of $O(\ep)$. See also Remark \ref{R:p_th_s1}.
\end{rem}

\begin{proof}
Let $z_0=(x_0,y_0)$ be the fixed point of $\wtt(\cdot,t_0)$ and
$(\phi,\psi)=(\partial_{t_0}x_0,\partial_{t_0}y_0)$. From \eqref{eq4.57}, we only need to estimate $\p_{t_0}
\wtt(z_0,t_0)$. Notice, from assumptions (B1) and (B2),
\[
F_t(z, t, \ep) = \ep \int_0^1 F_{t\ep} (z, t, \tau_1 \ep) d\tau_1 = \ep \int_0^1 \int_0^1 DF_{t\ep} (\tau_2 z,
t, \tau_1) z d\tau_2 d\tau_1
\]
and similar estimate holds for $G$. These estimates immediately implies
\[
\|\p_{t_0} \wtt(z_0,t_0)\|_{2\eta} \leq C'\ep \|z_0\|_{2\eta}^1
\]
and the thus the proposition follows.
\end{proof}

\section{Invariant Foliation} \label{S:InFo}

With the $C^k$ center-stable (center-unstable) integral manifolds constructed in the previous section,
we will give the sketch of the construction of the stable (unstable) fibres inside the center-stable
(center-unstable) manifold under the same assumptions in this section. We will use the stable fibres as an
illustration and similar construction also works for unstable fibers.

For $\xi_{cy}=(\xi_c,\xi_y)\in X_1^c\times Y_1$, let $\left(x(\xi_{cy})(t),y(\xi_{cy})(t)\right)$ be the solution
of \eqref{eq4.5} with the initial value (at $t=0$) on the center manifold
\begin{equation}\label{eq5.1}
\xi=\xi_{cy}+\Psi_s(\xi_{cy},t_0,\ep)+\Psi_u(\xi_{cy},t_0,\ep).\end{equation}
The solution stays on $\mathcal{M}_{\ep}^c (t)$ and satisfies
\begin{equation}\label{eq5.2}
\begin{pmatrix}
  x(t)\\
  y(t)
\end{pmatrix}
=U(t,\ep)\xi+\int_0^tU(t-\tau,\ep)\begin{pmatrix}
  F(x,y,\tau+t_0,\ep)+f_yy\\
  G(x,y,\tau+t_0,\ep)+g_xx
\end{pmatrix}d\tau.
\end{equation}
To simplify our notation, for $(\wt{x},\wt{y})\in X_1\times Y_1$, we write
\begin{eqnarray*}
\wt{F}(\wt{x},\wt{y},\xi_{cy},t,\ep)&=&F(x(\xi_{cy})(t)+\wt{x},y(\xi_{cy})(t)+\wt{y},t +t_0,\ep)\\
&&-F(x(\xi_{cy})(t),y(\xi_{cy})(t),t +t_0,\ep)
\end{eqnarray*}
or very often in short as $\wt{F}(\wt{x},\wt{y},\xi_{cy})$. Such notation also applies to $G$.

For each triple $(\xi_s,\xi_c,\xi_y)\in X_1^s\times X_1^c\times Y_1$ and $a_1<\eta<a_2$, it is the standard
knowledge that $(x(t), y(t))$ is a solution of
\eqref{eq4.5} satisfying
\[
P_s \big(x(0) - x(\xi_{cu})(0)\big)=\xi_s \quad \text{ and } \quad \big(\wt{x},\wt{y} \big) \triangleq
\big(x, y\big) - \big( (x(\xi_{cy}),y(\xi_{cy}) \big) \in B_{\eta}^{+}(\infty)
\]
where $B_{\eta}^{+}(\infty)$ was defined in \eqref{eq4.6}, if and only if $\big(\wt{x}(\cdot),\wt{y}(\cdot)\big)$
is a fixed point of
\begin{equation}\label{eq5.4}
\begin{aligned}
\mathscr{G}_s(\xi_s,\xi_{cy},t_0,\ep)(\wt{x},\wt{y})(t)\triangleq &U(t,\ep)\xi_s+\int_0^tU(t-\tau,\ep)
\begin{pmatrix}
P_{s}\big(\wt{F}(\wt{x},\wt{y},\xi_{cy})+f_y\wt{y}\big)\\
0\end{pmatrix}d\tau \\
&+\int_{+\infty}^tU(t-\tau,\ep)\begin{pmatrix}
P_{cu}\big(\wt{F}(\wt{x},\wt{y},\xi_{cy})+f_y\wt{y}\big)\\
\wt{G}(\wt{x},\wt{y},\xi_{cy})+g_x\wt{x}
\end{pmatrix}d\tau.
\end{aligned}
\end{equation}

One first notices that, for fixed $\xi_{cy}$, $\mathscr{G}_s$ has the same form as $\mathscr{T}_s$ with only an
additional parameter $\xi_{cy}$. Moreover, by \eqref{eq4.3}
\[
\wt{F}(0,0,\xi_{cy},\ep)=\wt{G}(0,0,\xi_{cy},\ep)=0 \quad \;
|D\wt{F}|_{C^0}=|DF|_{C^0}\leq\ovr\quad \;
|D\wt{G}|_{C^0}=|DG|_{C^0}\leq\ovr
\]
where $D$ is the differentiation with respect to $(\wt{x},\wt{y})$ or $(x,y)$. Through exactly the same procedure
as in Section \ref{S:InMa}, we obtain that $\mathscr{G}_s$ defines a contraction on $B_{\eta}^{+}(\infty)$ under
the norm $|\cdot|_{\eta,\eps}^{+}$ defined in \eqref{eq4.6}. Clearly, if $\xi_s=0$, $(\wt{x},\wt{y})=(0,0)$ is
the unique fixed point of \eqref{eq5.4}. Moreover, in the study of the $C^k$ smoothness of the fixed point $(\wt{x},\wt{y})$ with respect to $\xi_s$, the linear terms $f_y$ and $g_x$, which are not small, disappear again. The evolution operator $U$ has a uniform bound though it depends on $\ep$. Therefore the exactly standard argument in
regular perturbations \cite{CLL91} (where no differentiation in $t$ is need which would product $\frac 1\ep$)
applies and yield the smoothness of $(\wt{x},\wt{y})\in B_{\eta}^{+}(\infty)$ in $\xi_s$. Therefore, by fixing $\eps$ small and then choosing $\ovr$ and $\ep$ sufficiently small accordingly, we obtain the following theorem

\begin{theorem}\label{thm5.2}
Assume (A1)-(A4), (B1) -- (B5), (C1), and (C2). If there exists $\eta<0$ with $a_1<k\eta<\eta<a_2$, then
for each triple $(\xi_s,\xi_c,\xi_y)\in X_1^s\times X_1^c\times Y_1$, \eqref{eq5.4}
has a unique fixed point $(\wt{x},\wt{y})\in B_{\eta_0}^{+}(\infty)$ such that
\begin{enumerate}
\item[i)] If $\xi_s=0$, $(\wt{x},\wt{y})\equiv(0,0)$.
\item[ii)] $(D_{\xi_s}^j\wt{x},D_{\xi_s}^j\wt{y})\in
C^0\left(X_1^s\times X_1^c\times
Y_1,B_{j\eta_0}^{+}(\infty)\right)$, where $j=1,\cdots,k$.
\end{enumerate}
\end{theorem}

Let $(\wt{x}, \wt{y})$ be the fixed point of \eqref{eq5.4} corresponding to $\xi_s$ and $\xi_{cy}$, we define
\begin{equation}\label{eq5.5} \begin{aligned}
\sigma_{cu}(\xi_s,\xi_{cy},t_0)=\xi+\left(\wt{x}(0),\wt{y}(0)\right) \quad \;
\mathcal{W}_\ep^{s}(\xi_{cy},t_0)=\Big\{&\sigma_{cu}(\xi_s,\xi_{cy},t_0)\Big| \xi_s\in
X_1^s\Big\}
\end{aligned} \end{equation}
where $\xi$ is given in \eqref{eq5.1}. Usually $\mathcal{W}_\ep^{s}(\xi_{cy},t_0)$ is called the stable fiber
based at $\xi$.

\begin{rem}\label{rem5.3}
Clearly, $(\wt{x}+x,\wt{y}+y)(0) \in \mathcal{M}_{\ep}^{cs}(t_0)$, where $(x,y)$ is the solution of \eqref{eq5.2}
with parameters $(\xi_c,\xi_y,t_0)$ and thus $\mathcal{W}_\ep^{s}(\xi_{cy},t_0) \subset \mathcal{M}_{\ep}^{cs}(t)$
\end{rem}

To study the smooth of $\mathcal{W}_\ep^{s}(\xi_{cy},t_0)$ with respect to $\xi_{cy}$, for a positive
integer $k\geq2$, define
\begin{align}\nonumber
\Lambda_k=\Big\{(\eta,\eta')\in\mathbb{R}^2\big|&a_1<k\eta<\eta<\min\{0,a_2\},
a_1'<j\eta'<a_2',\\\label{eq5.8}&a_1<\eta+j\eta'<a_2,\
j=1,2,\cdots,k-1,\Big\}.
\end{align}

\begin{theorem}\label{thm5.3}
For $k\geq2$, assume (A1)-(A4), (B1)--(B5), (C1)--(C2), and $\Lambda_k$ is nonempty. For any compact subset
$\Sigma$ of $\Lambda_k$, by fixing $\eps$ small and then choosing $\ovr$ and $\ep$ sufficiently small
accordingly, then for any $(\eta,\eta')\in\Sigma$, \eqref{eq5.4} has a unique fixed point $(\wt{x},\wt{y})\in
B_{\eta}^{+}(\infty)$ such that
\begin{enumerate}
\item[i)] $(D_{\xi_{cy}}^j\wt{x},D_{\xi_{cy}}^j\wt{y})\in
C^0\left(X_1^s\times X_1^c\times
Y_1,B_{\eta+j\eta'}^{+}(\infty)\right),$
\item[ii)] $(D_{\xi_s}^{m-j}D_{\xi_{cy}}^j\wt{x},D_{\xi_s}^{m-j}D_{\xi_{cy}}^j\wt{y})
\in C^0\left(X_1^s\times X_1^c\times
Y_1,B_{(m-j)\eta+j\eta'}^{+}(\infty)\right),$ \end{enumerate} where
$m=2,\cdots,k,\ j=1,\cdots,m-1$.
\end{theorem}

Our spectral gap assumption on $\Lambda_k$ is essentially the same as in \cite{CLL91} and the proof of the
theorem again follows from the same procedure which is based on the definition of the Frechet derivatives.
One may notice that $\sigma_{cu}$ was only proved to belong to $C_{\xi_s}^{k-1-j} C_{\xi_c}^j$, while we have
it in $C_{\xi_s}^{k-j} C_{\xi_{cy}}^j$ if $j<k$. In fact, it is easy to verify that the same proof works to
yields our above slightly stronger version. For details see \cite{Lu10}.

Finally, a natural issue is the asymptotic estimates of the stable fibers as $\ep \to0$. As in Section
\ref{S:InMa}, we use equation \eqref{E:x_*} as the approximation of \eqref{eq4.5} in the slow direction and
keep the same notations as in \eqref{E:G_*}. Given any $\xi_c\in X_1^c$, let $\xi_*=\xi_c +(\Psi_s^*+\Psi_u^*)
(\xi_c)\in\mathcal{M}_*^c$, where we recall that $\CM_*^c$ is the center manifold of \eqref{E:x_*},
$\Psi_s^*,\Psi_u^*$ are independent of $t_0$, and $x_*(\xi_c)(t)$ be the solution on $\mathcal{M}_{*}^c$ such
that $x_*(\xi_c)(0)=\xi_*$. Let $\wt{x}_*(t)$ satisfy
\[
\begin{aligned}
\wt{x}_*(t)=e^{t\af}\xi_s+\big(\int_0^te^{(t-\tau)\af}
P_{s}&+\int_{+\infty}^te^{(t-\tau)\af}
P_{cu}\big) \wt{F}_*^\ep (\wt{x}_*(\tau), \xi_c, \tau) d\tau.
\end{aligned}
\]
where
\[
\wt{F}_*^\ep (\wt{x}, \xi_c, t) \triangleq F_*^\ep(\wt{x} +x_*(\xi_c),t+t_0)-F_*^\ep(x_*(\xi_c),t+t_0).
\]
Therefore, $(\wt{x}_*+x_*(\xi_c))(t)$ is the solution of the unperturbed fibre starting
at the based point $\xi_*$ with height $\xi_s$ such that
\[
(\wt{x}_*+x_*)(0)=\xi_s+(I-P_u)\xi_*+h_u^*(\xi_s + (I-P_u)\xi_*,\ep) \triangleq \sigma_{cu}^*(\xi_s,\xi_c, \ep)
\]
where $h_u^*: X_1^c \times X_1^s \times \R \to X_1^u$ is the defining function of the center-stable manifold
of \eqref{E:x_*}.

\begin{theorem}\label{thm5.5}
For $k=2$, assume (A1)-(A5), (B1)-(B5), (C1)-(C2), and $\Lambda_2$
in \eqref{eq5.8} is nonempty. For $\xi_y=0$ and the above given
$\xi_c$, we have
$$
\big|\sigma_{cu}(\xi_s,\xi_c)-\sigma_{cu}^*(\xi_s,\xi_c)\big|_{X_1\times Y_1}\leq C'\ep,
$$
where $C'$ depends on $K,a_1,a_1',a_2',\eta,\eta',\ovr,\eps,|\xi_c|_{X_1}$, and $|\xi_s|_{X_1}$.
Moreover, if $k\ge 3$ and $|\p_t G|_{C_{t, \ep}^0 C_{x, y}^2 (X_1\times Y_1 \times \R^2, Y)} \le C_0$, then
\[\begin{split}
&\big|\sigma_{cu}(\xi_s,\xi_c)-\sigma_{cu}^*(\xi_s,\xi_c)\big|_{X_1\times Y_1}\\
\leq& C'\ep \big(|g_x|_{L(X, Y)}
+ |D_x G_*^\ep|_{C^0 (X_1 \times \R, L(X, Y))} + |\p_t G_*^\ep|_{C_t^0 C_x^1(X_1 \times \R, Y)}\big)
\end{split}\]
\end{theorem}

We will write $\xi_c$ instead of $\xi_{cy}$ as we only consider the case $\xi_y=0$ in this theorem.

\begin{proof}
By \eqref{eq5.4} and the definition of $\wt{G}$
\begin{equation}\label{eq5.13}
\begin{aligned}
\wt{y}(t)=& I_1(t) + I_2(t) \triangleq \int_{+\infty}^t e^{(t-\tau)(\frac{J}{\ep}+g_y)}
\Big(\wt{G}(\wt{x},\wt{y},\xi_c,\ep)-\wt{G}(\wt{x},0,\xi_c,\ep)\Big) \Big)d\tau\\
&+\int_{+\infty}^te^{(t-\tau)(\frac{J}{\ep}+g_y)}\Big(\wt{G}(\wt{x},0,\xi_c,\ep)+g_x\wt{x}
\Big)d\tau.
\end{aligned}
\end{equation}
Since $|D_{(\wt{x},\wt{y})}\wt{G}|_{C^0}=|D_{(x,y)}G|_{C^0}\leq\ovr$
and $a_1<\eta<\eta+\eta'<a_2$,
\begin{equation}\label{eq5.14}
\begin{aligned}
\sup_{t\geq0}\frac{1}{\eps}e^{-(\eta+\eta')t}|I_1 (t) |_{Y_1} \leq
\frac{K\ovr}{\eta+\eta'-a_1}|\wt{y}|_{\eta+\eta',\eps,Y_1}.
\end{aligned}
\end{equation}
To estimate $I_2(t)$, we integrate by parts to obtain
\[\begin{aligned}
I_2(t)& =\ep (J +\ep g_y)^{-1}\Big( -\big(g_x \wt{x}(t) + \wt{G}(\wt{x}(t),0,\xi_c,t, \ep)\big)
+\int_{+\infty}^te^{(t-\tau)(\frac{J}{\ep}+g_y)}\Big[\\
&\big(g_x + D_xG(x(\xi_c)+\wt{x},y(\xi_c),\tau+t_0,\ep)\big)\dot{\wt{x}} + \partial_{t} \wt{G} (\wt{x}, 0,
\xi_c,\tau,\ep)  \\
&+\big(DG(x(\xi_c)+\wt{x}, y(\xi_c),\tau+t_0,\ep) - DG(x(\xi_c), y(\xi_c), \tau +t_0, \ep)\big)
(\dot{x}(\xi_c), \dot y(\xi_c)) \Big]d\tau \Big),
\end{aligned}\]
where $\wt{x}$, $(x(\xi_c)(\tau),y(\xi_c)(\tau))$, and their time derivatives in the above integral are all
evaluated at $\tau$. Using the differential equations satisfied by $\dot x(\xi_c)$ and $\dot {\wt{x}}$, it is
easy to show
\begin{equation} \label{E:dotx}
|\dot x(\xi_c)|_{\eta',1,X}^+ + |\dot {\wt{x}}|_{\eta + \eta',1,X}^+ \le C'.
\end{equation}
Since $\xi_y=0$, from \eqref{eq4.5} and use \eqref{eq4.61} to estimate its right side,
we obtain for any $\tau$,
\begin{equation} \label{E:doty}\begin{aligned}
|\dot{y}(\xi_c)|_{\eta', 1,Y} \leq &\frac {C'}\ep |y(\xi_c)|_{\eta',1,X_1}^+ + \big(|g|_{L(X, Y)}
+ |D_x G_*^\ep|_{C^0 (X_1 \times \R, Y)}\big) |x(\xi_c)|_{\eta',1,X_1}^+ \\
\leq& C' (|g_x|_{L(X, Y)} + |D_x G_*^\ep|_{C^0(X_1 \times \R, L(X, Y))}
+ |\p_t G_*^\ep|_{C_t^0C_x^1(X_1 \times \R, Y)}),
\end{aligned}\end{equation}
where $C'$ depends on $K,a_1',\eta',\ovr,\eps,|\xi_c|_{X_1}$. Using \eqref{E:dotx}, the bound on $D\wt{G}$,
the estimates on $|\wt{x}|_{\eta, 1, X_1}$ and $|\wt{x}|_{\eta + \eta', 1, X_1}$ from Theorem \ref{thm5.2},
assumptions for $k=2$, it is straight forward to see $|I_2|_{\eta+\eta', \eps, Y_1} \le C'\ep$ and along
with \eqref{eq5.14}, we have  $|\wt{y}|_{\eta+\eta',\eps,Y_1} \le C' \ep$.

In order to obtain a more careful estimate in terms of $g_x$ and $G_*^\ep$ when $k\ge 3$ and assuming the
extra assumption on $\p_t G$, skipping $t$ and $\ep$, we rewrite
\[
D_xG(x(\xi_c)+\wt{x},y(\xi_c))= D_x G_*^\ep (x(\xi_c)+\wt{x}) + \int_0^1 D_{xy} G (x(\xi_c)+\wt{x},
s y(\xi_c)) ds y(\xi_c).
\]
Similarly, rewrite $\wt{G} (\wt{x}, 0, \xi_c)$ (as well as $\p_t \wt{G}$ and $DG(x(\xi_c) + \wt{x}, \ldots)
- DG(x(\xi_c), \ldots)$)
\[\begin{split}
&\wt{G} (\wt{x}, 0, \xi_c) = G(x(\xi_c) + \wt{x}, y(\xi_c)) - G(x(\xi_c), y(\xi_c) = \int_0^1 D_x G_*^\ep
(x + s\wt{x})ds \wt{x} \\
&\qquad \qquad \qquad \qquad + \int_0^1 \int_0^1 D^2 G\big(x(\xi_c) + s_1 \wt{x}, s_2 y(\xi_c)
\big) ds_1 ds_2 \Big( \big(\wt{x}, 0\big), \big(0, y(\xi_c) \big)\Big).
\end{split}\]
Therefore, in the estimate of $I_2$, each term either directly has a factor $g_x$ or $G_*^\ep$ or indirectly
from $y(\xi_c)$ and \eqref{eq4.61} which implies
\[
|\wt{y}|_{\eta+\eta', \eps, Y_1} \le C'\ep (|g_x|_{L(X, Y)} + |D_x G_*^\ep|_{C^0(X_1 \times \R, L(X, Y))}
+ |\p_t G_*^\ep|_{C_t^0C_x^1(X_1 \times \R, Y)}).
\]

Using integral equations of $\wt{x}(t)$ and $\wt{x}_*(t)$, we have
\begin{eqnarray*}
&&\wt{x}(t)-\wt{x}_*(t) = \Big(\int_{0}^{t} P_{s} +\int_{+\infty}^t P_{cu}\Big) e^{(t-\tau)\af}
\Big(\wt{F}(\wt{x},\wt{y},\xi_c,\ep)-\wt{F}_*^\ep (\wt{x}_*,\xi_c)+
f_y\wt{y}\Big)d\tau.
\end{eqnarray*}
We can write, skipping $\tau+t_0$ and $\ep$,
\[\begin{split}
&\wt{F}(\wt{x},\wt{y},\xi_c,\ep)-\wt{F}_*^\ep (\wt{x}_*,\xi_c) \\
=& F(x(\xi_c) + \wt{x}, y(\xi_c) + \wt{y}) - F(x(\xi_c), y(\xi_c)) - F(x_*(\xi_c) + \wt{x}_*, 0)
+ F(x_*(\xi_c, 0))\\
=& \int_0^1 D F\big(x(\xi_c) + s \wt{x} + (1-s) \wt{x}_*, y(\xi_c) + s \wt{y}\big) ds \big( \wt{x}-\wt{x}_*,
\wt{y} \big) + \int_0^1 \int_0^1 \\
&D^2F \big(s_1 x(\xi_c) + (1-s_1) x_*(\xi_c) + s_2 \wt{x}_*, s_1 y(\xi_c)\big) ds_1 ds_2
\Big( \big(x(\xi_c) - x_*(\xi_c), y(\xi_c)\big), \big(\wt{x}_*, 0\big)\Big)
\end{split}\]
Combining the estimates on $|x(\xi_c) - x_*(\xi_c)|_{\eta',1,X_1}$ and $|y(\xi_c)|_{\eta',1,Y_1}$ from
\eqref{eq4.61}, $|\wt{x}_*|_{\eta, 1, X_1}$ from the standard theory (like in Theorem \ref{thm5.2}),  $|\wt{y}|_{\eta+\eta',\eps,Y_1}$ from the above, we obtain the desired estimates on
$|\wt{x} -\wt{x}_*|_{\eta+\eta',1,X_1}$ and thus complete the proof.
\end{proof}

\section{Normally elliptic singular perturbations to homoclinic solutions} \label{S:homo}
In this section, we will discuss the persistence of a homoclinic solution
under normally elliptic singular perturbations. We assume
(A1)-(A5) for $k=2$, (A5$'$), (B1)-(B5) in Section 4 and (C1)-(C2)
after Theorem \ref{thm4.5}. In this whole section, we assume
\begin{enumerate}
\item[(D1)] $A$ generates a strongly continuous group on $X$ and
$X_1^u$ has finite dimension.
\item[(D2)] There exist $\eta$ and $\eta'$ such that
\[\begin{aligned}
&a_1<2\eta< \eta<\min\{0,a_2\} \ , \ \max\{0,a_1'\}<\eta'<2\eta'<a_2',\\
&a_1<\eta+\eta'<a_2 \ , \ a_1+\eta'<0,
\end{aligned}\]
where $a_1,a_2,a_1',a_2$ are defined in (B5) and (C2).
\item[(D3)] \eqref{eq3.2} has a homoclinic orbit
$x_h(t)$ such that $|Ax_h(t)|_{X_1}$ is bounded and
\[
\sup_{t\geq0}e^{-a_1t}|x_h(t)|_{X_1}<\infty\ , \
\sup_{t\leq0}e^{-a_2' t}|x_h(t)|_{X_1}<\infty.
\]
\item[(D4)] There exists a $C^2$ invariant quantity $H:X_1\rightarrow\mathbb{R}$ with
$DH \in C^1 (X_1, L(X, \mathbb{R}))$ such that
\[
H(0)=0 \ , \ DH(0)=0.
\]
\item[(D5)] At $x_0=x_h(0)$,
\[
DH(x_0)\neq 0\ , \ \dim (T_{x_0}\mathcal{M}_0^{u}\bigcap
T_{x_0}\mathcal{M}_0^{cs})=1,
\]
where $\CM_0^u$ and $\CM_0^{cs}$ are the unperturbed unstable and center-stable manifolds
of $0$.
\end{enumerate}
Our goal is to study if \eqref{eq3.1} has a homoclinic solution to $0$ when $0<\ep\ll1$ and how this problem
are handled under normally elliptic singular perturbations. We will consider both the weakly dissipative and
the conservative cases via a geometric approach based on invariant manifolds. For the former, a more analytic
method based on the Lyapunov-Schmidt reduction may also work \cite{CH82, SZ00} to give the persistent
homoclinic solution, but the geometric method provide more information such as the transversality of the
intersection of the stable and unstable manifolds. For the latter, we are not aware of such an analytic method
even in similar regular perturbations, so we follow the geometric approach as in \cite{SZ03}.

Since $f$ is independent of $t$ when $\ep=0$, in this section, we will write
\[
f_0 (x) = f(x, 0, t, 0) \qquad g_0(x) = g(x, 0, t, 0).
\]
We use $B_{\rho}(p,S)$ to denote the ball in a space $S$ of radius $\rho$ centered at $p$ which is often
skipped if $p=0$. We will also keep using $P_X$ and $P_Y$ to denote the projections.

In he following Subsection \ref{SS:coord}, a coordinate system around the unperturbed homoclinic
orbit. Subsection \ref{SS:diss} is devoted to study the persistence of the homoclinic orbit under weakly
dissipative perturbations and Subsection \ref{SS:cons} is to study the conservative and autonomous case,
i.e. $f,g$ are assumed to be independent of $t$ for all $\ep\geq0$ in Subsection \ref{SS:cons}. The example
of the elastic pendulum will be revisited.

\subsection{Coordinates around the unperturbed homoclinic orbit} \label{SS:coord}

Locally near $0$, we cut off the nonlinearity as in section 4 to
obtain $h_{cs},h_u$ and thus the local integral manifolds. We also
use $\mathcal{M}_{\beta}^{\alpha}(t_0)$ to denote the global
integral manifolds corresponding to the time $t_0$ extended by the
flow from the local ones of systems \eqref{eq3.1} and \eqref{eq3.2},
where $\alpha=cs,u,cu,s,c \ ,\ \beta=0,\ep$. The assumption $H(0)=0$
and the invariance of the fibers and $H$ imply
$H|_{\mathcal{M}_{0}^{s}}\equiv 0$ and thus $T_{x}\mathcal{M}_0^{s}
\subset\ker(DH(x))$ for any $x \in \CM_0^s$. Moreover, by assumption
(D3), $\mathcal{M}_{0}^{cs}$ can be foliated into the disjoint union
of $C^2$ invariant stable fibres which are $C^1$ with respect to the
based point. Therefore, there exists a nonlinear projection $f^s\in
C^1 (\mathcal{M}_{0}^{cs}, \mathcal{M}_{0}^{c})$, which maps points
on each fiber to their based point, such that
\[
f^s|_{\mathcal{M}_{0}^{s}}=(0,0)\ , \ f^s|_{\mathcal{M}_{0}^{c}}=I.
\]
The fiber invariance implies $H=H \circ f^s$ on $\mathcal{M}_{0}^{cs}$. So for any $x \in \CM_0^s$ and
$\delta x\in T_{x}\mathcal{M}_{0}^{cs}$, using the assumption $DH(0)=0$, we obtain
\[
DH(x)\delta x=DH(f^s(x))Df^s(x)\delta x =DH(0)Df^s(x)\delta x=0,
\]
which implies
\begin{equation} \label{E:DH1}
T_{x}\mathcal{M}_{0}^{cs}\subset\ker(DH(x)), \qquad \forall x \in \CM_0^s.
\end{equation}
Similar properties also hold for the unstable and center-unstable manifolds.

In the enlarged phase space, we trivially extend the domain of $H$ from $X_1$ to
$X_1\times Y_1$. Clearly,
\begin{equation}\label{eq6.1}
H(0,0)=0 \ , \ DH(0,0)=0.
\end{equation}
To study the perturbation of the homoclinic solution, we need to take a cross section. Let
\[
v=Ax_0+f_0(x_0).
\]
Since $v\in X_1\subset X$, there exists a hyperplane $\Sigma'\subset X$ that is transverse to $v$. Let
$\Sigma=(\Sigma'\bigcap X_1)\times Y_1$, by using $v\in X_1$, one
can prove $v$ and $\Sigma$ are transverse in $X_1\times Y_1$. Let
$Q_v,Q_v'$ be the projections from $X_1\times Y_1$ and $X\times Y$
onto $\mathbb{R}v$ with kernel $\Sigma$ and $\Sigma'\times Y$,
respectively. We will identify the range of $Q_v$ and $Q_v'$, i.e.
$\mathbb{R}v$, with $\mathbb{R}$. Let
\[
\wt{\mathcal{M}}_0^u=\mathcal{M}_0^u\bigcap(x_0+\Sigma) \ , \
\wt{\mathcal{M}}_0^{cs}=\mathcal{M}_0^{cs}\bigcap(x_0 + \Sigma) \ , \
\overline{X}_1^u=T_{x_0}\wt{\mathcal{M}}_0^u \ , \
\overline{X}_1^{cs}=T_{x_0}\wt{\mathcal{M}}_0^{cs}.
\]
From \eqref{E:DH1}, we have
\[
\overline{X}_1^{cs} \ , \ \overline{X}_1^{u} \ ,\ Y_1\subset
\ker(DH(x_0))\bigcap\Sigma\triangleq\Pi.
\]
We use $\mbox{Codim}_W(Z)$ to represent the codimension of a linear
subspace $Z$ in a Banach space $W$. On the one hand, since $\ker(DH(0))$ is a hyperplane, $v \in \ker(DH(0))$,
and $v \notin \Sigma$, we have $\mbox{Codim}_{\Sigma}(\Pi)=1$. On the other hand, (D4) implies
$\overline{X}_1^{cs}\bigcap \overline{X}_1^{u}=\{0\}$. Moreover, dim-$X_1^u <\infty$ implies  $\mbox{Codim}_{\Sigma}(\overline{X}_1^{cs}\oplus \overline{X}_1^{u}\oplus Y_1)=1$. Therefore,
\[
\Pi=\overline{X}_1^{cs}\oplus \overline{X}_1^{u}\oplus Y_1.
\]
Let $\omega\in\Sigma$ be transversal to $\Pi$ such that
$DH(x_0)\omega=1$ and $Q_{\omega},Q_{cs},Q_{u},Q_y$ be projections
from $\Sigma$ onto $\omega,\overline{X}_1^{cs},\overline{X}_1^{u}$
and $Y_1$. Thus,
\[
\Sigma=\mbox{span}\{\omega\}\oplus\Pi=\mbox{span}\{\omega\}\oplus\overline{X}_1^{cs}\oplus
\overline{X}_1^{u}\oplus Y_1.
\]
We will use coordinates
\begin{equation}\label{eq6.11}\begin{aligned}
&(d,x^{cs},y,x^u)= (Q_{\omega}(p-x_0),Q_{cs}(p-x_0),Q_y(p-x_0),Q_u(p-x_0))\\
=&(DH(x_0)(p-x_0),Q_{cs}(p-x_0),Q_y(p-x_0),Q_u(p-x_0))
\end{aligned}
\end{equation}
to represent any $p\in\Sigma+x_0$. Locally, there exist $\delta>0$ and
\[
\Upsilon_0:B_{\delta}(\overline{X}_1^{cs})\longrightarrow\mathbb{R}\times\overline{X}_1^u \ , \
\Psi_0:B_{\delta}(\overline{X}_1^{u})\longrightarrow\mathbb{R}\times\overline{X}_1^{cs},
\]
such that the graphs of $\Upsilon_0,\Psi_0$ are open subsets of $\wt{\mathcal{M}}_0^{cs},
\wt{\mathcal{M}}_0^u$, respectively. We extend $\Upsilon_0$ to $B_{\delta}(\overline{X}_1^{cs})\times
Y_1$ trivially in $y$.

To study the perturbed problem, let $r$ be the cut-off radius
defined in section 4, there exist $t_1>0,t_2<0$ such that
\begin{equation}\label{eq6.2}
|x_{1,2}|_{X_1}<\frac{r}{2\big(1+|P_{cs}|(1+|Dh_u|_{C^0})+|P_u|(1+|Dh_{cs}|_{C^0})\big)},
\end{equation}
where $x_{1,2}=x_h(t_{1,2})$. Recall that $\Phi (t, t_0, x+y, \ep)$ and $\Phi^0 (t, x)$ denote the flow
maps with the terminal time $t$ of \eqref{eq3.1} and \eqref{eq3.2}, respectively. We first show that for any $t_0\in\mathbb{R}$, $\mathcal{M}_{\ep}^{cs}(t_0)$ does intersect $\Sigma$ near $x_0$ for $\ep\ll1$.

\begin{lemma}\label{lemma6.1}
For any $t_0\in\mathbb{R}$, there exists a unique $t'=t'(t_0, \ep)$ such that
\[\begin{aligned}
&\Phi(t_0,t_0+t',x_1',\ep)\in x_0+\Sigma,\\
&\big|\Phi(t_0,t_0+t',x_1',\ep)-x_0\big|_{X_1\times Y_1} + |t'-t_1| + |\partial_{t_0}t'|\leq C'\ep,
\end{aligned}\]
where
\[
x_1'= x_1'(t', \ep)=P_{cs}x_1+h_u(P_{cs}x_1,t_0+t',\ep) \in \CM_\ep^{cs} \cap X_1
\]
and $C'$ depends on constants in assumptions of this section.
\end{lemma}

\begin{proof}
The proof is obviously based on the Implicit Function Theorem, however, we have to be rather careful
due to the singular perturbation natural of the problem. We will use $\partial_1\Phi,\partial_2\Phi$ to
denote the differentiation with respect to terminal and initial time,
respectively. For any $t_0\in\mathbb{R}$, let
\[
\gamma(t',\ep)=Q_v'(\Phi(t_0,t_0+t',x_1',\ep)-x_0) \ , \
\gamma(t',0)=Q_v'(\Phi^0(-t',x_1)-x_0).
\]
Theorem \ref{thm3.1} implies, for $t'$ on any bounded interval,
\begin{equation}\label{eq6.3}
|\gamma(t',\ep)-\gamma(t',0)|\leq C'\ep.
\end{equation}
To show the $C^1$ closeness of $\gamma(\cdot, \ep)$ and $\gamma(\cdot, 0)$, using the definition
of $x_1'$, one can compute
\[\begin{aligned}
\Big|\partial_{t'}\gamma(t',\ep)&-\partial_{t'}\gamma(t',0)\Big| \leq \Big|Q_v'D\Phi(t_0,t_0+t',x_1',\ep)
\partial_{t_0}h_u(x_1',t_0+t',\ep)\Big|\\
&+\Big|Q_v'\Big(D\Phi(t_0,t_0+t',x_1',\ep) V_{\ep}(t_0+t',x_1') -D\Phi^0(-t',x_1)V_0(x_1)\Big)\Big|,
\end{aligned}\]
where $V_{\ep}(t,x),V_0(x)$ represent the velocity field of \eqref{eq3.1} and \eqref{eq3.2} at $(t,x)$,
respectively. From Proposition \ref{P:appro} and Remark \ref{R:other}, we have $|x_1'-x_1|_{X_1}\leq C'\ep$.
Explicit computations based on the forms of \eqref{eq3.1} and \eqref{eq3.2} imply
$|P_X(V_{\ep}(t_0+t',x_1')-V_{0}(x_1))|_{X}\leq C'\ep$. Applying Theorem \ref{thm3.2}, we obtain
\[
\Big|Q_v'\big(D\Phi(t_0,t_0+t',x_1',\ep)V_{\ep}(t_0+t',x_1')
-D\Phi^0(-t',x_1)V_0(x_1)\big)\Big| \leq C'\ep.
\]
From Theorem \ref{thm4.26} and Remark \ref{R:other}, we have
$|\partial_{t_0}h_u(x_1',t_0+t',\ep)|_X\leq C'\ep$ and thus
\begin{equation}\label{eq6.4}
\Big|Q_v'D\Phi(t_0,t_0+t',x_1',\ep)\partial_{t_0}h_u(x_1',t_0+t',\ep)\Big|\leq
C'\ep.
\end{equation}
Therefore, we have proved $\gamma(t',\ep)$ and $\gamma(t',0)$ are $C^1$ close for $t'$ on
bounded intervals.

Since the system \eqref{eq3.2} is autonomous when $\ep=0$,
\[
\gamma(t_1, 0)=0, \qquad \partial_{t'}\gamma(t_1,0)=-Q_v'\partial_1\Phi^0(-t_1,x_1)=-Q_v'v=-1.
\]
By implicit function theorem, there exists a unique $t'=t'(t_0)$
such that
\[
\Phi(t_0,t_0+t',x_1',\ep) \in x_0 + \Sigma \qquad |t'(t_0)-t_1|\leq C'\ep.
\]
Moreover, from Theorem \ref{eq3.1} and the $C^1$ smoothness of $\Phi(-t', x_1) \in X_1$ in $t'$
which is due to the assumption $Ax_h\subset X_1$, it is easy to obtain
\begin{equation}\label{eq6.5}
\big|\Phi(t_0,t_0+t',x_1',\ep)-x_0\big|_{X_1\times Y_1}\leq C'\ep.
\end{equation}
Finally, note that
\[\begin{aligned}
\partial_{t_0}\gamma(t',\ep)=&Q_v'\Big(V_{\ep}(t_0,\Phi(t_0,t_0+t',x_1',\ep))
-D\Phi(t_0,t_0+t',x_1',\ep)V_{\ep}(t_0+t',x_1')\\
&+D\Phi(t_0,t_0+t',x_1',\ep)\partial_{t_0}h_u\Big).
\end{aligned}\]
By Theorem \ref{thm3.2} and \eqref{eq6.4}, \eqref{eq6.5}, we have
$|\partial_{t_0}\gamma(t',\ep)|\leq C'\ep$, which implies the desired estimate on $\p_{t_0} t'$.
\end{proof}

Next we consider the tangent space $T(\Sigma\bigcap\mathcal{M}_{\ep}^{cs}(t_0))$ near
$\Phi(t_0,t_0+t'(t_0),x_1',\ep)$ based on Theorem \ref{thm3.2} and more directly Proposition \ref{thm4.20}
and Lemma \ref{L:FTL}.
Let $E^\ep (t, t_0, x)$ be the evolution operator defined in Theorem \ref{thm3.2} for $x \in X_1$, i.e.
\[
\p_1 E^\ep (t, t_0, x) = \big(\frac J\ep + D_y g_0 (\Phi^0(t-t_0, x))\big) E^\ep (t, t_0, x), \qquad
E^\ep (t_0, t_0, x) =I_Y.
\]
We notice that the operator $E$ defined in Proposition \ref{thm4.20} is only $O(\ep)$ away from $E^\ep$ on
any finite interval. Therefore, we have

\begin{lemma}\label{lemma6.2}
Let $t'=t'(t_0, \ep)$ be the one found in Lemma \ref{lemma6.1}. For any $C>0$ and small $\delta>0$
(independent of $\ep$) and $(\xi_{cs},\xi_y)\in B_{\delta}(P_{cs}x_1,X_1^{cs})\times B_{C\ep}(Y_1)$,
we have
\begin{equation*}\label{eq6.7}
\big|\Phi(t_0,t_0+t',\xi_{cs}+\xi_y+h_{u}(\xi_{cs},\xi_y,t_0+t',\ep),\ep)-x_0\big|_{X_1\times
Y_1}\leq C'\delta.
\end{equation*}
Moreover, if $(\delta x,\delta y)\in X_1^{cs}\times Y_1$ with $|{\delta x}|_{X_1}+|\delta
y|_{Y_1}\leq 1$, then
\begin{equation*}\label{eq6.8}\begin{aligned}
&\Big|D_{\xi_{cs}}\big(\Phi (t_0,t_0+t',\xi_{cs}+\xi_y+h_{u}(\xi_{cs},\xi_y,t_0+t',\ep),\ep)\big)
\delta x\\
&\hspace{2cm}-D_{\xi_{cs}}\big(\Phi^0(-t',\xi_{cs}+h_{u}^0(\xi_{cs}))\big){\delta
x}\Big|_{X_1}\leq C'\ep,
\end{aligned}
\end{equation*}
and
\begin{equation*}
\begin{aligned}
&P_XD_{\xi_y}\big(\Phi (t_0,t_0+t',\xi_{cs}+\xi_y+h_{u}(\xi_{cs},\xi_y,t_0+t',\ep),\ep)\big)\delta
y \in B_{C'\ep} (X) \cap B_{C'} (X_1)\\
&\Big|\Big(P_YD_{\xi_y}\big(\Phi (t_0,t_0+t',\xi_{cs}+\xi_y+h_{u}(\xi_{cs},\xi_y,t_0+t',\ep),\ep)\big)\\
&\hspace{3.5cm}-E(t_0,t_0+t';\xi_{cu}+h_u^0(\xi_{cu}),\ep)\Big)\delta
y\Big|_{Y_1}\leq C'\ep.
\end{aligned}
\end{equation*}
where $C'$ depends on $C$ and those constants in assumptions.
\end{lemma}

The $C^2$ smoothness of $\Phi$ and $h_u$ and the assumption dim$-X_1^u<\infty$, which implies the equivalence
between $|\cdot|_{X^u}$ and $|\cdot|_{X_1^u}$, are used in the proof. Moreover, even though Proposition
\ref{thm4.20} is stated only for $\xi_y=0$, our assumption $|\xi_y|_{Y_1} =O(\ep)$ combined with the smoothness of
$h_u$ in $\xi_y$ (Theorem \ref{thm4.6}) is sufficient to guarantee the above estimates.

In the next lemma, we will write $\CM_\ep^{cs} \cap (x_0 + \Sigma)$
locally near $x_0$ in the coordinate system $(d, x^{cs}, y, x^u)$.
The main issues are the size of the coordinate chart of the manifold
and the regularity estimates.

\begin{lemma}\label{lemma6.3}
For any $b>0$, there exist $\ep_0>0$, $b'>0$ and
\[
\Upsilon=(\Upsilon^d,\Upsilon^u):B_{b'} (\overline{X}_1^{cs})\times B_{b}(Y_1)\times
\mathbb{R}\times(0,\ep_0) \rightarrow(\mathbb{R},\overline{X}_1^{u})
\]
such that $x_0 + \mbox{Graph}(\Upsilon(t_0,\ep))$ is an open subset of $\wt{\mathcal{M}}_{\ep}^{cs}(t_0)$
where
\[\begin{aligned}
\mbox{Graph}(\Upsilon(t_0,\ep)) \triangleq
\big\{\Upsilon^d(x^{cs}, y,t_0,\ep) \omega + x^{cs}&+\ep y +\Upsilon^u(x^{cs},y,t_0,\ep)\big|\\
&x^{cs}\in B_{b'}(\overline{X}_1^{cs}),y\in B_{b}(Y_1)\big\}.
\end{aligned}\]
Moreover, $\Upsilon$ are $C^2$ in $x^{cs},y$ and satisfy
\begin{align}
\label{eq6.14} &\big|\Upsilon-\Upsilon_0\big|_{C^1(B_{b'}
(\overline{X}_1^{cs})\times B_{b}(Y_1), \overline{X}_1^u \times \R)}+|\partial_{t_0}\Upsilon|_{C^0}\leq
C'\ep,
\end{align}
where $C'$ only depends on $b$ and those constants in assumptions.
\end{lemma}

For $\ep=0$, we define $\Upsilon (x^{cs}, y, t_0, 0) = \Upsilon_0 (x^{cs})$.
Notice in the definition of $\mbox{Graph}(\Upsilon)$ we
scale $y$ to $\ep y$. This is to avoid the dependence on $\ep$ of
the domain where $\Upsilon$ is defined.

\begin{proof}
The first step of the proof is to establish a correspondence between an open set of
$\wt{\CM}_\ep^{cs}(t_0)$, which is a hypersurface of $\CM_\ep^{cs}(t_0)$, near $x_0$ and a hypersurface of
$\CM_\ep^{cs} (t_0+t')$ near $x_1'$.

Let $w=P_{cs}(Ax_1+f_0(x_1))$ and $\wt{X}_1^{cs}\subset
X_1^{cs}$ such that $X_1^{cs}=\mathbb{R}w\oplus \wt{X}_1^{cs}$. For any $b_1>0$, define
\[\begin{aligned}
\wt{\mathcal{F}}(a,\xi_{cs}',\xi_y,\ep)=
\Phi\big(t_0,&t_0+t',P_{cs}x_1+aw+\xi_{cs}'+\ep\xi_y\\
&+h_u(P_{cs}x_1+aw+\xi_{cs}',\ep\xi_y,t_0+t',\ep),\ep\big)-x_0.
\end{aligned}\]
Here $a\in[-\delta,\delta]$, $\xi_{cs}'\in
B_{\delta}(\wt{X}_1^{cs})$ and $\xi_y\in B_{b_1}(Y_1)$, where
$\delta>0$ sufficiently small but independent of $\ep$. From Lemma \ref{lemma6.2},
we have
\begin{equation}\label{eq6.17}
\big|\wt{\mathcal{F}}(\cdot,\ep)-\wt{\mathcal{F}}(\cdot,0)\big|_
{C^1([-\delta,\delta]\times B_{\delta}(\wt{X}_1^{cs})\times
B_{b_1}(Y_1),X_1\times Y_1)}\leq C'\ep,
\end{equation}
where $C'$ depends on those constants in the assumptions. Lemma
\ref{lemma6.1} implies $Q_v\wt{\mathcal{F}}(0,0,0,\ep)=0$ for all $\ep\in[0,\ep_0)$.
From the Implicit Function Theorem, we obtain that, when $\delta$ and $\ep_0$ are sufficiently
small there exists $a:B_{\delta}(\wt{X}_1^{cs}) \times B_{b_1}(Y_1)\times[0,\ep_0)\rightarrow
[-\delta,\delta]$ such that
\begin{equation} \label{E:a} \begin{aligned}
\mathcal{F} (\xi_{cs}',\xi_y,\ep) \triangleq \wt{\mathcal{F}} (a(\xi_{cs}',\xi_y,\ep),\xi_{cs}',\xi_y,\ep)
\in\Sigma, \quad \forall \, (\xi_{cs}',\xi_y)\in B_{\delta}(\wt{X}_1^{cs})\times B_{b_1}(Y_1).
\end{aligned}\end{equation}
For $\ep=0$, we identify $a(\xi_{cs'},\xi_y,0)$ with $a_0(\xi_{cs}')$, which satisfies
\[\begin{aligned}
\Phi^0(-t_1,P_{cs}x_1+&a_0(\xi_{cs}')w+\xi_{cs}' +h_u^0(P_{cs}x_1+a_0(\xi_{cs}')w+\xi_{cs'})) - x_0
\in\Sigma\cap X_1.
\end{aligned}\]
Moreover, by assumption (D2) and Theorem \ref{thm4.3}, $a$ is $C^2$ in $\xi_{cs}',\xi_y$ , $a_0$ is $C^2$ in
$\xi_{cs}'$, $a(0,0,\ep)=a_0(0)=0$, and
\begin{equation}\label{eq6.21}
\big|a(\cdot,\cdot,\ep)-a_0(\cdot) \big|_{C^1(B_{\delta}(\wt{X}_1^{cs})\times
B_{b_1}(Y_1)\times[0,\ep_0),[-\delta,\delta])}\leq C'\ep.
\end{equation}
Consequently,
\begin{equation}\label{eq6.22}
|\mathcal{F}(\cdot,\cdot,\ep)-\mathcal{F}_0(\cdot)|_{C^1 (B_{\delta}(\wt{X}_1^{cs})\times
B_{b_1}(Y_1)\times[0,\ep_0), X_1 \times Y_1)}\leq C'\ep
\end{equation}
where
\[
\mathcal{F}_0(\xi_{cs}') =\mathcal{F}(\xi_{cs}',\xi_y,0) =\wt{\mathcal{F}}
(a(\xi_{cs}',\xi_y,0),\xi_{cs'} ,\xi_y,0) =\mathcal{F}(a_0(\xi_{cs}'), \xi_{cs}',0,0).
\]
To obtain the estimate on $\p_{t_0} a$, notice \eqref{E:a} is equivalent to $Q_v' \wt{\mathcal{F}}=0$.
Differentiate it with respect to $t_0$ and wrote $\xi_{cs}$ for $P_{cs}x_1+aw+\xi_{cs}'$, we note that
\[\begin{aligned}
&\big|\partial_{t_0}Q_v'\wt{\mathcal{F}}(a,\xi_{cs}',\xi_y,\ep)\big| \\
\leq &\Big|Q_v' \big(V_{\ep} (t_0,\wt{\mathcal{F}} (a,\xi_{cs}',\xi_y,\ep)+x_0)- (1+\partial_{t_0}t')
D\Phi(t_0,t_0+t', \xi_{cs}+\ep\xi_y \\
&+h_u(\xi_{cs},\ep\xi_y,t_0+t',\ep),\ep) V_{\ep}(t_0+t',\xi_{cs} +\ep\xi_y +h_u(\xi_{cs},\ep\xi_y,t_0+t',\ep))
+D\Phi\partial_{t_0}h_u\big)\Big|\\
\leq&\Big|Q_v' \big(V_{0}(\wt{\mathcal{F}}(a,\xi_{cs}',\xi_y,0) +x_0)-D\Phi^0(-t_1,\xi_{cs}
+h_u^0(\xi_{cs}))V_{0}(\xi_{cs}+h_u^0(\xi_{cs}'(a)))\big)\Big|+C'\ep.
\end{aligned}\]
Here we use \eqref{eq6.17}, Lemma \ref{lemma6.1}, \ref{lemma6.2}, Proposition \ref{thm4.26}, Theorem
\ref{thm3.2} to obtain the above estimates. Finally, the term other than $C'\ep$ in the above right side
vanishes since the system is autonomous when $\ep=0$. Therefore $|\partial_{t_0}a|\leq C'\ep$ which along
with a similar procedure implies
\begin{equation}\label{eq6.23}
|\partial_{t_0}\mathcal{F}|_X \leq C'\ep.
\end{equation}

We claim for any $b>0$ there exist $b_1$, $b'>0$
independent of $\ep$, such that the map
\[
(Q_{cs}\mathcal{F},\frac{1}{\ep}Q_y\mathcal{F})^{-1}:B_{b'}(\overline{X}_1^{cs})\times
B_{b}(Y_1)\longrightarrow B_{\delta}(\wt{X}_1^{cs})\times B_{b_1}(Y_1)
\]
is well defined and its $C^1$ norm is uniform in $\ep$.

To prove this, we need solve the equations
\begin{equation} \label{E:inverse}
a.)\quad Q_{cs}\mathcal{F}(\xi_{cs}',\xi_y,\ep)=x^{cs}, \qquad b.) \quad \frac{1}{\ep}
Q_y\mathcal{F}(\xi_{cs}',\xi_y,\ep)=y.
\end{equation}
We first find a good approximation of this system of equations. By \eqref{eq6.22} and Lemma \ref{lemma6.2},
one can compute
\begin{equation}\label{eq6.24}
\begin{aligned}
&Q_{cs}D_{\xi_{cs}'}\mathcal{F}=Q_{cs}D_{\xi_{cs}'}\mathcal{F}_0+\ep O_1, \qquad
Q_{cs}D_{\xi_y}\mathcal{F}=\ep O_2,\\
&Q_yD_{\xi_{cs}'}\mathcal{F}=\ep O_3, \qquad Q_yD_{\xi_y}\mathcal{F}=\ep E+\ep^2 O_4,
\end{aligned}\end{equation}
where $O\triangleq \begin{pmatrix} O_1 &O_2\\O_3 &O_4\end{pmatrix} \in L(\wt{X}_1^{cs} \times Y_1,
X_1 \times Y_1)$ is bounded uniformly in $\ep$ and $E$ is the linear evolutionary operator defined in
Lemma \ref{lemma6.2} at the base point
\[
\xi^0 \triangleq P_{cs}x_1+a_0(\xi_{cs}')w+\xi_{cs}'+h_u^0(P_{cs}x_1+a_0(\xi_{cs}')w+\xi_{cs}').
\]
It implies that for fixed $\xi_{cs}'$ and $\ep<<1$,
\[
\big|\frac{1}{\ep}Q_y\mathcal{F}(\xi_{cs}',\xi_y,\ep)- \big(E\xi_y+\frac{1}{\ep}Q_y \mathcal{F}(\xi_{cs}',0,\ep)\big)\big|_{C^1(B_{b}(Y_1), Y_1)}\leq C'\ep.
\]
Since $E$ and $E^{-1}$ both have upper bounds independent of $\ep$ and $\mathcal{F}$ is $C^2$, by \eqref{eq6.24}
and an Implicit Function Theorem argument, for any $y\in B_{b}(Y_1)$, $\xi_{cs}'\in B_{\delta}(\wt{X}_1^{cs})$
and $\ep\in[0,\ep_0)$ b.) of \eqref{E:inverse} has a unique solution $\xi_y(\xi_{cs}',y,\ep)$ and
$|D_{\xi_{cs}'} \xi_y|\leq C'$. Along with \eqref{eq6.24}, it implies that, as a mapping of $\xi_{cs}'$,
\[
\big|Q_{cs}\mathcal{F}(\xi_{cs}',\xi_y(\xi_{cs}', y,\ep),\ep)
-Q_{cs}\mathcal{F}_0(\xi_{cs}')\big|_{C^1 (B_\delta (\wt{X}_1^{cs}), \overline{X}_1^{cs})}\leq C'\ep.
\]
Since $Q_{cs}\mathcal{F}_0$ is independent of $\ep$ and is locally
invertible, one can use an Inverse Function Theorem argument again to
prove there exist sufficiently small $b'>0$, $\ep_0>0$ so that for
$(x^{cs},y,\ep)\in B_{b'}(0,\overline{X}_1^{cs})\times
B_b(0,Y_1)\times[0,\ep_0)$, there exists a unique
$\xi_{cs}'(x^{cs},y,\ep)$ which is also $C^2$ in $x^{cs}$ and $y$
satisfying it along with $\xi_y (\xi_{cs}'s, y, \ep)$ solve \eqref{E:inverse}. Therefore, we proved the
existence of $(Q_{cs}\mathcal{F},\frac{1}{\ep}Q_y\mathcal{F})^{-1}$. The estimate on its $C^1$ norms follow
directly from \eqref{eq6.24}.

For $(x^{cs},y,t_0,\ep)\in B_{b'}(0,\overline{X}_1^{cs})\times
B_{b}(0,Y_1)\times\mathbb{R}\times[0,\ep_0)$, let
\begin{equation}\label{eq6.26}
x^{cs}+\ep y+\Upsilon(x^{cs},y,t_0,\ep)=\mathcal{F}\Big(
\big(Q_{cs}\mathcal{F},\frac{1}{\ep}Q_y\mathcal{F}\big)^{-1}(x^{cs},y),\ep\Big).
\end{equation}
When $\ep=0$, \eqref{eq6.26} is becomes
\[
x^{cs}+\Upsilon_0(x^{cs})=\mathcal{F}_0\Big(\big(Q_{cs}\mathcal{F}_0\big)^{-1}(x^{cs})\Big).
\]
Since $\mathcal{F}$ is $C^2$, $\Upsilon$ is also $C^2$. The estimates on $D \Upsilon$ follow
in a straight forward manner by differentiating \eqref{eq6.26} and using \eqref{eq6.24} and \eqref{eq6.23}.
\end{proof}

Finally, we present a similar coordinate representation of the stable manifold which is obtained in rather
similar fashion. For any $t_0\in\mathbb{R}$, there exists a unique $t''=t''(t_0, \ep)$ with $|t''-t_2|\leq
C'\ep$ such that
\[
\Phi(t_0, t'' + t_0, x'', \ep) \in x_0 + \Sigma, \text{ where } x''(t'', \ep) = P_u x_2 + h_{cs} (P_u x_2,
t_0 + t'', \ep)
\]
with similar estimates as in Lemma \ref{lemma6.1}. Moreover, $D\Phi (t_0, t_0 + t'', \cdot, \ep)$ satisfies
similar estimates as in Lemma \ref{lemma6.2} except there is no $D_{\xi_y}$ terms.

There exists $b>0$ sufficiently small but independent of $\ep$ and
\[
\Psi=(\Psi^d,\Psi^{cs},\Psi^y):B_{b}(\overline{X}_1^{u})\times\mathbb{R}
\times[0,\ep_0) \rightarrow(\mathbb{R},\overline{X}_1^{cs},Y_1)
\]
such that $\mbox{Graph}(\Psi(\cdot, t_0, \ep))$ is an open subset of $\wt{\mathcal{M}}_{\ep}^u(t_0)$ where
\[\begin{aligned}
\mbox{Graph}(\Psi(\cdot,t_0, \ep) \triangleq \Big\{&x^u+\Psi^d(x^u,t_0,\ep)
+\Psi^y(x^u,t_0,\ep) +\Psi^{cs}(x^u,t_0,\ep)\big|x^u\in B_{b}(\overline{X}_1^{u}) \Big\}.
\end{aligned}\]
Moreover, $\Psi^y$ is $C^2$ in $x^u$ and satisfy
\begin{equation}\label{eq6.28}
\big|\Psi(\cdot,t_0,\ep)-\Psi(\cdot,t_0,0)\big|
_{C^1(B_{b}(0,\overline{X}_1^{u}),\mathbb{R}\times\overline{X}_1^{cs}\times
Y_1)}\leq C'\ep \ , \ |\partial_{t_0}\Psi|_{X\times Y} \leq C'\ep,
\end{equation}
where $C'$ is independent of $\ep$.

\subsection{Persistence of homoclinic orbits under weakly dissipative
perturbation} \label{SS:diss}

In this subsection, we assume additionally
\begin{enumerate}
\item[(A7)]For $i=0,1,2$, the following quantities have a uniform bound
$C_0$,
\[(\partial_{\ep}^{2-i}D^if,\partial_{\ep}^{2-i}D^ig)\in C^0(X_1\times
Y_1\times\mathbb{R}^2,L_i(X_1\times Y_1,X_1\times Y_1)).\]
\end{enumerate}
In order to study the persistence of the homoclinic solution of \eqref{eq3.1}, we first derive the Melnikov
integral to measure the distance between $\wt{\mathcal{M}}_{\ep}^{cs}(t_0)$ and
$\wt{\mathcal{M}}_{\ep}^u(t_0)$. From the construction of the coordinate system, the intersection of $\wt{\mathcal{M}}_{\ep}^{cs}(t_0)$ and $\wt{\mathcal{M}}_\ep^u(t_0)$ is equivalent to the following system:
\begin{eqnarray}\label{eq6.29}
x^u&=&\Upsilon^u(x^{cs},y,t_0,\ep) \ , \
x^{cs}=\Psi^{cs}(x^u,t_0,\ep) \ , \ \ep
y=\Psi^y(x^u,t_0,\ep),\\\label{eq6.30}
d&=&\Upsilon^d(x^{cs},y,t_0,\ep)=\Psi^d(x^u,t_0,\ep).
\end{eqnarray}
From \eqref{eq6.14}, \eqref{eq6.28}, and a contraction mapping argument, one can easily prove

\begin{lemma}\label{lemma6.4}
There exists $\ep_0$ such that for every $\ep\in[0,\ep_0)$ and
$t_0\in\mathbb{R}$, there exist
$x^{cs}=x^{cs}(t_0,\ep),x^u=x^u(t_0,\ep),y=y(t_0,\ep)$, which are
continuous in $t_0$ and $\ep$, satisfying \eqref{eq6.29}. Moreover,
\begin{equation}\label{eq6.31}
|x^{cs}|_{X_1}+|x^u|_{X_1}+\ep|y|_{Y_1}\leq C'\ep,
\end{equation}
where $C'$ depends on constants in assumptions and uniform in $t_0$
and $\ep$.
\end{lemma}

Let
\begin{eqnarray*}
&&P^u(t_0,\ep)=(\Psi^d(x^u(t_0,\ep),t_0,\ep),x^u(t_0,\ep),x^{cs}(t_0,\ep),\ep y(t_0,\ep))+x_0,\\
&&P^{cs}(t_0,\ep)=(\Upsilon^d(x^{cs}(t_0,\ep),y(t_0,\ep),t_0,\ep),x^u(t_0,\ep),x^{cs}(t_0,\ep),\ep y(t_0,\ep))+x_0\\
&&(x_{-}(t),y_{-}(t))\triangleq\Phi(t,t_0,P^u(t_0,\ep),\ep) \qquad
(x_{+}(t),y_{+}(t))\triangleq\Phi(t,t_0,P^{cs}(t_0,\ep),\ep).
\end{eqnarray*}
From the coordinate system we constructed in the previous
subsection, clearly, the center-stable and unstable manifolds intersect if
\begin{equation} \label{E:H}
P^u=P^{cs}\Longleftrightarrow
\Psi^d(x_{-},t_0,\ep)=\Upsilon^d(x_{+},y_{+},t_0,\ep)
\Longleftrightarrow H(P^u)=H(P^{cs}),
\end{equation}
where \eqref{eq6.11} is used.

\noindent {\bf Melnikov method.} From \eqref{eq6.1}, we have
\begin{align}
H(P^u)=H(P^u)-H(0) =&\int_{-\infty}^{t_0}DH\big(\Phi(t,t_0,P^u,\ep)\big)
\p_t {\Phi}(t,t_0,P^u, \ep) dt \notag\\
=&\int_{-\infty}^{t_0}DH\big(x_{-}(t)\big)\big(f (x_{-}(t),y_{-}(t),t,\ep)
-f_0(x_{-}(t))\big)dt \label{eq6.32}
\end{align}
where the last equality follows from the fact that $H$ is invariant under \eqref{eq3.2} which implies, for any
$x\in X_1$, $DH(x)(Ax+f_0(x))=0$.

We claim for all $t\leq t_0$ and $\max\{a_1',0\}<\eta'<a_2'$,
\begin{equation}\label{eq6.33}
\big|x_{-}(t)-x_h(t-t_0)\big|_{X_1}+\big|y_{-}(t)\big|_{Y_1}\leq C'\ep e^{\eta' (t-t_0)}.
\end{equation}
In fact, from Lemma \ref{lemma6.2}--\ref{lemma6.4} and Theorem \ref{thm3.1}, this inequality is obvious
for $t \in [t_2+t_0, t_0]$. For $t\le t_0 + t_2$, since $(x_{-}, y_{-})$ and $x_h$ remain in a small
neighborhood of the perturbed and unperturbed unstable manifold, respectively. From \eqref{eq4.61} (and
Remark \ref{R:other}) and the standard Lipschitz dependence of the unstable orbits on the initial base
points in terms of the exponentially weighted norm (obtained from the uniform contraction mapping principle),
\eqref{eq6.33} follows.
\[
H(P^u) = \int_{-\infty}^{t_0} DH\big(x_h(t-t_0)\big)\big(D_yf y_{-}(t) +\ep\partial_{\ep}f\big) dt+O(\ep^2)
\]
where $D_y f$ and $\p_\ep f$ are both evaluated at $(x_h(t-t_0),0,t,0)$ and the $C^2$ smoothness of $H$ and
$DH(0)=0$ are used to guarantee the convergence of the above integral. To estimate the $y_-(t)$ term, we write
the variation of parameter formula,
\[
y_{-}(t)=e^{(t-t_0)\frac{J}{\ep}}y(t_0)+\int_{t_0}^te^{(t-\tau)\frac{J}{\ep}}
g(x_{-},y_{-},\tau,\ep)d\tau.
\]
On the one hand, integrating by parts and using \eqref{eq6.33}, we obtain
\[
\Big|\int_{-\infty}^{t_0} DH\big(x_h(t-t_0)\big)D_yf(x_h(t-t_0),0,t,0) e^{(t-t_0)\frac{J}{\ep}}
y(t_0,\ep)dt\Big|=O(\ep^2).
\]
On the other hand, we compute by changing the integration order, integrating by parts, and using the
exponential bounds of the orbits in $t$,
\begin{eqnarray*}
&&\int_{-\infty}^{t_0}D
H(x_h(t-t_0))D_yf(x_h(t-t_0),0,t,0)
\Big(\int_{t_0}^te^{(t-\tau)\frac{J}{\ep}}g(x_{-}(\tau),y_{-}(\tau),\tau,\ep)d\tau\Big)dt\\
&=&-\int_{-\infty}^{t_0}\Big(\int_{-\infty}^{\tau}D
H(x_h)D_yf(x_h,0,t,0)e^{(t-\tau)\frac{J}{\ep}}dt\Big)g(x_{-},y_{-},\tau,\ep)d\tau\\
&=&-\ep \int_{-\infty}^{t_0} DH(x_h)D_yf(x_h,0,\tau,0) J^{-1} g(x_{-},y_{-},\tau,\ep)d\tau\\
&&+\ep \int_{-\infty}^{t_0}\Big(\int_{-\infty}^{\tau}\frac{d}{dt}\big(DH(x_h) D_yf(x_h,0,t,0)\big)
J^{-1}e^{(t-\tau)\frac{J}{\ep}}dt\Big)g(x_{-},y_{-},\tau,\ep)d\tau
\end{eqnarray*}
It's easy to see from \eqref{eq6.33}
\begin{eqnarray*}
&& \int_{-\infty}^{t_0} DH(x_h)D_yf(x_h,0,t,0) J^{-1} g(x_{-},y_{-},\tau,\ep)d\tau\\
&=&-\int_{-\infty}^{t_0} DH(x_h(t-t_0))D_yf(x_h(t-t_0),0,t,0)J^{-1}g_0(x_h(t-t_0)))dt+O(\ep^2).
\end{eqnarray*}
Again, integrating by parts on $e^{(t-\tau) \frac J\ep}$ and using the assumption $Ax_h \in X_1$, one may
compute
\[\begin{aligned}
&\int_{-\infty}^{t_0}\Big(\int_{-\infty}^{\tau}\frac{d}{dt}\big(D
H(x_h)D_yf(x_h,0,t,0)\big) J^{-1}e^{(t-\tau)\frac{J}{\ep}}dt\Big)g(x_{-},y_{-},\tau,\ep)d\tau
=O(\ep).
\end{aligned}\]
Summarizing all the estimates, we obtain
\begin{equation}\label{eq6.34}
H(P^u)= \ep\int_{-\infty}^0 \omega(t, t_0) dt +O(\ep^2),
\end{equation}
where
\begin{equation} \label{E:omega}
\omega(t,t_0)= D H(x_h(t)) \big(\p_{\ep}f  -D_yf J^{-1}g \big) (x_h(t),0,t+t_0,0)\Big).
\end{equation}

Even though $\Phi(t,t_0,P^{cs}(t_0,\ep),\ep)$ does not necessarily stay in a small neighborhood of the origin
for all $t>>0$, we will still obtain a similar approximate for $H(P^{cs})$. In fact, by the same argument
leading to \eqref{eq6.33}, one can show that an inequality similar to \eqref{eq6.33} holds for $t \in [t_0,
T]$ as long as $(x_+(t), y_+(t))$ stay in the $r$ neighborhood of $0$ for all $t \in [t_0+t_1, T]$, where $r$
is the cut-off radius in the construction of the center-stable manifold. Therefore, let $a=\frac{1}{a_1-\eta'}
<0$ and $T_1=a\log\ep>0$, we have, for any $t\in[t_0,T_1+t_0]$,
\begin{equation}\label{eq6.35}
\Big|(x_+(t), y_+(t))-x_h(t-t_0)\Big|_{X_1 \times Y_1}\leq C' e^{\eta' (t-t_0)}\ep,
\end{equation}
where $C'$ is independent of $t$ and $\ep$. In particular,
\begin{equation}\label{eq6.36}
\begin{aligned}
&|(x_+, y_+)(T_1+t_0) -x_0(T_1)|_{X_1\times Y_1}\leq C'\ep^{1+a\eta'},
\qquad |x_h(T_1)|_{X_1}\leq C'\ep^{aa_1}
\end{aligned} \end{equation}
and
\begin{equation}\label{eq6.37}
aa_1=1+a\eta'>\frac{1}{2} \ , \ 1+2a\eta'>0.
\end{equation}
Since $H(0)=0$, we can compute
\begin{eqnarray*}
H(P^{cs})-\ep\int_{+\infty}^{0} \omega(t,t_0)dt = &&H(x_+(T_1+t_0)) -\ep \int_{+\infty}^{T_1}
\omega(t,t_0)dt\\
&& + H(P^{cs})-H(x_+(T_1+t_0)) -\ep\int_{T_1}^{0}\omega(t,t_0)dt.
\end{eqnarray*}
Using \eqref{eq6.36} and a similar procedure as in the approximation
of $H(P^u)$, we obtain
\begin{eqnarray*}
&&\Big|H(P^{cs})-H(x_+(T_1+t_0))-
\ep\int_{T_1+t_0}^0\omega(t,t_0)dt\Big|\leq C'\ep^{2+2a\eta'},\\
&&\Big|\ep\int_{+\infty}^{T_1+t_0}\omega(t,t_0)dt\Big|\leq C'\ep^{1+aa_1},\\
&&\Big|H(x_+(T_1+t_0))\Big|\leq|D^2H|_{C^0}|x_+(T_1+t_0)|^2\leq C'\ep^{2+2a\eta'}.
\end{eqnarray*}
Therefore
\begin{equation}\label{eq6.38}
H(P^u)-H(P^{cs}(\ep))=\ep M(t_0)+ O(\ep^{2+2a\eta'}), \qquad  M(t_0)=\int_{-\infty}^{+\infty}\omega (t, t_0)
dt.
\end{equation}
By Lemma \ref{lemma6.4}, \eqref{E:H}, and an Intermediate Value Theorem argument, we obtain

\begin{lemma}\label{lemma6.5}
Suppose $M(t_0)$ has a simple zero at some $t_0$, then there exists $\ep_0$ such that for each
$\ep\in[0,\ep_0)$, there exists $t^*$ such that $\CM_\ep^{cs}(t^*)$ and $\CM_\ep^u(t^*)$ intersects near
$x_0=x_h(0)$.
\end{lemma}

\begin{rem} \label{R:Melnikov}
Compare the Melnikov functional obtained in the above with the one
under regular perturbations, we observe an extra term $-D_yf J^{-1}
g_0$. In fact, it is easy to see where this term comes from the
coordinate change as in \eqref{E:y_1}. Let
\[
x_1 = x, \qquad y_1 = y + \ep J^{-1} g_0 (x).
\]
Then the $y$ equation takes the form of \eqref{E:doty_1} where $g_1 = O(\ep)$. From Proposition \ref{P:appro}
and \ref{thm4.20} and their remarks, it is easy to prove that the contribution from the $y$ equation to the
invariant manifolds are of order $O(\ep^2)$ and it does not appear in the leading Melnikov functional. The
$x$ equation now takes the form of
\[
\dot x_1 = A x_1 + f(x_1, y_1 - \ep J^{-1} g_0 (x_1), t, \ep).
\]
Ignoring $y_1$, the Melnikov functional of this regularly perturbed equation is exactly the one obtained in
Lemma \ref{lemma6.5}. The only reason we did not take this approach is that thus transformation reduces the
smoothness of the system by 1 order which would require $k\ge 3$ in the assumptions.
\end{rem}

\noindent {\bf Homoclinic solution.} Lemma \ref{lemma6.5} gives a condition for nonempty intersection of
center-stable and unstable manifold. This intersection means the
existence of a solution which converges to the steady solution as
$t\longrightarrow-\infty$. As $t$ increases and $t\leq
a\log{\ep}+t_0$, based on the stable foliation in the center-stable
manifold, this solution will approach a neighborhood of the steady
state inside the center manifold. In order to find conditions for this solution to converges to the steady
state as $t \to \infty$, in this subsection, we focus on the case when the unperturbed center manifold is at
least neutral and the perturbation is weakly dissipative so that the perturbed center manifold is weakly stable.
In this case, the size of the basin of attraction of $0$ on the center manifold is the key issue. We
{\it illustrate} how the method in the regular perturbation cases can be adapted here under certain assumptions,
which are not optimal as we are only giving an illustration. In order to specify the assumptions, we first look
at the Taylor's expansions of $f$ and $g$,
\[\begin{aligned}
f(x,y,t,\ep)&=f(x,y,t,0)+\ep f_2(x,y,t,\ep) =f_xx+f_yy+f_1(x,y)+\ep f_2(x,y,t,\ep),\\
g(x,y,t,\ep)&=g(x,y,t,0)+\ep g_2(x,y,t,\ep) =g_xx+g_yy+g_1(x,y)+\ep g_2(x,y,t,\ep),
\end{aligned}\]
where $f_{x,y}=D_{x,y}f(0,0, t, 0)$ and $g_{x,y}=D_{x,y}g(0,0,t,0)$ which are independent of $t$.
Let $P_{c,su}$ be linear projections from $X_1$ onto $X_1^{c,su}$
which are invariant under $e^{t(A+f_x)}$, where $X_1=X_1^c\oplus
X_1^{su}$. For any $x\in X_1$, we denote $x_c=P_cX_1$ and
$x_{su}=P_{su}x$. In addition, we assume
\begin{enumerate}
\item[(E1)] $\dim{X_1^s}<+\infty$ and $(f,g)$ are $C^3$ in $(x,y)$ with upper bound uniform in $t$.
\item[(E2)] Let $A(\ep)=A+f_x-\ep$ and $\frac{J(\ep)}{\ep}=\frac{J}{\ep}+g_y-\ep$. We further assume
$a_1'\leq0$ in (C2), which implies
\begin{eqnarray*}
|e^{tA(\ep)}|_{L(X_1^c)} + |e^{t\frac{J(\ep)}{\ep}}y|_{L(Y_1)} &\leq& Ke^{-\ep
t}\quad\mbox{for}\quad t\geq0,\quad x_c\in X_1^c.
\end{eqnarray*}
\item[(E3)] For $(x_c,x_{su},y,t,\ep)\in X_1^c\times X_1^{su}\times
Y_1\times\mathbb{R}\times[0,\ep_0)$,
\[\begin{aligned}
&P_cf_y=0\ , \ D_{(x_c,y)}^2P_cf_1(0,0,0)=0,\\
&P_cf_2(x_c,x_{su},y,t,\ep)=-x_c+\ep
B_0(\ep)(x_c,x_{su},y)+B_1(x_c,x_{su},y,t,\ep),\\
&B_0(\ep)\ \mbox{is a bounded linear operator acting on}\
(x_c,x_{su},y),\\
&B_1(0,0,0,t,\ep)=0\ , \ DB_1(0,0,0,t,\ep)=0.
\end{aligned}\]
\item[(E4)] For $(x_c,x_{su},y,t,\ep)\in X_1^c\times X_1^{su}\times
Y_1\times\mathbb{R}\times[0,\ep_0)$,
\[\begin{aligned}
&g_x=0\ , \ g_1(0,0,0)=0\ ,\ D_{(x_c,y)}^2g_1(0,0,0)=0,\\
&g_2(x_c,x_{su},y,t,\ep)=-y+\ep
B_2(\ep)(x_c,x_{su},y)+B_3(x_c,x_{su},y,t,\ep),\\
&B_2(\ep)\ \mbox{is a bounded linear operator acting on}\
(x_c,x_{su},y),\\
&B_3(0,0,0,t,\ep)=0\ , \ DB_3(0,0,0,t,\ep)=0.
\end{aligned}\]
\end{enumerate}

\begin{rem}\label{rem6.6}
It looks that the above assumptions are too restrictive. However, one should first try
to `diagonalize' the linear part to remove $f_y$ and $g_x$. (As a
separate topic, we will discuss this transformation in the Appendix.)
With this `diagonalized' linear part, one is in a position to carry
out a normal form transformation to eliminate some quadratic terms.
Assumption (E) should be considered for the form after performing a
normal form transformation.
\end{rem}

For sufficiently small $r$, from Theorem \ref{thm4.7}, for $\ep \in [0, \ep_0)$ and $t_0 \in \R$, there
exists a local center manifold $\mathcal{M}_{\ep}^c(t_0)$ which contains as an open subset the graph of
\[
h_{su} = (\Psi_u,\Psi_s) : B_{r}(X_1^c)\times B_{r}(Y_1)\times [0,\ep_0) \to X_1^s\times X_1^u
\]
where $h_{su} (\cdot, \cdot, 0)$ is understood as independent of $y$. Moreover, from Section \ref{S:InMa},
$h_{su}$ is uniformly bounded in $C^2$ in $(x_c,y)$ and Proposition \ref{P:appro} and Remark \ref{R:appro}
imply
\begin{equation}\label{eq6.39}\begin{aligned}
&h_{su}(0,0,\ep)=0 \ , \ |Dh_{su}(x_c,y,\ep)|\leq C'(\ep+|x_c|+|y|)\\
&|h_{su}(x_c,y,\ep)| \leq C'(\ep+|x_c|+|y|)(|x_c+|y|) \leq r,\\
&\big\{x_c+y+h_{su}(x_c,y,\ep)\big\}=\mathcal{M}_{\ep}^c(t_0) \cap\big(B_{r}(X_1^s)\times B_{r}(X_1^u)
\times B_r(X_1^c)\times B_{r}(Y_1)\big),
\end{aligned}\end{equation}
where $C'$ depends on those constants in assumptions. Here the assumptions
that $X^u$ and $X^s$ are finite dimensional are used. On the center manifold,
the flow is reduced to the $x_c$ and $y$ direction only, where the
solutions are given in the form of
\begin{equation}\label{eq6.40} \begin{aligned}
x_c(t)=&e^{(t-t_{\star})A(\ep)}x_c(t_{\star})+\int_{t_{\star}}^{t}e^{(t-\tau)A(\ep)}
\wt{f}(x_c,y,\tau,\ep)d\tau,\\
y(t)=&e^{(t-t_{\star})\frac{J(\ep)}{\ep}}y(t_{\star})+\int_{t_{\star}}^{t}e^{(t-\tau)\frac{J(\ep)}{\ep}}
\wt{g}(x_c,y,\tau,\ep)d\tau,
\end{aligned} \end{equation}
where
\[\begin{aligned}
&\wt{f}(x_c,y,t,\ep)=P_c(f_1+\ep f_2)(x_c+h_{su}(x_c,y,\ep),y,t,\ep),\\
&\wt{g}(x_c,y,t,\ep)=(g_1+\ep g_2)(x_c+h_{su}(x_c,y,\ep),y,t,\ep).
\end{aligned}\]
From assumptions (E1)---(E4) and \eqref{eq6.39}, $(\wt{f},\wt{g})$ satisfies
\[
\wt{f}(0,0,t,\ep)=0\ , \ \wt{g}(0,0,t,\ep)=0\ ,\ |D\wt{f}|+|D\wt{g}|\leq C'(\ep^2+|x_c|^2+|y|^2).\]
Suppose we have solution $(x_c(t),y(t))$ such that
$|x_c(t_{\star})|+|y(t_{\star})|\leq \delta\ep^{\frac{1}{2}}$ and
$|x_c(t)|+|y(t)|\leq(2+K)\delta\ep^{\frac{1}{2}}$ for
$t\in[t_{\star},T']$ for some $T'$. By using Gronwall's inequality,
we obtain from \eqref{eq6.40}
\[
e^{\ep (t-t_{\star})}(|x_c(t)|+|y(t)|)\leq K\delta\ep^{\frac{1}{2}}e^{C'\delta^2\ep(t-t_{\star})}.
\]
Since $C'$ is independent of $t$ and $\delta$, by taking $\delta< \frac 1{\sqrt{2C'}}$, we can extend $T'$
to $+\infty$ and derive
\[
|x_c(t)|+|y(t)| \leq K\delta\ep^{\frac{1}{2}}e^{-\frac 12 \ep(t-t_{\star})}.
\]
Therefore the basin of attraction on the center manifold, and thus on the center-stable manifold
as well, contains the graph of $h_{su}$ over the ball $B_{\delta \ep^{\frac 12}} (X_1^c \times Y_1)$.

Recall that $(x_{+},y_{+})$ denote the solution with initial time $t^*$ we obtained in Lemma \ref{lemma6.5}
which satisfies
\[
(x_{+}(t^*),y_{+}(t^*))\in\mathcal{M}_{\ep}^{cs}(t^*)\cap\mathcal{M}_{\ep}^u(t^*).
\]
In particular, for sufficiently small $\delta$ and by choosing $t_{\star}=T_1+t^*$, \eqref{eq6.36} implies for $t>T_1+t^*$
\[
|P_cx^{+}(t)|+|y^{+}(t)|\leq K\delta\ep^{\frac{1}{2}}e^{-\frac 12 \ep (t-T_1-t^*)} \to 0 \text{ as }
t\to +\infty.
\]

\begin{theorem}\label{thm6.7}
Assume (A1)---(A4), (A5$'$), (A7) for $k=2$, (B1)---(B5), (C1)---(C2), (D1)---(D5) and (E1)---(E4).
Suppose the Melnikov function $M(t_0)$ has simple zero points, then there exists $\ep_0>0$ such that for any
$\ep\in[0,\ep_0)$, \eqref{eq3.1} has homoclinic solutions to the origin.
\end{theorem}

\noindent {\bf Elastic Pendulum revisited.} Finally, we would like to revisit
\eqref{eq1.1}. We let $y=\ep u$ and $\dot{y}=u_1-\ep^2\gamma u$. We
rewrite \eqref{eq1.1} as a first order system.
\begin{equation}\label{eq6.41}
\left\{\begin{aligned}
&\dot{x}=\frac{x_1}{(1+\ep u)^2}\\
&\dot{x}_1=-g(1+\ep u)\sin{x}-2\ep\gamma x_1+\ep F_1(x,\ep u,t,\ep)\\
&\dot{u}=\frac{1}{\ep}u_1-\ep\gamma u\\
&\dot{u}_1=-\frac{1}{\ep}u-\ep\gamma u_1+\frac{x_1^2}{(1+\ep
u)^3}+\ep^3\gamma^2u+ g\cos{x}+\ep F_2(x,\ep u,t,\ep).
\end{aligned}\right.\end{equation}
We assume
\begin{equation} \tag{P1}
\gamma>0\ , \ F_1(\pi,\cdot,t,\ep)\equiv0\ ,\ \partial_tF_2(\pi,\cdot,t,\ep)\equiv0
\end{equation}
where the assumptions on $F$ are for simplicity. By implicit function theorem,
there exists a locally unique steady state $(\pi,0,u^{\ep},u_1^{\ep})$ which satisfies
$(u^{\ep},u_1^{\ep}) = (O(\ep), O(\ep^3))$ and
\begin{equation}\label{eq6.42}
u_1^{\ep}-\ep^2\gamma u^{\ep}=0\ , \ -u^{\ep}-\ep^2\gamma
u_1^{\ep}+\ep^4\gamma^2u^{\ep}-\ep g+\ep^2F_2(\pi,\ep u^{\ep},t,\ep)=0.
\end{equation}
Let $(\wt{x},v,v_1)=(x-\pi,u-u^{\ep},u-u_1^{\ep})$ to translate the steady state to $0$, \eqref{eq6.41} becomes
\begin{equation}\label{eq6.43}
\left\{\begin{aligned}
\dot{\wt{x}}=&x_1+(\frac 1{(1+\ep u^{\ep}+\ep v)^2}-1)x_1\\
\dot{x}_1=&g\sin{\wt{x}}+\ep g(v+u^{\ep})\sin{\wt{x}}-2\ep\gamma x_1+\ep F_1(\wt{x}+\pi,\ep u^{\ep}+\ep v,t,\ep)\\
\dot{v}=&\frac{1}{\ep}v_1-\ep\gamma v\\
\dot{v}_1=&-\frac{1}{\ep}v-\ep\gamma v_1+\frac{x_1^2}{(1+\ep
u^{\ep}+\ep v)^3}+\ep^3\gamma^2v+g(1-\cos{\wt{x}})\\
&+\ep (F_2(\wt{x}+\pi,\ep u^{\ep}+\ep v,t,\ep)-F_2(\pi,\ep
u^{\ep},t,\ep)).
\end{aligned}\right.\end{equation}
We rewrite the right hand side of last equation in \eqref{eq6.43} as
\[
-\frac{1}{\ep}v-\ep\gamma v_1+\ep D_xF_2(\pi,\ep u^{\ep},t,\ep)\wt{x}+\overline{g}(\wt{x},x_1,v,v_1,t,\ep),
\]
where in view of \eqref{eq6.42}
\[\begin{aligned}
\overline{g}(\wt{x},x_1,v,v_1,t,\ep)
\triangleq &\frac{x_1^2}{(1+\ep u^{\ep}+\ep
v)^3}+g(1-\cos{\wt{x}})+\ep^3\gamma^2v+\ep \big(F_2(\wt{x}+\pi,\ep
u^{\ep}\\
&+\ep v,t,\ep)-F_2(\pi,\ep u^{\ep},t,\ep)-D_xF_2(\pi,\ep
u^{\ep},t,\ep)\wt{x}\big).
\end{aligned}\]
Clearly, \eqref{eq6.43} does not satisfy assumptions (E1)--(E4) due to the presence of the linear term
$\ep D_xF_2(\pi,\ep u^{\ep},t,\ep)\wt{x}$. We will eliminate this term by a linear coordinate
transformation. A more general procedure of this type of transformation can be found in Appendix \ref{S:dia}.
To simply our notations, let
\[
M_1\triangleq\begin{pmatrix}0 & 1 \\ g & 0\end{pmatrix}\ ,\
M_2\triangleq\begin{pmatrix} -\ep\gamma & \frac{1}{\ep}\\
-\frac{1}{\ep} & -\ep\gamma
\end{pmatrix}= \frac J\ep -\ep \gamma\ , \ M_3\triangleq\begin{pmatrix}0 & 0\\
D_xF_2(\pi,\ep u^{\ep},t,\ep) & 0\end{pmatrix}.
\]
By implicit function theorem there exist $L_1\in L(\mathbb{R}^2,\mathbb{R}^2)$
with $|L_1|\leq C'\ep^2$, where $C'$ depends on constants in
assumptions, such that
\[
JL_1-\ep^2\gamma L_1-\ep L_1M_1+\ep^2M_3=0,
\]
which implies
\[\begin{pmatrix}
I & 0\\
L_1 & I
\end{pmatrix}^{-1}\begin{pmatrix}
M_1 & 0\\
\ep M_3 & M_2
\end{pmatrix}\begin{pmatrix}
I & 0\\
L_1 & I
\end{pmatrix}=\begin{pmatrix}  M_1 & 0\\
0 & M_2 \end{pmatrix}.
\]

When $\ep=0$, $0$ is a hyperbolic fixed point of the first two
equations of \eqref{eq6.43}, thus, all assumptions in (E3) for $f$
are automatically satisfied. It remains only straight forward verifications that (E1)--(E4) are satisfied.
Therefore, one can apply Theorem \ref{thm6.7}.

\subsection{Persistence of homoclinic orbit under conservative
perturbation} \label{SS:cons}

In this subsection, in addition to those assumptions given at the beginning of the section, we further
assume
\begin{enumerate}
\item[(D7)] $X_1^s$ is finite dimensional. Moreover,
\[\begin{aligned}
&\partial_tf(x,y,t,\ep)\equiv0\ , \ \partial_tg(x,y,t,\ep)\equiv0\ , \ \dim
(T_{x_0}\mathcal{M}_0^{s}\bigcap T_{x_0}\mathcal{M}_0^{cu})=1.
\end{aligned}\]
\item[(D8)] There exists a family of invariant quantities $\{H(\cdot, \ep)$ for
\eqref{eq3.1} which, in terms of the Taylor expansion in $u=\frac{y}{\ep}$, takes the form
\[\begin{aligned}
&H (x,\ep u, \ep) =H_0(x,\ep)+H_1(x,\ep)u+H_2(x,\ep)(u,u)+H_3(x,u,\ep), \ ,\ H_0(x,0)=H(x)\\
&H_i\in C^{3-i}(X_1\times\mathbb{R},L_i(Y_1,\mathbb{R}))\ \mbox{for}\ i=0,1,2, \quad H_3\in C^3(X_1\times
Y_1\times\mathbb{R},\mathbb{R}).
\end{aligned}\]
Here with a slight abuse of the notation, we still denote the unperturbed invariant functional by $H(x)$.
Moreover, we assume there exist $c_0, c_2>0,c_1\geq0$, such that for any $\xi_c\in
X_1^c$ and $(x,u)\in B_r(X_1)\times B_b(Y_1)$
\[\begin{aligned}
&H_0(0,\ep)=0 \ , \ DH_{0}(0,\ep)=0\ ,\ D^2H_0(0,0)(\xi_c,\xi_c)\geq c_0|\xi_c|^2,\\
&H_1(0,\ep)=0 \ , \ |DH_1(0,\ep)|\leq c_1\ , \ |H_2(0,\ep)(u,u)|\geq c_2|u|^2,\\
&H_3(x,0,\ep)=0\ , \ D_uH_3(x,0,\ep)=0\ , \ |D^2H_3(x,u,\ep)|\leq
C_0\ep \ , \ \overline{a}\triangleq c_0c_2-c_1^2>0.
\end{aligned}\]
\end{enumerate}
Under these assumptions, one may still compute the Melnikov functional, but mostly it turns out to be
identically zero. Our goal is find the intersection of the center-stable and the center-unstable manifolds.

We first refine the coordinates on the cross section $\Sigma$ defined in Subsection \ref{SS:coord}. Let
\[
\wt{\CM}_{\ep,0}^{cs, cu, s, u} = \CM_{\ep, 0}^{cs, cu, s, u} \cap (x_0 + \Sigma) \quad
\overline{X}_1^{cs, cu, s, u} = T_{x_0} \CM_0^{cs, cu, s, u} \cap \Sigma \quad \overline{X}_1^c=\overline{X}_1^{cs}\cap\overline{X}_1^{cu}.
\]
Much as in Subsection \ref{SS:coord}, we have
\[
\Pi\triangleq \Sigma \cap \big(\ker(DH(x_0) \oplus Y_1\big) = Y_1 \oplus \overline{X}_1^s \oplus
\overline{X}_1^u \oplus \overline{X}_1^c.
\]
Recall we took $\omega \in \Sigma \backslash \Pi$ with $DH(x_0) \omega=1$. Let $Q_{\omega, y, s, u, c}$ be
the projections on $\Sigma$ given by $\Sigma = \R\omega \oplus \Pi$ and the above decomposition. For any $p\in\Sigma+x_0$, its coordinates can be written as
\begin{equation*}
\begin{aligned}
(d,x^s, x^c, y,x^u)=&(Q_{\omega},Q_{s},Q_{c},Q_y,Q_u)(p-x_0) =(DH(x_0),Q_{s},Q_{c},Q_y,Q_u)(p-x_0).
\end{aligned}
\end{equation*}
Similar to that the center-stable and unstable manifolds in $\Sigma$ can be written
as graphs of $\Upsilon$ and $\Psi$ as given in Lemma \ref{lemma6.3}, for any $b>0$, there exist $r>0$ and
$\Upsilon_1 (\cdot, \ep): B_r (\overline{X}_1^{cu}) \times B_b(Y_1) \to X_1^s \times \R$ and $\Psi_1(\cdot, \ep):
B_r (\overline{X}_1^s) \to X_1^{cu} \times Y_1 \times \R$ such that
\[\begin{aligned}
&\Big\{(\Upsilon_1^d+\Upsilon_1^s)(x^c,x^u,y,\ep)+x^c+x^u+\ep
y\Big\}\subset \wt{\mathcal{M}}_{\ep}^{cu},\\
&\Big\{(\Psi_1^d+\Psi_1^y+\Psi_1^{c}+\Psi_1^{u})(x^s,\ep)+x^s\Big\}\subset \wt{\mathcal{M}}_{\ep}^{s},
\end{aligned}\]
and $\Upsilon_1,\Psi_1$ satisfy similar properties as $\Upsilon,\Psi$ in \eqref{eq6.14}, and
\eqref{eq6.28}.

To find the intersection of $\wt{\mathcal{M}}_{\ep}^{cs}$ and $\wt{\mathcal{M}}_{\ep}^{cu}$, we
first try to match all coordinates except the $d$ direction. Given $(x^c, y) \in B_r (\overline{X}_1^c) \times
B_{b}(Y_1)$, by using the Contraction Mapping Theorem and Lemma \ref{lemma6.3}, we obtain a unique pair
$x^{s,u}(x^c,y,\ep)$ such that
\begin{equation}\label{eq6.44}
\begin{aligned}
&\Upsilon^u(x^c,x^s(x^c,y,\ep),y,\ep) = x^u(x^c,y,\ep) \quad   x^s(x^c,y,\ep)= \Upsilon_1^s
(x^c,x^u(x^c,y,\ep),y,\ep)
\end{aligned}\end{equation}
and they satisfy, for some $C'$ independent of $\ep$.
\[\begin{split}
&x^{s,u}(0,\cdot,0)\equiv 0\ , \ D_{x^c} x^{s,u}(0,0,0)=0\ , \ |x^{s,u} (\cdot, \cdot, \ep) - x^{s,u}
(\cdot, \cdot, 0)|_{C^1}\leq C'\ep.
\end{split}\]

Among the above points on $\wt{\mathcal{M}}_{\ep}^{cs}$, next we identify the one on $\wt{\mathcal{M}}_{\ep}^s$.
More precisely, substituting $x^s (x^c, y, \ep)$ into $\Upsilon,\Psi_1$ and using Contraction Mapping Theorem,
Lemma \ref{lemma6.3}, and an inequality for $\Psi_1$ similar to \eqref{eq6.28}, we obtain a unique pair
$(x^c(\ep),y(\ep))$ such that
\[
(x^c(\ep), y(\ep)) = (\Psi_1^c, \frac 1\ep \Psi_1^y) \big(x^s (x^c(\ep),y(\ep),\ep), \ep\big), \qquad
|x^c(\ep)|_{X^1} + |y(\ep)|_{Y_1} \leq C' \ep
\]
which implies
\[\begin{aligned}
&\Upsilon \big(x^c(\ep),x^s(x^c(\ep),y(\ep),\ep),y(\ep),\ep\big)
+x^c(\ep)+x^s(x^c(\ep),y(\ep),\ep)+\ep y(\ep)\\
=&\Psi_1 \big(x^s(x^c(\ep),y(\ep),\ep),\ep)+x^s(x^c(\ep),y(\ep),\ep\big)\in\wt{\mathcal{M}}_{\ep}^s.
\end{aligned}\]
Similarly, there exist $(x^c_1(\ep),y_1(\ep)) = O(\ep)$ satisfying
\[\begin{aligned}
&\Upsilon_1 \big(x_1^c(\ep),x^u(x_1^c(\ep),y_1(\ep),\ep),y_1(\ep),\ep\big)
+x_1^c(\ep)+x^u(x_1^c(\ep),y_1(\ep),\ep)+\ep y_1(\ep)\\
=&\Psi \big(x^u(x_1^c(\ep),y_1(\ep),\ep),\ep\big)+x^u(x_1^c(\ep),y_1(\ep),\ep)\in\wt{\mathcal{M}}_{\ep}^u.
\end{aligned}\]
Let, for $\tau\in[0,1]$
\[\begin{split}
&q(\tau)= (q_c(\tau),q_y(\tau))\triangleq (1-\tau) \big(x^c(\ep), y(\ep) \big) +\tau \big(x_1^c(\ep),
y_1(\ep)\big)\\
& p_\ep^u(\tau) =  \Upsilon_1(q(\tau),x^u(q(\tau),\ep),\ep)+q_c(\tau)+x^u(q(\tau),\ep)+\ep q_y(\tau)
\triangleq x_\ep^u (\tau) +\ep q_y(\tau) \in \CM_\ep^{cu} \\
& p_\ep^s (\tau)= \Upsilon(q(\tau),x^s(q(\tau),\ep),\ep)+q_c(\tau)+x^s(q(\tau),\ep)+\ep q_y(\tau)
\triangleq x_\ep^s (\tau) +\ep q_y(\tau)\in \CM_\ep^{cs}.
\end{split}\]
We will show there exists $\tau_0 \in [0, 1]$ such that $p_\ep^u(\tau_0) = p_\ep^s(\tau_0)$, which is
equivalent to $H(x_\ep^u(\tau_0)) = H(x_\ep^s(\tau_0))$ as we have matched all other coordinates in
\eqref{eq6.44}. Since $DH(x_0)\big|_{\Sigma}\neq0$ and $|y(\ep)_{Y_1}, |y_1(\ep)|_{Y_1}=O(\ep)$, by
assumption (D8), it is clear that $H(x_\ep^u(\tau_0)) = H(x_\ep^s(\tau_0))$, and thus $p_\ep^u(\tau_0)
= p_\ep^s(\tau_0)$, if and only if $H(p_\ep^u(\tau_0), \ep) = H(p_\ep^s(\tau_0), \ep)$. Let
\[
h(\tau)= H (p_\ep^s(\tau), \ep) - H(p_\ep^u(\tau), \ep).
\]

To analyze $h(\tau)$, note (D8) implies $H_{\ep}(0,0)=0 \ , \ DH_{\ep}(0,0)=0$. For any
\[
z=\xi_c+\ep u+h_{su}(\xi_c, \ep u,\ep) \triangleq x+\ep u \in \big(B_r(X_1)\times B_{\ep b}(Y_1)\big) \cap
\CM_\ep^c,
\]
by using \eqref{eq6.39}, we have
\[
|h_{su}|_{X_1} \le C' \big (|\xi_c|_{X_1}( \ep + |\xi_c|_{X_1}) + \ep^2 |u|_{Y_1} (1+ |u|_{Y_1})\big)
\]
and thus
\[\begin{aligned}
&H_0(x,\ep)\geq \big(\frac{c_0}2 - C'(\ep + |\xi_c|_{X_1}) \big) |\xi_c|_{X_1}^2-C'\ep^3|u|_{Y_1}^2,
\end{aligned}\]
where $C'$ depends on the constants in the assumptions and $b$. Moreover,
\[\begin{aligned}
&|H_1(x,\ep)u|\leq c_1 \big(1+C'(\ep + |\xi_c|_{X_1}) \big) |\xi_c|_{X_1} |u|_{Y_1} + C'\ep^2 |u|_{Y_1}^2,\\
&H_2(x,\ep)(u,u)\geq \big(c_2-C'(\ep^2 |u|_{Y_1} + |\xi_c|_{X_1})\big)|u|_{Y_1}^2, \qquad
|H_3(x,u,\ep)|\leq C_0\ep|u|_{Y_1}^2.
\end{aligned}\]
By the last inequality in (D8) and choosing sufficiently small $r$, there exists $c_*>0$ such that
\[
H(z, \ep)\geq c_*(|\xi_c|^2+|u|_{Y_1}^2).
\]
It implies that $H(\cdot, \ep)>0$ in $\mathcal{M}_{\ep}^c\bigcap
\big(B_{r}(X_1) \times B_{b\ep}(Y_1)\big)$ except at $0$, with
quadratic lower bound (after the scaling $y=\ep u$). Actually, it also implies the origin is stable both
in forward and backward time on the center manifold. Consequently,
$\mathcal{M}_{\ep}^{\alpha}$ are unique, where $\alpha=c,cu,cs$. From the invariance of $H(\cdot, \ep)$,
$H(\cdot, \ep)>0$ in $\mathcal{M}_{\ep}^{cs, cu}\backslash \CM_\ep^{s,u}$.

From the definition of $h$, it is clear $h(1)\geq0\geq h(0)$. By the Intermediate Value Theorem, $h(\tau_0)
=0$ for some $\tau_0 \in [0, 1]$.

\begin{theorem}\label{thm6.8}
Assume (A1)---(A5), (A5$'$) for $k=2$, (B1)---(B5), (C1)---(C2) and
(D1)---(D8). There exists $\ep_0>0$ such that for any
$\ep\in[0,\ep_0)$, the center-stable manifold and center-unstable
manifold of \eqref{eq3.1} has nonempty intersection.
\end{theorem}

The intersection of the center-stable and center-unstable manifold is generically transversal and forms
a high dimensional tube homoclinic to the center manifold. See \cite{SZ03} for more
discussion in the regular perturbation case.

\noindent {\bf Elastic pendulum revisited.} Assume the elastic pendulum system \eqref{eq1.1} is
conservative, i.e. $\gamma=0$ and the perturbation $(\ep F_1, \ep F_2)$ comes from a small perturbation
$\ep G(x, y, \ep)$ to the potential energy. System \eqref{eq1.1} becomes
\begin{equation}\label{eq6.45}
\left\{\begin{aligned}
&\dot{x}=\frac{x_1}{(1+y)^2} \qquad
\dot{x}_1=-g(1+y)\sin{x}-\ep D_xG(x,y,\ep)\\
&\ep\dot{y}=y_1 \qquad
\ep\dot{y}_1=-y+\frac{\ep^2 x_1^2}{(1+y)^3}+\ep^2 g\cos{x}-\ep^3 D_yG(x,y,\ep).
\end{aligned}\right.\end{equation}
Its energy is given by the sum of the kinetic energy, gravitational, elastic, and perturbational energy
\[
H = \frac 12 (1+y)^2 x_1^2 + \frac {y_1^2}{2\ep^2} + \frac {y^2}{2\ep^2} - g (1+y) \cos x + \ep G(x, y, \ep).
\]
From the Implicit Function Theorem, for each $\ep$, there exist a unique fixed point $(x^{\ep},0,u^{\ep},0)$
with $(x^{\ep},y^{\ep})=(\pi + O(\ep), g\ep^2 + O(\ep^3))$ such that
\begin{equation}\label{eq6.46}
g(1+y^{\ep})\sin{x^{\ep}}+\ep D_xG(x^{\ep},y^{\ep},\ep)= y^{\ep}-\ep^2 g\cos{x}+\ep^3
D_yG(x^{\ep},y^{\ep},\ep)=0,
\end{equation}
Let $\wt{x}=x-x^{\ep},\wt{y}=y-y^{\ep}$, we can rewrite \eqref{eq6.45} as
\[
\left\{\begin{aligned}
\dot{\wt{x}}=&\frac{x_1}{(1+y^{\ep}+\wt{y})^2}\\
\dot{x}_1=&-g(1+y^{\ep}+\wt{y})\sin{(x^{\ep}+\wt{x})}-\ep D_xG(x^{\ep}+\wt{x},y^{\ep}+\wt{y},\ep)\\
\dot{\wt{y}}=&\frac{1}{\ep}y_1\\
\dot{y}_1=&-\frac{1}{\ep}\wt{y}-\frac{1}{\ep}y^{\ep}+\frac{\ep
x_1^2}{(1+y^{\ep}+\wt{y})^3}+\ep g\cos{(x^{\ep}+\wt{x})} -\ep^2 D_yG(x^{\ep}+\wt{x},y^{\ep}+\wt{y},\ep),
\end{aligned}\right.\]
whose invariant energy takes the form
\[\begin{aligned}
H(\wt{x},x_1,v,u_1, \ep) =&\frac{x_1^2}{2(1+ y^{\ep}+\ep
v)^2}+\frac{u_1^2}{2}+\frac{(\frac{y^{\ep}}{\ep}+v)^2}{2}
-\frac{(y^{\ep})^2}{2\ep^2} +\ep\Big(G(x^{\ep}+\wt{x},y^{\ep}+\ep
v,\ep)\\
&-G(x^{\ep},y^{\ep},\ep)\Big) -g\Big((1+y^{\ep}+\ep v)\cos{(x^{\ep}+\wt{x})} -(1+y^{\ep})\cos{x^{\ep}}\Big)
\end{aligned}\]
where $v=\frac{\wt{y}}{\ep}\ ,\ u_1=\frac{y_1}{\ep} = \dot y$. Its Taylor's expansion yields
\[\begin{aligned}
H_0=&\frac{x_1^2}{2(1+y^{\ep})^2}-g(1+y^{\ep})\big(\cos{(x^{\ep}+\wt{x})}-\cos{x^{\ep}}\big)
+\ep\big(G(x^{\ep}+\wt{x},y^{\ep},\ep)-G(x^{\ep},y^{\ep},\ep)\big)\\
H_1=&\big(-\frac{\ep x_1^2}{(1+\ep u^{\ep})^3} +\frac{y^{\ep}}{\ep}-\ep
g\cos{(x^{\ep}+\wt{x})}+\ep^2 D_yG(x^{\ep}+\wt{x},y^{\ep},\ep),\, 0\big),\\
H_2=&\begin{pmatrix}\frac{1}{2}+\frac{\ep^3}{2}D_y^2G(x^{\ep}+\wt{x},y^{\ep},\ep)
+\frac{3\ep^2 x_1^2}{2(1+\ep u^{\ep})^4} & 0\\
0 & \frac{1}{2}\end{pmatrix}.
\end{aligned}\]
One can use \eqref{eq6.46} to verify the above $H_i$, where
$i=0,1,2$, satisfy assumption (D8). Therefore, for $\ep\ll1$, the
center stable manifold and center-unstable manifold of
\eqref{eq6.45} intersect near the unperturbed homoclinic orbit
$x_h(t)$, which generically form a 2-parameter family of solutions
homoclinic to a small neighborhood of the fixed point on the center manifold which is foliated by periodic
orbits corresponding to small amplitude fast oscillations.

\section{Appendix} \label{S:dia}

In the appendix, we outline a procedure to block diagonalize the linearization of
\eqref{eq3.1} as a steady state via a linear transformation. And we will discuss two cases, namely, $A$
is a bounded linear operator on $X$ and $A$ is generator of a semigroup on $X$. This lays the foundation for
further normal form transformations to eliminate some nonlinear terms. Assume
\begin{enumerate}
\item[(B)] For $(t,\ep)\in\mathbb{R}\times[0,\ep_0)$,
\[
(f,g)(0,0,t,\ep)=0\ , \ \partial_t(Df,Dg)(0,0,t,\ep)=0.
\]
\end{enumerate}

The linearization of \eqref{eq3.1} at $0$ is given by the operator
\[
\CA_\ep = \begin{pmatrix} A+D_xf(0,\ep) &
D_yf(0,\ep)\\
D_xg(0,\ep) & \frac{J}{\ep}+D_yg(0,\ep)
\end{pmatrix}.
\]
We look for linear operators $L_1^\ep \in L(X_1, Y)$ and $L_2^\ep\in L(Y_1, X)$ such that their graphs are
invariant under $\CA_\ep$. For simplicity, we write $D_{x,y}(f,g)(0,\ep)$ as $D_{x,y}(f,g)$.
Therefore, $(L_1^{\ep},L_2^{\ep})$ should satisfy the following system:
\begin{equation}\label{eq7.1}
\begin{aligned}
&(J+\ep D_yg)L_1^{\ep}-\ep L_1^{\ep}(A+D_xf+D_yfL_1^{\ep})+\ep D_xg=0,\\
&L_2^{\ep}(\ep D_xgL_2^{\ep}+J+\ep D_yg)-\ep(A+D_xf)L_2^{\ep}-\ep
D_yf=0.
\end{aligned}
\end{equation}
When $A$ is bounded, treating the above as the equation for zero points of a mapping from $L(X,Y_1) \times
L(Y,X)\times\mathbb{R}$ to $L(X,Y)\times L(Y_1,X)$, simply from the Implicit Function Theorem, we obtain

\begin{lemma}\label{lemma7.1}
Assume $A\in L(X,X)$, (A2) and (A3) for $k=1$ and (B). There exists
$\ep_0>0$ such that, for any $\ep\in[0,\ep_0)$, \eqref{eq7.1} has a unique
pair solution of $O(\ep)$:
\[
(L_1^{\ep},L_2^{\ep})\in L(X,Y_1)\times L(Y,X).
\]
\end{lemma}

Let $\begin{pmatrix} \tilde x \\ \tilde y\end{pmatrix} = \begin{pmatrix}
I & L_2^{\ep}\\
L_1^{\ep} & I
\end{pmatrix}^{-1}$ as the new variables. System \eqref{eq3.1} becomes
\begin{equation}\label{eq7.2}
\left\{\begin{aligned}
&\dot{\tilde x}=(A+D_xf+D_yfL_1^{\ep}) \tilde x+\wt{F}(\tilde x, \tilde y,t,\ep)\\
&\dot{\tilde y}=(\frac{J}{\ep}+D_yg+D_xgL_2^{\ep}) \tilde y+\wt{G}(\tilde x,\tilde y,t,\ep).
\end{aligned} \right.
\end{equation}
Using the following embedding of spaces
\[
L(X,Y_1)\subset L(X,Y)\ , \ L(Y,Y_1)\subset L(Y,Y)\ , \ L(Y,Y_1)\subset L(Y_1,Y_1),
\]
it is easy to show $\wt{F}$ and $\wt{G}$ satisfy propertie (A3) and thus
all results in previous sections still hold for the new system \eqref{eq7.2}.

In the case when $A$ is unbounded, we use an integral equation appraoch
to solve \eqref{eq7.1} under the assumption
\begin{enumerate}
\item[(F)] There exist closed subspaces $X^{s,c,u}$ of $X$ such that $X=X^{s}\oplus
X^{c}\oplus X^{u}$ and $A+D_xf(0,\ep)$ is invariant on $X^{s,c,u}$.
Let $A^{u,c,s}=\big(A+D_xf(0,\ep)\big)\big|_{X^{s,c,u}}$. We further assume
$A^c$ is a bounded linear operator on $X^c$ and there exist
$\omega_{s}<0$ and $\omega_{u}>0$ such that
\[|e^{tA^{s}}|\leq
Ke^{\omega_{s}t}\ \ \mbox{for} \ \ t\geq0\ , \ |e^{tA^{u}}|\leq
Ke^{\omega_{u}t}\ \ \mbox{for} \ \ t\leq0.\]
\end{enumerate}
To find $L_1^{\ep}$, we consider
\begin{eqnarray}
\label{eq7.3}
G^{u}(L) x^u &=&\int_{+\infty}^0e^{tJ}(\ep LD_yfL^{u}-\ep D_xg-\ep D_ygL^{u})e^{-\ep tA^{u}}x^u dt, \; x^u
\in X^u\\
\label{eq7.4} G^c(L) x^c&=&\ep(J+\ep D_yg)^{-1}(LD_yfL^c-D_xg+L^cA^c) x^c, \; x^c \in X^c\\
\label{eq7.5} G^{s}(L) x^s&=&\int_{-\infty}^0e^{tJ}(\ep LD_yfL^{s}-\ep D_xg-\ep D_ygL^{s})e^{-\ep tA^{s}} x^s dt,
\; x^s \in X^s.
\end{eqnarray}
where $L^{u, c, s} = L P_{u, c, s}$ and $P_{u,c,s}$ are the projections defined by the decomposition.
Let $\overline{G}(L)=G^{u}(L)P_u + G^c(L)P_c + G^{s}(L)P_s$.

\begin{lemma}\label{lemma7.2}
If \eqref{eq7.6} below is satisfied, then $\overline{G}$
is a contraction from a bounded ball in $L(X,Y)$ to itself and its unique fixed point
$L$ satisfies the first equation of \eqref{eq7.1} and
\[
|L|_{L(X,Y)}\leq C'\ , \ |L|_{L(X_1,Y_1)}\leq C'\ep.
\]
\end{lemma}

\begin{proof}
In this proof, we use the equivalent norm on $X$ defined by
\[
\|x\|=\max\{|P_ux|, \, P_cx|, \, |P_sx|\},
\]
which induces equivalent norms $\|\cdot\|$ of operators. In particular,
\[
\|L\|= |L^u|_{L(X^u,Y)} +|L^c|_{L(X^c,Y)} +|L^s|_{L(X^s,Y)}.
\]
Let
\[
\mathcal{B}=\Big\{L\Big| \|L^u\| \leq1\Big\}, \qquad C_0= \max\big\{ \|A^c\|, \|D_yf(0,\ep)\|,
\|D_xg(0,\ep)\|, \|D_yg(0,\ep)\|\big\}.
\]
$\overline{G}$ satisfies the following estimates on $\mathcal{B}(\rho)$
\begin{eqnarray*}
&&\|\overline{G}(L)\|\leq 5C_0K\Big(\frac{1}{|\omega_s|}+\frac{1}{|\omega_u|}+
2\ep|J^{-1}|\Big),\\
&&\|\overline{G}(L_1)-\overline{G}(L_2)\| \leq 3KC_0\Big(\frac{1}{|\omega_s|}+\frac{1}{|\omega_u|}+
2\ep|J^{-1}|\Big) \|L_1-L_2\|.
\end{eqnarray*}
If $\omega_{u,s}$ satisfy
\begin{equation}\label{eq7.6}
\frac{1}{|\omega_s|}+\frac{1}{|\omega_u|}+ 2\ep|J^{-1}|<\frac{1}{5 KC_0},
\end{equation}
then clearly $\overline{G}$ defines a contraction mapping on $\mathcal{B}$.
Therefore, $\overline{G}$ has a fixed point $L$. Integrating
\eqref{eq7.3} by parts
\[\begin{aligned}
L^u x^u=&J^{-1}(\ep LD_yfL^{u}-\ep D_xg-\ep D_ygL^{u}) x^u\\
&+\int_{+\infty}^0J^{-1}e^{tJ}(\ep LD_yfL^{u}-\ep D_xg-\ep
D_ygL^{u})e^{-\ep tA^{u}}\ep A^u x^u dt
\end{aligned}\]
shows $L^u\in L(X_1^u,Y_1)$ and
\begin{equation} \label{E:Lu}
JL^u = \ep LD_yfL^{u}-\ep D_xg -\ep D_ygL^{u} + \ep L^u A^u
\end{equation}
which immediately implies
\[\begin{aligned}
|L|_{L(X_1^u,Y_1)} \leq  C'\ep.
\end{aligned}\]
Identity \eqref{E:Lu} and the similar one for $L^s$ along with the definition of $G^c$ implies $L$
satisfies \eqref{eq7.1}.
\end{proof}

To solve for $L_2^{\ep}$, we consider $L \in L(Y_1, X_1)$ with $L^{u,c,s}= P_{u,c,s} L$ and
\[
P^{u,c,s} \wt{G} (L)= \left\{\begin{aligned}
&\ep\int_{+\infty}^0e^{-\ep
tA^u}\big(D_yf-L^uD_xgL-L^uD_yg\big)e^{tJ}dt\\
&\ep\big(D_yf-L^cD_xgL+A^cL\big)(J+\ep D_yg)^{-1}\\
&\ep\int_{-\infty}^0e^{-\ep
tA^s}\big(D_yf-L^sD_xgL-L^sD_yg\big)e^{tJ}dt.
\end{aligned}\right.
\]

\begin{lemma}\label{lemma7.3}
If $\omega_{u,s}$ satisfy \eqref{eq7.6}, there exists a unique
$L\in L(Y_1,X_1)$ such that $\wt{G}(L)=L$. Moreover, $L\in L(Y,X)$
with estimates $|L|_{L(Y_1,X_1)}\leq C'$, $|L|_{L(Y,X)}\leq C'\ep$.
\end{lemma}

\begin{proof}
The proof of the lemma follows from the same procedure on the set
\[
\mathcal{B}_1=\Big\{L\Big||L^u|_{L(Y_1,X_1^u)}+|L^c|_{L(Y_1,X_1^c)}+|L^s|_{L(Y_1,X_1^s)}\leq1\Big\}.
\]
\end{proof}

\end{document}